\documentclass[11pt,reqno]{amsart}
\usepackage[square,numbers]{natbib}

\usepackage[colorinlistoftodos]{todonotes}
\RequirePackage[OT1]{fontenc}
\RequirePackage{amsthm,amsmath,bm,natbib,graphicx,enumitem,bbm,comment}
\RequirePackage[colorlinks,citecolor=blue,urlcolor=blue]{hyperref}
\RequirePackage{hypernat}
\usepackage{amssymb,tabularx,multicol,multirow,booktabs,csquotes}
\usepackage[top=1in, bottom=1in, left=1in, right=1in]{geometry}
\usepackage{fancyhdr,lipsum, makecell}
\usepackage{booktabs, makecell}
\usepackage[space]{grffile}
\usepackage{lineno}
\usepackage{url}
\usepackage{textcomp}
\usepackage{xpatch}
\RequirePackage{xr}

\def \ve {\varepsilon}
\newtheorem{theorem}{Theorem}
\newtheorem{corollary}{Corollary}
\newtheorem{lemma}{Lemma}
\newtheorem{prop}[theorem]{Proposition}

\DeclareMathOperator{\sech}{sech}

\newcommand{\bE}{\mathbb{E}}

\newcommand{\bP}{\mathbb{P}}
\newcommand{\bQ}{\mathbb{Q}}
\newcommand{\bR}{\mathbb{R}}

\newcommand{\bX}{\mathbb{X}}

\newcommand{\bZ}{\mathbb{Z}}

\def \bx{\mathbf{x}}

\def \bQ{\mathbf{Q}}
\def \sf {\mathsf}
\def \bX{\mathbf{X}}

\def \be{\begin{align*}}
\def \ee{\end{align*}}

\def \E{\mathbb{E}}
\def \P{\mathbb{P}}

\def \sumn{\sum_{i=1}^n}
\def \bmu{\pmb{\mu}}
\def \R{\mathbb{R}}

\def \C {\mathcal{C}}

\def \bzero {\mathbf{0}}

\newcommand\blfootnote[1]{%
	\begingroup
	\renewcommand\thefootnote{}\footnote{#1}%
	\addtocounter{footnote}{-1}%
	\endgroup
}

\def \bsigma{\boldsymbol{\sigma}}

\title[Detection Thresholds for Ising Models]{Sharp Signal Detection under Ferromagnetic Ising Models}

\author{Sohom Bhattacharya}
\address{Sohom Bhattacharya\newline Department of Statistics, Stanford University, CA, USA,
	\newline \tt{sohomb@stanford.edu}}

\author{Rajarshi Mukherjee}
\address{Rajarshi Mukherjee\newline Department of Biostatistics, Harvard University, MA, USA,
	\newline \tt{ram521@mail.harvard.edu}}

\author{Gourab Ray}
\address{Gourab Ray\newline Department of Mathematics and Statistics, University of Victoria, Vancouver, BC, Canada,
	\newline \tt{gourabray@uvic.ca}}

\begin{document}


\begin{abstract} 

	In this paper we study the effect of dependence on detecting a class of structured signals in Ferromagnetic Ising models. Natural examples of our class include Ising Models on lattices, and Mean-Field type Ising Models  such as dense Erd\H{o}s-R\'{e}nyi, and dense random regular graphs. Our results not only provide sharp constants of detection in each of these cases and thereby pinpoint the precise relationship of the detection problem with the underlying dependence, but also demonstrate how to be agnostic over the strength of dependence present in the respective models. 
	
\end{abstract}

\maketitle

\section{Introduction} 

\blfootnote{\textup{2010} \textit{Mathematics Subject Classification}: \textup{62G10}, \textup{62G20}, \textup{60C20}} 
\blfootnote{\textit{Keywords and phrases}: \textup{Ising Model}, 
	\textup{Signal Detection},
	\textup{Structured Sparsity},
	\textup{Sharp Constants}
}

Let $\bX=(X_1,\ldots,X_n)^\top\in \{\pm 1\}^n$ be a random vector with the joint distribution of $\bX$ given by an Ising model defined as:
\begin{align}
\P_{\beta, \bQ,\bmu}(\bX=\bx):=\frac{1}{Z(\beta,\mathbf{Q}, \mathbf{\bmu})}\exp{\left(\frac{\beta}{2}\bx^\top\mathbf{Q} \bx+\bmu^\top\bx\right)},\qquad \forall \bx \in \{\pm 1\}^n,
\label{eqn:general_ising}
\end{align}
where $\mathbf{Q}$ is an $n \times n$ symmetric and hollow matrix, $\bmu:=(\mu_1,\ldots,\mu_n)^\top\in \mathbb{R}^{n}$ is an unknown parameter vector to be referred to as the external magnetization vector, $\beta\in \mathbb{R}$ is a real number usually referred to as the ``inverse temperature", and $Z(\beta,\mathbf{Q}, \mathbf{\bmu})$ is a normalizing constant.  
It is clear that the pair $(\beta,\mathbf{Q})$ characterizes the dependence among the coordinates of $\bX$, and $X_i$'s are independent if $\beta\mathbf{Q}=\mathbf{0}_{n\times n}$.  The matrix $\mathbf{Q}$ will usually be associated with a certain sequence of simple labeled graphs $\mathbb{G}_n=(V_n,E_n)$ with vertex set $V_n=\{1,\dots,n\}$ and edge set $E_n \subseteq V_n\times V_n$ and corresponding $\mathbf{Q}=|V_n| G_n/2|E_n|$, where $G_n$ is the adjacency matrix of $\mathbb{G}_n$.  
Note that 
we do not absorb $\beta$ in the matrix $\bQ$. This is because we want to understand the effect of the nuisance parameter $\beta$ on the inference about $\bmu$. 

We are interested in testing against a collection of alternatives defined by a class of subsets $\mathcal{C}_s$ of $\mathbb{R}_+^n$ each of which is of size $s$. More precisely, given any class of subsets $\mathcal{C}_s\subset \{S\subset\mathbb{R}_+^n:|S|=s \}$ of $\mathbb{R}_+^n$ having size $s$ each, we consider testing the following hypotheses
\begin{equation} 
	H_0: \bmu=\mathbf{0} \quad {\rm vs} \quad H_1: \bmu \in \Xi(\mathcal{C}_s,A), \label{eqn:sparse_hypo}
\end{equation}
where
$${\Xi}(\mathcal{C}_s,A):=\left\{\begin{array}{c}\bmu\in \R_+^n: \mathrm{supp}(\bmu)\in \mathcal{C}_s  {\rm \ and\ }  \min\limits_{i\in {\rm supp}(\bmu)}\mu_i\geq A\end{array}\right\},$$
and
$$
{\rm supp}(\bmu):=\{i\in \{1,\ldots,n\}:\mu_i\ne 0\}.
$$
Thus the class of alternatives $\Xi(\mathcal{C}_s,A)$ puts non-zero signals on one of candidate sets in $\mathcal{C}_s$ where each signal set has size $s$.  
Of primary interest here is to explore the effect of $(\beta,\mathbf{Q})$ in testing \eqref{eqn:sparse_hypo} when $\mathcal{C}_s$ has low complexity in a suitable sense.

 \textcolor{black}{In this regard, previously, \cite{arias2011detection,arias2005near} studied the detection of block-sparse and thick shaped signals on lattices while \cite{addario2010combinatorial} considered general class of signals of combinatorial nature. However these papers crucially assume independence between the outcomes and thereby correspond to $\beta=0$ in~\eqref{eqn:general_ising} in our context. Following up on this line of research, several other papers have also considered detection of signals over lattices and networks (\textcolor{black}{see e.g. \cite{enikeeva2020bump,zou2017nonparametric,arias2018distribution,sharpnack2015detecting,walther2010optimal,butucea2013detection,konig2020multidimensional} and references therein}). However, in overwhelming majority of the literature, the underlying networks only  describe the nature of signals -- such as rectangles or thick clusters in lattices \citep{arias2011detection}. A fundamental question however remains -- ``how does dependence characterized by a network modulate the behavior of such detection problems ?" Only recently, \cite{enikeeva2020bump}  explored the effect of dependence \footnote{Effect of dependence in signal detection for Gaussian outcomes has also been explored for detecting unstructured arbitrary sparse signals \citep{hall2008properties,hall2010innovated}.} on such structured detection problems for stationary Gaussian processes -- with examples including linear lattices studied through the lens of Gaussian auto-regressive observation schemes. Dependence structures beyond Gaussian random variables are often more challenging to analyze (due to possible lack of closed form expressions of resulting distributions) and allow for interesting and different behavior of such testing problems -- see e.g. \cite{mukherjee2016global}. One of the motivations of this paper is to fill this gap in the literature and show how dependent binary outcomes can substantially change the results for detecting certain classes structured signals.}  One of the motivations of this paper is to pinpoint the precise effect of dependence on the behavior of such testing problems. In particular, \cite{mukherjee2016global,deb2020detecting} demonstrate how dependence might have a subtle effect on the minimax separation rate of sparse testing problems. In this paper we crystallize the effect of such dependence by going beyond optimal rates and characterizing sharp asymptotic constant for minimax separation while testing against suitably structured hypotheses $\mathcal{C}_s$.


To describe our results in the context of the model-problem pair \eqref{eqn:general_ising}-\eqref{eqn:sparse_hypo} we adopt a standard minimax framework as follows. Let a statistical test for $H_0$ versus $H_1$ be a measurable $\{0,1\}$ valued function of the data $\bX$, with $1$ denoting rejecting the null hypothesis $H_0$ and $0$ denoting not rejecting $H_0$. The worst case risk of a test $T: \{\pm 1\}^{\Lambda_n(d)}\to \{0,1\}$ for testing \eqref{eqn:sparse_hypo} is defined as 
\begin{align} 
	\mathrm{Risk}(T,{\Xi}(\mathcal{C}_s,A),\beta,\bQ)&:=\P_{\beta,\bQ,\mathbf{0}}\left(T(\bX)=1\right)+\sup_{\bmu \in {\Xi}(\mathcal{C}_s,A)}\P_{\beta,\bQ,\bmu}\left(T(\bX)=0\right). \label{eqn:risk}
\end{align}
We say that a sequence of tests $T_n$ corresponding to a sequence of model-problem pair (\ref{eqn:general_ising}) and (\ref{eqn:sparse_hypo}), to be asymptotically powerful (respectively asymptotically not powerful) against $\Xi(\mathcal{C}_s,A)$ if
\begin{equation}
	\label{eqn:powerful}
	\limsup\limits_{n\rightarrow \infty}\mathrm{Risk}(T_n,\Xi(\mathcal{C}_s,A),\beta,\bQ)= 0\text{ (respectively }\liminf\limits_{n\rightarrow \infty}\mathrm{Risk}(T_n,\Xi(\mathcal{C}_s,A),\beta,\bQ)>0).
\end{equation}
The goal of the current paper is to characterize how  for some low complexity class $\mathcal{C}_s$, the sparsity $s$,  and strength $A$ of the signal jointly determine if there is an asymptotically powerful test, and how the behavior changes with $(\beta,\bQ)$. 

With the above framework, our main results can be summarized as follows.

	\begin{enumerate}
\item [(i)] For some classical mean-field type models we show that detecting low complexity sets (see Theorem \ref{theorem:cw_known_beta} for exact definition) has same constant of detection for low and critical dependence and has a larger constant of minimax separation (i.e. strictly more information theoretic hardness) for higher dependence. Our examples naturally include the complete graph (Theorem \ref{theorem:cw_known_beta}), dense  Erd\H{o}s-R\'{e}nyi and dense random regular graphs(Theorem \ref{theorem:er_known_beta}).

\item [(ii)] For detecting thick rectangular signals in Ising models over lattices of general dimensions, we present the sharp minimax separation constants (Theorem \ref{thm:lattice_known_beta}) for low dependence and high dependence (under a ``pure phase" defined in Section \ref{sec:lattice}). In contrast to dense regular graphs, the problem has a strict monotone increasing nature of the constant as one approaches criticality from the low dependence direction (Lemma \ref{prop:monotonicity_susceptibility}). The exact monotonic nature of the constant  in the high dependence case is not clear and is left as future research direction.

\item [(iii)] We further demonstrate how the sharp optimal tests can be obtained adaptively over the dependence strength $\beta$ (Theorem \ref{thm:cw_unknown_beta} and Theorem \ref{thm:lattice_unknown_beta}).

\end{enumerate}

The rest of this paper is organized as follows. In Section \ref{sec:mean_field} we present our results for detecting low complexity type signals in dense regular type graphs. Subsequently, Section \ref{sec:lattice} considers detecting thick rectangular signals over lattices. Finally all the proofs and associated technical lemmas are collected in Section \ref{sec:proof_of_main_results}.

\subsection{Notation} $[n]$, for $\mathbf{v}=(v_1,\ldots,v_d)^T\in \mathbb{R}^d$ define $\bar{\mathbf{v}}=\frac{1}{d}\sum_{i=1}^d v_i$, for any set $A$ define $\mathbbm{1}(\cdot\in A)$ as the indicator function for the set. For any set $S\in \mathcal{S}$, let $\bmu_S(A)$ denote the vector with $\mu_i=A\mathbbm{1}_{i \in S}$. Let $\mathbf{1}$ denote the vector of all $1$s. We also let $\mathbbm{1}$ to denote generic indicator functions. \par
The results in this paper are mostly asymptotic (in $n$) in nature and thus requires some standard asymptotic  notations.  If $a_n$ and $b_n$ are two sequences of real numbers then $a_n \gg b_n$  (and $a_n \ll b_n$) implies that ${a_n}/{b_n} \rightarrow \infty$ (and ${a_n}/{b_n} \rightarrow 0$) as $n \rightarrow \infty$, respectively. Similarly $a_n \gtrsim b_n$ (and $a_n \lesssim b_n$) implies that $\liminf_{n \rightarrow \infty} {{a_n}/{b_n}} = C$ for some $C \in (0,\infty]$ (and $\limsup_{n \rightarrow \infty} {{a_n}/{b_n}} =C$ for some $C \in [0,\infty)$). Alternatively, $a_n = o(b_n)$ will also imply $a_n \ll b_n$ and $a_n=O(b_n)$ will imply that $\limsup_{n \rightarrow \infty} \ a_n / b_n = C$ for some $C \in [0,\infty)$). If $C>0$ then we write $a_n=\Theta(b_n)$.  If  $a_n/b_n\rightarrow 1$, then we  say $a_n \sim b_n$. For any $a,b \in \bZ$, $a \leq b$, let $[a,b] :=\{a,a+1,\ldots, b\}$. 

\section{Mean-Field type Interactions}\label{sec:mean_field}
In this section, we collect our results on the testing problem \eqref{eqn:sparse_hypo} for some specific examples of mean-field type models \citep{basak2015universality} such as Ising models on the complete graph, and dense Erd\H{o}s-R\'{e}nyi and random regular graphs.  
However, for the precise statement the upper bounds of our results we first define a class of signals $\mathcal{C}_s$. This is captured by a notion of complexity defined through the following weighted Hamming type of metric (see e.g. -- \cite{arias2005near,arias2011detection}) on $2^{[n]}$. Mathematically, for any two subsets $S_1,S_2\subset [n]$ we let $\gamma(S_1,S_2):=\sqrt{2}\left(1-\frac{|S_1\cap S_2|}{\sqrt{|S_1||S_2|}}\right)$
denote their distance. Subsequently, for any $\varepsilon>0$ we let $|\mathcal{N}(\mathcal{C}_s,\gamma,\varepsilon)|$ denote the $\varepsilon$-covering number of $\C_s$ w.r.t. $\gamma$. 
Our main result in terms of detecting signals in $\Xi(\mathcal{C}_s,A)$ pertains to classes of signals with suitably low complexity defined through the asymptotic behavior of $|\mathcal{N}(\mathcal{C}_s,\gamma,\varepsilon)|$. In this regard we first provide a complete picture for all temperature regimes in the mean-field Curie-Weiss model followed by demonstrating how similar results might be obtained in high temperature regimes for other dense regular type graphs.

\subsection{Complete Graph}\label{sec:curie_weiss}
We begin by stating and discussing our results for the Curie-Weiss model.

\begin{theorem}\label{theorem:cw_known_beta}
	Consider testing \eqref{eqn:sparse_hypo}  in the model \eqref{eqn:general_ising} with $\mathbf{Q}_{ij}=\frac{\mathbf{1}(i\neq j)}{n}$ correspond to the complete graph.
	\begin{enumerate}[label=\textbf{\roman*}.]
		\item  \label{thm:cw_known_beta_ub}
		Assume that $\mathcal{C}_s$ satisfies the following condition w.r.t. the metric $\gamma$ for some sequence $\varepsilon_n\rightarrow 0$: 
		 $$\Theta(\log n)= \log|\mathcal{N}(\mathcal{C}_s,\gamma,\varepsilon_n)|\ll s\ll n/\log|\mathcal{N}(\mathcal{C}_s,\gamma,\varepsilon_n)|.$$

Then the following hold.
	
		 \begin{enumerate}[label=\emph{\alph*})]
		 	\item  \label{thm:cw_known_beta_ub_high_temp} For $\beta < 1$, $s \ll \frac{n}{\log n}$, there exists a sequence of asymptotically powerful test if 
		 	\begin{equation}
		 	\liminf_{n \rightarrow \infty}\sqrt{s}\tanh(A) (\log|\mathcal{N}(\mathcal{C}_s,\gamma,\varepsilon_n)|)^{-1/2}>\sqrt{2}.
		 	\end{equation}
		 \item  \label{thm:cw_known_beta_ub_critical_temp}	For $\beta=1$, the same conclusion is valid for $s \ll \frac{\sqrt{n}}{\log n}$.
		 
		 \item \label{thm:cw_known_beta_ub_low_temp} For $\beta >1$ and $s \ll \frac{n}{\log n}$, there exists a sequence of asymptotically powerful test if 
		 	\begin{equation}
		 	\liminf_{n \rightarrow \infty}\sqrt{s}\tanh(A) (\log|\mathcal{N}(\mathcal{C}_s,\gamma,\varepsilon_n)|)^{-1/2}>\sqrt{2} \cosh(\beta m),
		 	\end{equation}
		 	where $m$ is the unique positive root of the equation $m=\tanh(\beta m)$.
		 \end{enumerate}

		
		\item\label{thm:cw_known_beta_lb} 
	Assume that, there exists a subset $\tilde{\mathcal{C}}_s\subset \mathcal{C}_s$ of disjoint sets such that $$\Theta(\log n)=\log|\tilde{\mathcal{C}}_s| \ll s \ll n/ \log|\tilde{\mathcal{C}}_s|.$$ 
	Then the following hold.
\\
		 \begin{enumerate}[label=\emph{\alph*})]
		 	\item \label{thm:cw_known_beta_lb_high_temp} For $\beta < 1$ and $s \ll \frac{n}{\log n}$, all tests are asymptotically powerless if
		 	\begin{equation}
		 	\liminf_{n \rightarrow \infty} \sqrt{s} \tanh(A) (\log|\tilde{\mathcal{C}}_s|)^{-1/2}<\sqrt{2}.
		 	\end{equation}
		 	
		 	\item \label{thm:cw_known_beta_lb_critical_temp} For $\beta =1$, the same conclusion is valid for $s \ll \frac{\sqrt{n}}{\log n}$. 
		 	
		 	\item \label{thm:cw_known_beta_lb_low_temp} For $\beta >1$ and $s \ll \frac{n}{\log n}$, no tests are asymptotically powerful if 
		 	\begin{equation}
		 	\liminf_{n \rightarrow \infty}\sqrt{s}\tanh(A) (\log|\tilde{\mathcal{C}}_s|)^{-1/2}< \sqrt{2} \cosh(\beta m),
		 	\end{equation}
		 \end{enumerate}
	\end{enumerate}
\end{theorem}



%
%


A few remarks are in order regarding the conditions, involved optimal procedures, and implications of Theorem \ref{theorem:cw_known_beta}. We first note that the upper and lower bounds are sharp as long as there exists $\tilde{C}_s,\mathcal{N}(\C_s,\gamma,\varepsilon_n)$ such that $\log|\tilde{C}_s|\sim \log |\mathcal{N}(\C_s,\gamma,\varepsilon_n)|$ for some $\varepsilon_n\rightarrow 0$. In particular, the sharp constants match for testing a class of signals whose size is dominated (on log-scale) by the size of a subclass of disjoint sets.    
Next we note that the conditions on $s$ posited in the various parts of the theorem are actually optimal. Indeed, as was discussed in \cite{deb2020detecting},  for $\beta=1$ and $s\gg \sqrt{n}$ the rate of detection is much faster (requiring only $s\tanh(A)\gg n^{1/4}$) and the problem might not offer a phase transition at the level of sharp constants.  Next we note that the optimal test above is based on a suitably calibrated scanning procedure. However, the scanning procedure is somewhat different depending on the regime of dependence. In high and critical dependence ($0\leq \beta\leq 1$) the procedure can be described as follows. For $S\in \mathcal{N}(\C_s,\gamma,\varepsilon_n)$ we first define $Z_S=\sum_{i\in S}X_i/\sqrt{s}$ and our scan test rejects for large values of $Z_{\max}:=\max\limits_{S\in \mathcal{N}(\C_s,\gamma,\varepsilon_n)} Z_S$.   In contrast, for the high dependence regime ($\beta>1$) we perform a somewhat randomized scan test by first generating $W_n \sim N(\bar{\bX},1/(n\beta))$ and subsequently for $W_n >0$, rejecting $H_0$  for large values of $Z_{\max}-m\sqrt{s}$ and when $W_n \leq 0$ rejecting $H_0$  for large values of $Z_{\max}+m\sqrt{s}$  \footnote{$m:=m(\beta)$ is the unique positive root of $m=\tanh(\beta m)$.} 
The fact that these sequence of tests are indeed sharp optimal is thereafter demonstrated by precise analyses of the Type II errors (through a mean-variance control) and a matching lower bound calculation obtained through a truncated second moment approach. Both the analyses of the tests and the proofs for  matching lower bounds rely on the moderation deviation behavior of $Z_S,S\in \mathcal{N}(\C_s,\gamma,\varepsilon_n) $ -- which are obtained through Lemma \ref{lemma:cw_moderate_deviation}. Finally, a direct analysis of the sharp constants of detection, $\sqrt{2}$ for $\beta\leq 1$ and $\sqrt{2\cosh^2(\beta m(\beta))}$ for $\beta>1$, reveals that the problem becomes harder as one moves away from the critical dependence $\beta=1$. In particular, the sharp constant of detection can be succinctly defined through by $\sqrt{2\cosh^2(\beta m(\beta))}$  which we display below in Figure 1 to demonstrate this phenomenon.
\begin{figure}[h]
	\begin{center}
	\includegraphics[scale=0.3]{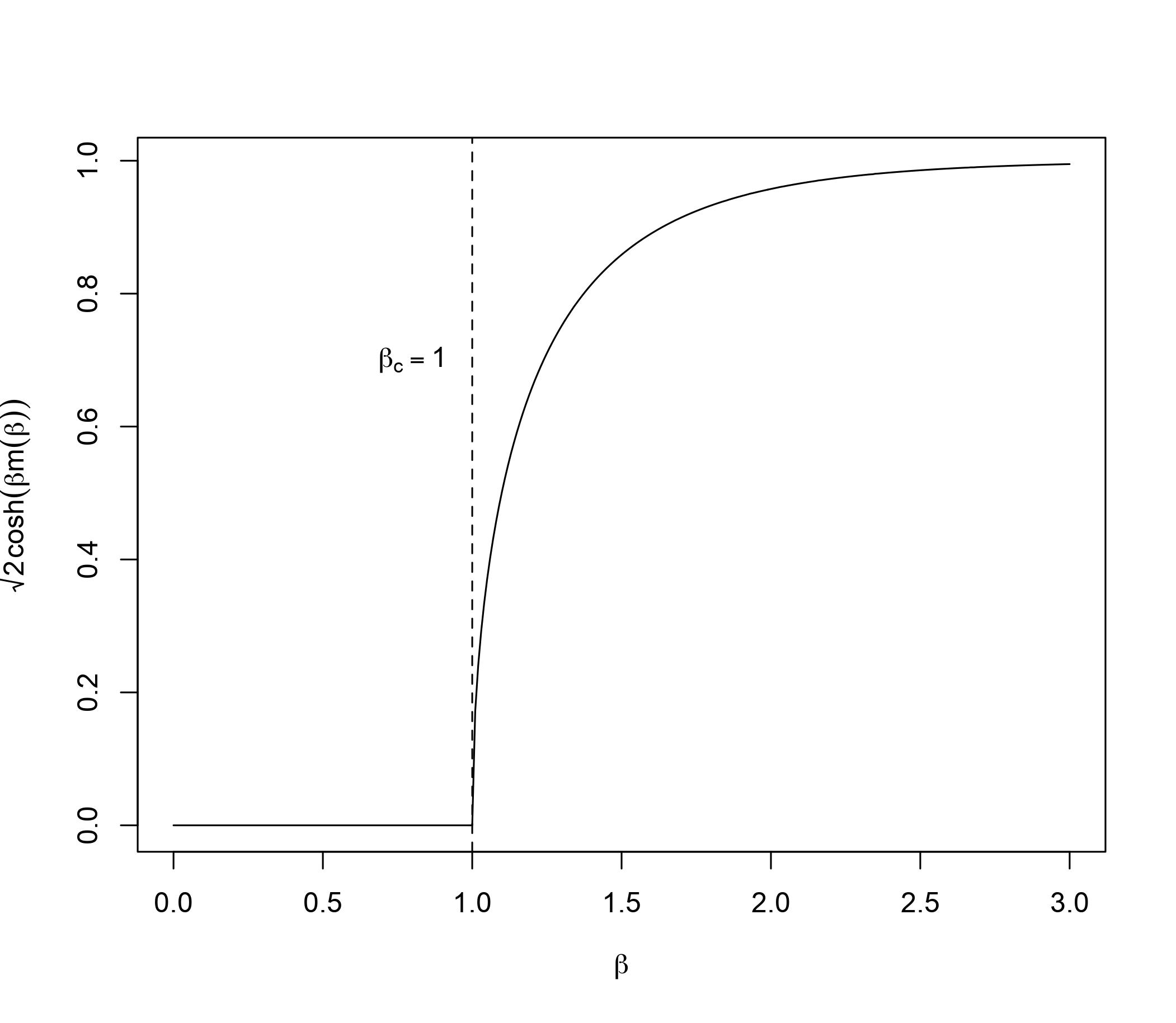}
\caption{Behavior of $\sqrt{2} \cosh(m(\beta))$}
	\end{center}
\end{figure}

To describe our next result, we note that the proof of the upper bound in Theorem \ref{theorem:cw_known_beta} assumes the knowledge of $\beta>0$. Our next result therefore pertains to showing that the sharp rates obtained above can actually be obtained adaptively over the knowledge of the inverse temperature $\beta>0$. To this end we begin by noting that using a consistent of $\beta>0$ might seem hopeless to begin with since consistent estimation of $\beta$ is not possible when $\beta\in [0,1)$. However, when $\beta\leq 1$, our optimal test does not depend on the specific knowledge of $\beta$. Therefore our idea can be described as follows: (1) construct a consistent test to decide whether $\beta\leq 1$ or $\beta> 1$; (2) if the test rejects in favor of $\beta> 1$ then use a pseudo-likelihood estimator of $\beta$ \textit{under the working model of $\bmu=0$} \footnote{this is important since joint estimation of $\beta,\bmu$ can be significantly harder information theoretically} and if the test return in favor of $\beta\leq 1$ then construct the $\beta$ independent optimal test in Theorem \ref{theorem:cw_known_beta}. 


\begin{theorem}\label{thm:cw_unknown_beta}
Theorem \ref{theorem:cw_known_beta} holds for unknown $\beta>0$  if  $\|\bmu\|_{\infty}=O(1)$.
	
\end{theorem}

\textcolor{black}{The proof of Theorem \ref{thm:cw_unknown_beta}, which is deferred to Section \ref{sec:proof_of_main_results}, requires the assumption on $\bmu$ in terms of its maximal element. Although this requirement can be relaxed to $\|\bmu\|_{\infty}=o(n/s)$, our arguments were unable to get rid of it completely. We keep further explorations in this regard for future research.
}


\subsection{Dense Regular Graphs}\label{sec:dense_reg}

The results in the Curie-Weiss model in the last section provides insight on the possible behavior of this testing problem under mean-field type models \citep{basak2015universality}.  Here we demonstrate that this intuition of similar behavior to the Curie-Weiss model with regard to this inferential problem is indeed true for some specific examples of mean-field type models such as dense Erd\H{o}s-R\'{e}nyi and random regular graphs, which is our first result in this direction.  In particular, we let $\mathbb{G}_n=(V_n,E_n)\sim \mathcal{G}_n(n,p)$ denote an Erd\H{o}s-R\'{e}nyi random graph with edges $E_n$ formed by joining pairs of vertices  $i,j\in V_n=\{1,\ldots,n\}$ independently with probability $p\in (0,1)$. In a similar vein we let $\mathbb{G}_n=(V_n,E_n)\sim \mathcal{G}_n(n,d)$ denote an randomly drawn graph from the collection of all $d$-regular graphs on $ V_n=\{1,\ldots,n\}$.

\begin{theorem}\label{theorem:er_known_beta}
The same conclusion of Theorem \ref{theorem:cw_known_beta} hold for any $\beta\geq 0$ when either (i) $\bQ=\frac{G_n}{np}$ with $G_n$ being the adjacency matrix of $\mathbb{G}_n\sim \mathcal{G}_n(n,p)$ with $p=\Theta(1)$;  or (ii) $\bQ=\frac{G_n}{d}$ with $G_n$ being the adjacency matrix of $\mathbb{G}_n\sim \mathcal{G}_n(n,d)$ with $d=\Theta(n)$.
\end{theorem}
A few remarks in order regarding the statement and proof of Theorem \ref{theorem:er_known_beta}. First, the result should be understood as a high probability statement w.r.t. the randomness of the underlying Erd\H{o}s-R\'{e}nyi random graph. In particular, we prove that same results as in Theorem \ref{theorem:cw_known_beta} hold with probability converging {\color{blue} to $1$} under the Erd\H{o}s-R\'{e}nyi measure on $\mathbb{G}_n$. In this regard, the requirement of $p=\Theta(1)$ is mostly used for the $\beta>1$ case and can be relaxed for $\beta\leq 1$ regime. However, to keep our discussions consistent over values of $\beta$ we only consider the $p=\Theta(1)$ case. Finally, the proof of the theorem mainly operates through careful comparison of suitable event probabilities and partition functions under Ising models over Erd\H{o}s-R\'{e}nyi and complete graphs respectively -- the proofs of which can be found in Section \ref{sec:technical_lemmas_mean_field}. We also show that the result above can be obtain without the knowledge of $\beta$.

\begin{theorem}\label{thm:er_unknown_beta}
Theorem \ref{theorem:er_known_beta} holds for unknown $\beta>0$  if  $\|\bmu\|_{\infty}=O(1)$.
\end{theorem}
The requirement of $\|\bmu\|_{\infty}=O(1)$ can be relaxed to $\|\bmu\|_{\infty}=o(n/s)$. However, at this moment we have not been able to relax this completely.


\section{Short Range Interactions on Lattices}\label{sec:lattice}

To describe the detection problem for nearest neighbor interaction type models, it is convenient to represent the points $i=1,\dots, n$ to be vertices of $d$-dimensional hyper-cubic lattice and the underlying graph (i.e. the graph corresponding to $\bQ$) to be the nearest neighbor (in sense of Euclidean distance) graph on these vertices. More precisely, given positive integers $n,d$, we consider a growing sequence of lattice boxes of dimension $d$ defined as 
$$\{\Lambda_{n,d}\}_{n \ge 1}=\{[-n^{1/d},n^{1/d}]^d\cap \mathbb{Z}^d\}_{n \ge 1}$$
where $\mathbb{Z}^d$ denotes the d-dimensional integer lattice. 
Subsequently, we consider a family of random variables defined on the vertices of $\Lambda_{n,d}$ as $\bX \in \{-1,+1\}^{\Lambda_{n,d}}$  with the following probability mass function (p.m.f.)
\begin{align}
\P_{\beta,\bQ,\bmu}(\bX=\bx)=\frac{1}{Z(\beta,\mathbf{Q}(\Lambda_{n,d}), \mathbf{\bmu})}\exp{\left(\frac{\beta}{2}\bx^\top\mathbf{Q}(\Lambda_{n,d}) \bx+\bmu^\top\bx\right)},\qquad \forall \bx \in \{\pm 1\}^{\Lambda_{n,d}},
\label{eqn:lattice_ising}
\end{align}
where as usual $\mathbf{Q}(\Lambda_{n,d})=(\bQ(\Lambda_{n,d})_{ij})_{i,j\in\Lambda_{n,d}}$ is a symmetric and hollow \textcolor{black}{array} (i.e. $\bQ(\Lambda_{n,d})_{ii}=0$ for all $i\in \Lambda_{n,d}$) with elements indexed by pairs of vertices in $\Lambda_{n,d}$ (organized in some pre-fixed lexicographic order), $\bmu:=(\mu_i:i\in \Lambda_{n,d})^\top\in \mathbb{R}^{\Lambda_{n,d}}$ referred to as the external magnetization vector indexed by vertices of $\Lambda_{n,d}$ , $\beta>0$ is the ``inverse temperature", and $Z(\beta,\mathbf{Q}(\Lambda_{n,d}), \mathbf{\bmu})$ is a normalizing constant. Note that in this notation $i,j\in \Lambda_{n,d}$ are vertices of the $d$-dimensional integer lattice and hence correspond to $d$-dimensional vector with integer coordinates. Therefore, using this notation, by nearest neighbor graph will correspond to the matrix $\bQ_{ij}=\bQ_{ij}(\Lambda_{n,d})=\mathbf{1}(0<\|i-j\|_1= 1)$.

Our results in this model will be derived for any dependence other than a critical dependence $\beta_c(d)$ to be defined next. The case of critical dependence for lattices remains open even in terms of optimal rate for the minimax separation and therefore we do not pursue the issue of optimal constants in this regime. To describe a notion of critical temperature, consider  $\mathcal{Q}(d)$ be the sequence of matrices $\{\bQ(\Lambda_{n,d})\}_{n\geq 1}$ and define
		\begin{align*}
		\beta_c(d)=\beta_{c}(\mathcal{Q}(d))=\inf\left\{\beta>0: \lim_{h\downarrow 0}\lim_{n\rightarrow \infty}\E_{\beta,\bQ(\Lambda_{n,d}),\bmu(h)}\left(\frac{1}{n}\sum_{i\in \Lambda_{n,d}} X_i\right)>0\right\},
		\end{align*}
		{where we let $\bmu(h)$ denote the vector in $\mathbb{R}^{|\Lambda_{n,d}|}$ with all coordinates equal to $h$.
		The existence and equivalence of the above notions (such as ones including uniqueness of infinite volume measure) of critical temperature in nearest neighbor Ising Model is a topic of classical statistical physics and we refer the interested reader to the excellent expositions in \cite{friedli2017statistical,duminil2017lectures} for more details.} {This value of $\beta_c(d)$ (which is known to be strictly positive for any fixed $d\geq 1$) is referred to as the critical inverse temperature in dimension $d$ and the behavior of the system of observations $\bX$ changes once $\beta$ exceeds this threshold. For $d=1$, it is known from the first work in this area \citep{ising1925beitrag} that $\beta_c(1)=+\infty$ and consequently the Ising model in 1-dimension is said to have no phase transitions. The seminal work of \cite{onsager1944crystal} provides a formula for $\beta_c(2)$ and obtaining an analytical formula for $\beta_c(d)$ for $d\geq 3$ remains open. Consequently, results only pertain to the existence of a strictly non-zero finite $\beta_c(d)$ which governs the macroscopic behavior of the system of observations $X_i,i\in \Lambda_{n,d}$ as $n\rightarrow \infty$. In particular, the average magnetization $n^{-1}\sum X_i$ converges to $0$ in probability for $\beta<\beta_c(d)$ and to an equal mixture of two delta-dirac random variables $m_+(\beta)$ and $m_-(\beta)=-m_+(\beta)$, for $\beta>\beta_c(d)$ (\cite{lebowitz1977coexistence}).} This motivates defining Ising models in pure phases as follows. 
		
	 Let $\Lambda^{\sf w}_{n,d}$ denote the graph obtained by identifying the vertices not in $\Lambda_{n,d}$ into a single vertex $\partial\Lambda_{n,d} $ and then erasing all the self loops. Let $\bQ(\Lambda^{\sf w}_{n,d})$ denote the adjacency matrix corresponding to nearest neighbour interaction in this modified graph.
	{	
		We denote 
		\begin{align*}
		\P^{+}_{\beta,\bQ(\Lambda_{n,d}),\bmu}(\bX=\bx)=\P_{\beta,\bQ(\Lambda^{\sf w}_{n,d}),\bmu}(\bX=\bx|\partial \Lambda_{n,d}=+1),\\ \P^{-}_{\beta,\bQ(\Lambda_{n,d}),\bmu}(\bX=\bx)=\P_{\beta,\bQ(\Lambda^{\sf w}_{n,d}),\bmu}(\bX=\bx|\partial\Lambda_{n,d}=-1),
		\end{align*}
		to be Ising Models (see \eqref{eqn:general_ising}) with $+$ and $-$ boundary conditions respectively. On the other hand $\P_{\beta,\bQ,\bmu}$ is referred to as the Ising model with \emph{free boundary condition}. It is well known \citep{friedli2017statistical}, that for $0\leq \beta<\beta_c(d)$ the asymptotic properties of the models $\P^{+}_{\beta,\bQ,\bmu}$, $\P^{-}_{\beta,\bQ,\bmu}$, and $\P_{\beta,\bQ,\bmu}$ are similar (i.e. they have all the same infinite volume $n\rightarrow \infty$ weak limit). However, for $\beta>\beta_c(d)$, the model $\P_{\beta,\bQ,\bmu}$ behaves asymptotically as the mixture of $\P^{+}_{\beta,\bQ,\bmu}$ and $\P^{-}_{\beta,\bQ,\bmu}$. 
		We only present our result for the measure $\P^{+}_{\beta,\bQ,\bmu}$ in such cases.} \textcolor{black}{Although we believe that a similar result might hold for both negative boundary condition (i.e. $\P^{-}_{\beta,\bQ,\bmu}$) as well as free boundary condition (i.e. the original model $\P_{\beta,\bQ,\bmu}$) we do not yet have access to a rigorous argument in this regard.} In the rest of this section, we therefore shall use the superscript ``$\rm bc$" in our probability, expectation, and variance operators (e.g.  $\P^{\rm bc}_{\beta,\bQ(\Lambda_{n,d}),\bmu}$,  $\E^{\rm bc}_{\beta,\bQ(\Lambda_{n,d}),\bmu}$, and  $\mathrm{Var}^{\rm bc}_{\beta,\bQ(\Lambda_{n,d}),\bmu}$) where $\mathrm{bc}\in \{\mathsf{f},+\}$ and stands for ``boundary condition" referring to either the free boundary condition (when $\rm bc=\mathsf{f}$) or plus boundary condition (when $\rm bc=+$).
		
		To present our results in this case, we begin with describing the structure of our alternatives. Since our model has an inherent geometry given by the lattice structure in $d$-dimensions, it is natural to consider signals which can be described by such geometry. Similar to one of the emblematic cases considered in \cite{arias2005near,arias2011detection,walther2010optimal,butucea2013detection,konig2020multidimensional}, here we discuss testing against block sparse alternatives of size $s$ define by $\Xi(\C_{s},A)$ with
 \begin{align}
\C_s=\C_{s,\rm rect}:=\left\{\prod\limits_{j=1}^d[a_j:b_j]\cap \Lambda_n(d): \ b_j-a_j=\lceil s^{1/d}\rceil\right\}.\label{eqn:cube_signals}
 \end{align}
 Although we only present the results for sub-cube detection with equal size of the sides , one can extend the results to detection of thick clusters (see \cite{arias2011detection} for details) with modifications of the arguments presented here. 
\textcolor{black}{The development and analysis of multi-scale tests similar to those explored in \cite{arias2005near,walther2010optimal,konig2020multidimensional} is also important. However we keep this for future research to keep our discussions in this paper focused on understanding the main driving principles behind the sharp constants of detection under Ising dependence.
} 

This a class of alternatives is known to require $\tanh(A)\asymp \sqrt{\frac{\log{n}}{s}}$ for consistent detection (see e.g. 	\citep{arias2005near,arias2008searching,arias2011detection} for independent outcomes case and \cite{enikeeva2020bump,deb2020detecting} for dependent models) and allows a sharp transition at a level of multiplicative constants. Indeed, such sharp constants of phase transition has been derived for either independent outcome models 
		\citep{arias2005near,arias2008searching,arias2011detection} and for dependent Gaussian outcomes \cite{enikeeva2020bump}. For our problem, the derivation of the sharp optimal constant of detection is somewhat subtle and we first discuss a way to formalize this asymptotic constant below.
		
		We begin by noting that it is reasonable to believe that a sequence of optimal test can be obtained by  scanning over a suitable subclass of the potentially signal rectangles. We define the test first to gain intuition about the sharp constant of detection. To put ourselves in the context of notation similar to Section \ref{sec:mean_field}, we use $\C_{s,\rm rect}$ to denote the class of all thick rectangles of volume $s$ and for any $S\in \C_{s,\rm rect}$
		\begin{align*}
		    Z_{S}=\frac{1}{\sqrt{s}}\sum_{i\in S}\left(X_i-\E^{\rm bc}_{\beta,\bQ(\Lambda_{n,d}),\mathbf{0}}(X_i)\right),
		\end{align*}
		 We would like to scan over a suitable subclass $\tilde{\C}_{s,\rm rect}$ which captures the essential complexity of $\C_{s,\rm rect}$ both in terms of theoretical and computational aspects. From a theoretical perspective, we shall require a precise understanding of the moderate deviation behavior of $Z_S$ and which in turn crucially relies on understanding the asymptotic behavior of the variance of $Z_S$. In particular, it is reasonable to believe that in the ``pure phase" (which is the case for free boundary high temperature and low temperature plus boundary case) $Z_S$ should asymptotically behave like a Gaussian random variable and thus its moderate deviation exponent is characterized by its variance (See \cite[Theorem V.7.2]{ellis2006entropy} for a result of this nature). To justify that there is a valid candidate for the limit of $\mathrm{Var}^{+}_{\beta,\bQ(\Lambda_{n,d}),\mathbf{0}}(Z_S)$ we have the following result which is one of the main components of this section.
		
\begin{theorem}\label{thm:variance_limit_lattice}
{Suppose $\mathrm{dist}(S,\partial\Lambda_{n,d}):=\inf\limits_{j \in \partial\Lambda_{n,d}, i\in S}\|i-j\|_1\gg \log^2{n}$.}  Then there exists $\chi^{\mathsf{f}}(\beta)$ (for $\beta<\beta_c(d)$) and $\chi^{+}(\beta)$ (for $\beta>\beta_c(d)$) such that
\begin{align}
    \lim_{s\rightarrow \infty}\mathrm{Var}^{\mathrm bc}_{\beta,\bQ(\Lambda_{n,d}),\mathbf{0}}(Z_S)&=\chi^{\mathrm bc}(\beta),\quad \text{for} \quad \beta<\beta_c(d), \mathrm{bc} \in \{\sf f,+\}; \label{eqn:var_limit_lattice_hightemp}\\
    \lim_{s\rightarrow \infty}\mathrm{Var}^{+}_{\beta,\bQ(\Lambda_{n,d}),\mathbf{0}}(Z_S)&=\chi^{+}(\beta),\quad \text{for} \quad \beta>\beta_c(d). \label{eqn:var_limit_lattice_lowtemp}
\end{align}
\end{theorem}

Armed with Theorem \ref{thm:variance_limit_lattice} we are now ready to state the main sharp detection threshold for detecting rectangles over lattices.

\begin{theorem}\label{thm:lattice_known_beta}
Suppose $\bX\sim \P^{\rm bc}_{\beta,\bQ(\Lambda_{n,d}),\bmu}$ for $\rm bc \in \{\mathsf{f}, +\}$and consider testing \eqref{eqn:sparse_hypo} against $\C_s=\C_{s,\rm rect}$. Then the following hold with $\rm bc \in \{\mathsf{f},+\}$ for $\beta<\beta_c(d)$ and $\rm bc=+$ for $\beta>\beta_c(d)$.
\begin{enumerate}
    \item A test based on $\max_{S\in \tilde{\C}_{s,\rm rect}} Z_S$ is asymptotically powerful if $s\gg (\log{n})^d$ $$\liminf\tanh(A)\sqrt{s/\log{(n/s)}}>\sqrt{2\chi^{\rm bc}(\beta)}.$$
    
    \item All tests are asymptotically powerless if $\limsup\tanh(A)\sqrt{s/\log{(n/s)}}<\sqrt{2\chi^{\rm bc}(\beta)}$. 
\end{enumerate}

\end{theorem}

A few remarks are in order regarding the results in Theorem \ref{thm:lattice_known_beta}. First, we have not tried optimize the requirement $s\gg (\log{n})^d$ in our upper bound above. Indeed, as noted in \cite{deb2020detecting} one needs $s\gg \log{n}$ for any successful detection and hence our results on sharp constants matches the requirement on $s$ up to log-factors. Moreover, our next verifies that the susceptibility is indeed an increasing function of $0\leq \beta<\beta_c(d)$.

\begin{prop}\label{prop:monotonicity_susceptibility}
$\chi(\beta)$ is differentiable and strictly monotonically increasing for $\beta \in (0,\beta_c(d))$.
\end{prop}
It is worth noting that this monotonicity without the strictness can be seen as a consequence of the Edwards-Sokal coupling and the monotonicity in $p$ of the FK-Ising model (\cite{grimmett2006random}). Finally, we demonstrate that the results above can be obained without the knowledge of $\beta\neq \beta_c(d)$.

\begin{theorem}\label{thm:lattice_unknown_beta}
Theorem \ref{thm:lattice_known_beta} continues to hold for unknown $\beta\neq \beta_c(d)$ and $\rm bc=\mathsf{f}$.
	
\end{theorem}

Although this completes the picture for non-critical temperature in nearest neighbor Ising models over lattices, the behavior of the testing problem at $\beta=\beta_c(d)$ remains open. At this moment we believe that the blessing of getting a better rate and/or constant at criticality continues to hold for lattices as well.  


\section{Discussions}\label{sec:discussions}
In this paper we have considered sharp constants of detecting structured signals under Ising dependence. Although we have derived sharp phase transitions in some popular classes of of both mean-field type Ising models (at all temperatures) and nearest-neighbor Ising models on lattices (at all non-critical temperatures) several related directions remain open. As an immediate interesting question pertains to the Ising model on lattices and figuring out the exact detection thresholds at the critical temperature to complete the narrative of precise benefit of critical dependence in this model. As was discussed in \cite{mukherjee2019testing} this might require new ideas. Moreover, even for non-critical temperatures it remains to explore the multi-scale procedures for adaptive testing of thick clusters for Ising models over lattices (see e.g. \cite{arias2005near,walther2010optimal,konig2020multidimensional} and references therein). Moreover distributional approximation for the test statistics used here is also a crucial direction for the sake of improved practical applicability of our result. We keep these questions as future research directions.

\section{Acknowledgements}
GR is supported by NSERC 50311-57400.

\bibliographystyle{plainnat}
\bibliography{biblio_structured}

\section{Proofs of Main Results}\label{sec:proof_of_main_results} We divide the proofs our main results according to Sections \ref{sec:mean_field} and \ref{sec:lattice}.

\subsection{\bf Proof of Results in Section \ref{sec:curie_weiss}}
\subsubsection{\bf Proof of Theorem \ref{theorem:cw_known_beta}}

\paragraph{\bf Proof of Theorem \ref{theorem:cw_known_beta}\ref{thm:cw_known_beta_ub}} We divide our proof according to the various parts of the theorem.

\paragraph{\bf Proof of Theorem \ref{theorem:cw_known_beta}\ref{thm:cw_known_beta_ub}\ref{thm:cw_known_beta_ub_high_temp} }
Recall from the discussion following the statement of Theorem \ref{theorem:cw_known_beta} that the claimed optimal test is given the by the scan test which can be described as follows. For $S\in \mathcal{N}(\C_s,\gamma,\varepsilon_n)$ we first define $Z_S=\sum_{i\in S}X_i/\sqrt{s}$ and our scan test rejects for large values of $Z_{\max}:=\max\limits_{S\in \mathcal{N}(\C_s,\gamma,\varepsilon_n)} Z_S$. The cut-off for the test is decided by the moderate deviation behavior of $Z_S$'s given in Lemma \ref{lemma:cw_moderate_deviation} -- which implies that for any $\delta>0$, the test given by $T_n(\delta)=\mathbf{1}\left\{Z_{\max}>\sqrt{2(1+\delta)\log{|\mathcal{N}(\C_s,\gamma,\varepsilon_n)|}}\right\}$ has Type I error converging to $0$.

Turning to the Type II error, consider any $\P_{\beta,\bQ,\bmu}\in \Xi(\C_s,A)$ and note that by monotonicity arguments (i.e. stochastic increasing nature of the distribution of $Z_{\max}$ as a function of coordinates of $\bmu$) it is enough to restrict to the case where $A=\Theta(\sqrt{\log{|\mathcal{N}(\C_s,\gamma,\varepsilon_n)|}})$. Thereafter, note that by GHS inequality(cf. \ref{lemma:GHS}) one has $\mathrm{Var}_{\beta,\bQ,\bmu}(\sum_{i\in \tilde{S}^{\star}} X_i)\leq \mathrm{Var}_{\beta,\bQ,\bzero}(\sum_{i\in \tilde{S}^{\star}} X_i)=O(s)$ for $\beta<1$ since $s\ll \frac{n}{\log{n}}$ by appealing to \cite[Lemma 9(a)]{deb2020detecting}. As a result, $\frac{\sum_{\i\in \tilde{S^{\star}}}(X_i-\E_{\beta,\bQ,\bmu}(X_i))}{\sqrt{s}}=O_{\P_{\beta,\bQ,\bmu}}(1)$. Therefore, as usual it is enough to show that there exists $\delta>0$ such that $t_n(\delta)-\frac{1}{\sqrt{s}}\E_{\beta,\bQ,\bmu}\left(\sum_{i\in \tilde{S}^{\star}}X_i\right)\rightarrow -\infty$. To show this end, first let $S^{\star}\in \C_s$ be such that the signal lies on $S^\star$ i.e. for all $i\in S^{\star}$ one has $\bmu_i\geq A$. Note that by monotonicity arguments it is enough to consider $\bmu_i=A$. By definition of covering, we can find a $\tilde{S}^{\star}\in \mathcal{N}(\C_s,\gamma,\varepsilon_n)$ such that $\gamma\left(\tilde{S}^{\star},S^{\star}\right)\leq \varepsilon_n$ i.e. $|\tilde{S}^{\star}\cap S^{\star}|\geq s(1-\varepsilon_n/\sqrt{2})$. 


Thereafter note that by Lemma 
\ref{lemma:expectation_underalt_lower_bound} we have that
\begin{align*}
        \E_{\beta,\bQ,\bmu}(\sum_{i\in \tilde{S}^{\star}}X_i)\geq A|\tilde{S}^{\star}\cap S^{\star}|-A^2\sum_{i\in \tilde{S}^{\star}\cap S^{\star}}\sum_{j\in S}\mathrm{Cov}_{\beta,\bQ,\tilde{\mathbf{\eta}}_{S^{\star}}(A)}(X_i,X_j),
    \end{align*}
However, by \cite[Lemma 9 (a)]{deb2020detecting} we have that
\begin{align*}
    A^2\sum_{i\in \tilde{S}^{\star}\cap S^{\star}}\sum_{j\in S}\mathrm{Cov}_{\beta,\bQ,\tilde{\mathbf{\eta}}_{S^{\star}}(A)}(X_i,X_j)&\lesssim \frac{\log{|\mathcal{N}(\C_s,\gamma,\varepsilon_n)|}}{s}\left(\frac{|\tilde{S}^{\star}\cap S^{\star}|| S^{\star}|}{n}+|\tilde{S}^{\star}\cap S^{\star}|\right)\\
    &\leq \frac{\log{|\mathcal{N}(\C_s,\gamma,\varepsilon_n)|s}}{{n}}+\log{|\mathcal{N}(\C_s,\gamma,\varepsilon_n)|}\ll As.
\end{align*}
 Consequently, we immediately have that there exists $\epsilon>0$ such that
\begin{align*}
     \E_{\beta,\bQ,\bmu}\left(\sum_{i\in \tilde{S}^{\star}}X_i\right)&\geq As(1+o(1))\geq \sqrt{2(1+\epsilon)s\log{|\mathcal{N}(\C_s,\gamma,\varepsilon_n)|}}
\end{align*}
Therefore, we can conclude that for any $\delta<\epsilon$ we have  $t_n(\delta)-\frac{1}{\sqrt{s}}\E_{\beta,\bQ,\bmu}\left(\sum_{i\in \tilde{S}^{\star}}X_i\right)\rightarrow -\infty$. This completes the proof of the upper bound for $0<\beta<1$.

\paragraph{\bf Proof of Theorem \ref{theorem:cw_known_beta}\ref{thm:cw_known_beta_ub}\ref{thm:cw_known_beta_ub_critical_temp} }
An optimal test is given the by the scan test which can be described similar to the case $\beta<1$. For $S\in \mathcal{N}(\C_s,\gamma,\varepsilon_n)$, define $Z_S=\sum_{i\in S}X_i/\sqrt{s}$ and reject for large values of $Z_{\max}:=\max\limits_{S\in  \mathcal{N}(\C_s,\gamma,\varepsilon_n)} Z_S$. Lemma \ref{lemma:cw_moderate_deviation} implies that for any $\delta>0$, the test given by $T_n(\delta)=\mathbf{1}\left\{Z_{\max}>\sqrt{2(1+\delta)\log{|\mathcal{N}(\C_s,\gamma,\varepsilon_n)|}}\right\}$ has Type I error converging to $0$ whenever $s\ll \sqrt{n}/\log{n}$.

Turning to the Type II error, consider any $\P_{\beta,\bQ,\bmu}\in \Xi(\C_s,A)$ and note that by monotonicity arguments it is enough to restrict to the case where $A=\Theta(\sqrt{\log{|\mathcal{N}(\C_s,\gamma,\varepsilon_n)|}})$. Thereafter, note that by GHS inequality(cf. \ref{lemma:GHS}) one has $\mathrm{Var}_{\beta,\bQ,\bmu}(\sum_{i\in \tilde{S}^{\star}} X_i)\leq \mathrm{Var}_{\beta,\bQ,\bzero}(\sum_{i\in \tilde{S}^{\star}} X_i)=O(s)$ even for $\beta=1$ since $s\ll \frac{\sqrt{n}}{\log{n}}$ by appealing to \cite[Lemma 9(c)]{deb2020detecting}. As a result, $\frac{\sum_{\i\in \tilde{S^{\star}}}(X_i-\E_{\beta,\bQ,\bmu}(X_i))}{\sqrt{s}}=O_{\P_{\beta,\bQ,\bmu}}(1)$. Therefore, as usual it is enough to show that there exists $\delta>0$ such that $t_n(\delta)-\frac{1}{\sqrt{s}}\E_{\beta,\bQ,\bmu}\left(\sum_{i\in \tilde{S}^{\star}}X_i\right)\rightarrow -\infty$. To show this end, first let $S^{\star}\in \C_s$ be such that the signal lies on $S^\star$ i.e. for all $i\in S^{\star}$ one has $\bmu_i\geq A$. \textcolor{black}{Note that by monotonicity arguments it is enough to consider $\bmu_i=A$.} By definition of covering, we can find a $\tilde{S}^{\star}\in \mathcal{N}(\C_s,\gamma,\varepsilon_n)$ such that $\gamma\left(\tilde{S}^{\star},S^{\star}\right)\leq \varepsilon_n$ i.e. $|\tilde{S}^{\star}\cap S^{\star}|\geq s(1-\varepsilon_n/\sqrt{2})$. 


Thereafter note that by Lemma 
\ref{lemma:expectation_underalt_lower_bound} we have that
\begin{align*}
        \E_{\beta,\bQ,\bmu}(\sum_{i\in \tilde{S}^{\star}}X_i)\geq A|\tilde{S}^{\star}\cap S^{\star}|-A^2\sum_{i\in \tilde{S}^{\star}\cap S^{\star}}\sum_{j\in S}\mathrm{Cov}_{\beta,\bQ,\tilde{\mathbf{\eta}}_{S^{\star}}(A)}(X_i,X_j),
    \end{align*}
However, by \cite[Lemma 9 (c)]{deb2020detecting} we have that
\begin{align*}
    A^2\sum_{i\in \tilde{S}^{\star}\cap S^{\star}}\sum_{j\in S}\mathrm{Cov}_{\beta,\bQ,\tilde{\mathbf{\eta}}_{S^{\star}}(A)}(X_i,X_j)&\lesssim \frac{\log{|\mathcal{N}(\C_s,\gamma,\varepsilon_n)|}}{s}\left(\frac{|\tilde{S}^{\star}\cap S^{\star}|| S^{\star}|}{\sqrt{n}}+|\tilde{S}^{\star}\cap S^{\star}|\right)\\
    &\leq \frac{\log{|\mathcal{N}(\C_s,\gamma,\varepsilon_n)|s}}{\sqrt{n}}+\log{|\mathcal{N}(\C_s,\gamma,\varepsilon_n)|}\ll As.
\end{align*}
 Consequently, we immediately have that there exists $\epsilon>0$ such that
\begin{align*}
     \E_{\beta,\bQ,\bmu}\left(\sum_{i\in \tilde{S}^{\star}}X_i\right)&\geq As(1+o(1))\geq \sqrt{2(1+\epsilon)s\log{|\mathcal{N}(\C_s,\gamma,\varepsilon_n)|}}
\end{align*}
Therefore, we can conclude that for any $\delta<\epsilon$ we have  $t_n(\delta)-\frac{1}{\sqrt{s}}\E_{\beta,\bQ,\bmu}\left(\sum_{i\in \tilde{S}^{\star}}X_i\right)\rightarrow -\infty$. This completes the proof of the upper bound for $\beta=1$.

\paragraph{\bf Proof of Theorem \ref{theorem:cw_known_beta}\ref{thm:cw_known_beta_ub}\ref{thm:cw_known_beta_ub_low_temp} }
We use a randomized scan test here described as follows: Given data $X_i, i\in [n]$, generate a random variable $W_n \sim N(\bar{\bX},1/(n\beta))$. If $W_n >0$, we reject the null hypotheses if $Z_{\max}:=\max\limits_{S\in \mathcal{N}(\C_s,\gamma,\varepsilon_n)} Z_S>m\sqrt{s}+t_n(\delta)$. If $W_n \leq 0$ we reject if $Z_{\max}>-m\sqrt{s}+t_n(\delta)$, where $m:=m(\beta)$ is the unique positive root of $m=\tanh(\beta m)$ and $t_n(\delta)=\sqrt{2(1+\delta)(1-m^2)\log{|\mathcal{N}(\C_s,\gamma,\varepsilon_n)|}}$. It turns out that our analysis of this test works whenever $\sum_{i=1}^n \mu_i\ll n$. This is however not an issue since  the test based on conditionally centered $\sum_{i=1}^n X_i$ works as soon as $\sum_{i=1}^n \mu_i\gg \sqrt{n}$ \citep{mukherjee2016global} -- and hence simple Bonferroni correction (i.e. the test rejects when either the randomized test described above or the the conditionally centered sum based test of \cite{mukherjee2016global}) yields our desired result. Hence we focus here on the case when $\sqrt{n}\ll \sum_{i=1}^n \mu_i\ll n^{3/4}$.

By the moderate deviation behavior of $Z_S$'s given in Lemma \ref{lemma:cw_moderate_deviation}, the Type-I error converges to $0$. \par

Turning to the Type II error, consider any $\P_{\beta,\bQ,\bmu}\in \Xi(\C_s,A)$. Now, let $S^*$ be the rectangle with non-zero signal with the true signal set under the alternative.
We choose $\delta$ with $\tanh(A)=\sqrt{2(1+\delta)^2(1-m^2)^{-1}\log{|\mathcal{N}(\C_s,\gamma,\varepsilon_n)|}/s}$. By definition of covering, we can find a $\tilde{S}^{\star}\in \mathcal{N}(\C_s,\gamma,\varepsilon_n)$ such that $\gamma\left(\tilde{S}^{\star},S^{\star}\right)\leq \varepsilon_n$ i.e. $|\tilde{S}^{\star}\cap S^{\star}|\geq s(1-\varepsilon_n/\sqrt{2})$. We want to control $\P_{\beta,\bQ,\bmu}(A_{n,\tilde{S}^{\star}} \cup B_{n,\tilde{S}^{\star}})$ where
\begin{align}\label{eq:rejlt}
A_{n,S}&=\{Z_S>m \sqrt{s}+t_n(\delta),W_n>0\},\nonumber\\
B_{n,S}&=\{Z_S>-m \sqrt{s}+t_n(\delta),W_n \leq 0\}.
\end{align}
We have to show
\begin{equation}
\P_{\beta,\bQ,\bmu}(A_{n,\tilde{S}^{\star}}\cup B_{n,\tilde{S}^{\star}})\rightarrow 1.
\end{equation}
We plan to show $\P_{\beta,\bQ,\bmu}(A_{n,\tilde{S}^{\star}}) -\P_{\beta,\bQ,\bmu}(W_n>0)\rightarrow 0$ and $\P_{\beta,\bQ,\bmu}(B_{n,\tilde{S}^{\star}})-\P_{\beta,\bQ,\bmu}(W_n\leq0)\rightarrow 0$. We prove of the first limit and the proof of the second one is similar.
\begin{align*} 
&\P_{\beta,\bQ,\bmu}(A_{n,\tilde{S}^{\star}})=\E_{\beta,\bQ,\bmu}\left(\mathbf{1}(W_n>0)\P_{\beta,\bQ,\bmu}\left[\bar{X}_S>m(\beta)+t_n(\delta)/\sqrt{s}|W_n\right]\right)\\
&\le \E_{\beta,\bQ,\bmu}\left(\mathbf{1}(W_n>0)\underbrace{\P_{\beta,\bQ,\bmu}\left[\frac{\sum_{i\in S}X_i-\E_{\beta,\bQ,\bmu}(\sum\limits_{i\in S}X_i|W_n)}{\sqrt{s}}>\sqrt{s}(m-\tanh(\beta W_n+A))+t_n(\delta)|W_n\right]}_{G_n(W_n)}\right)
\end{align*}
We claim that $G_n(W_n)\mathbf{1}(W_n>0)-\mathbf{1}(W_n>0)\stackrel{\P}{\rightarrow} 0$. This will imply $\P_{\beta,\bQ,\bmu}(A_n) -\P_{\beta,\bQ,\bmu}(W_n>0)\rightarrow 0$ by DCT. To show the in probability convergence of $G_n(W_n)\mathbf{1}(W_n>0)$ we note that for any $\epsilon>0$
\begin{align*}
\P_{\beta,\bQ,\bmu}(\mathbf{1}(W_n>0)(1-G_n(W_n))>\epsilon)&\leq \P_{\beta,\bQ,\bmu}(G_n(W_n)<1-\epsilon,W_n>0).
\end{align*}
We therefore need to understand $G_n(W_n)$ on the event $W_n>0$. 
{\color{black}By Lemma \ref{lemma:cw_aux_var_alt_high_temp}, $\P_{\beta,\bQ,\bmu}(|W_n-m|>\zeta_n,W_n>0)\rightarrow 0$  for some slowly decreasing $\zeta_n$.} More specifically, this follows from  \ref{lemma:cw_aux_var_alt_high_temp} since this lemma not only implies concentration of $W_n$ around $m_n$ in $\sqrt{n}$-scale and but also that $m_n-m=\Theta(\frac{1}{n}\sum_{i=1}^n\mu_i)=o(n^{-1/4})$ by the property of $\bmu$. Next fix $\eta>0$ such that if $W_n>0,|W_n-m|\leq \zeta_n$ for some decreasing  $\frac{1}{\sqrt{n}}\ll \zeta_n\rightarrow \ll \frac{1}{n^{1/4}}$, $\sech^2(\beta W_n+A)\leq (1+\eta)^2\sech^2(\beta m)$ and  $\sqrt{s}(m-\tanh(\beta W_n+A))\leq -(1-\eta)\sech^2(\beta m)A\sqrt{s}$. Hence, on the event $W_n>0,|W_n-m|\leq \zeta_n$ we have 
\begin{align*}
\ & \sqrt{\frac{s}{\sech^2(\beta W_n+A)}}(m-\tanh(\beta W_n+A))+\frac{t_n(\delta)}{\sech^2(\beta W_n+A)}\\
& \leq -\sqrt{2\log{|\mathcal{N}(\C_s,\gamma,\varepsilon_n)|}}\left[\frac{(1-\eta)(1+\delta)}{1+\eta}-1\right]
\end{align*}
Choosing $\eta$ small enough we have the last display $\Theta(-\sqrt{\log|\mathcal{N}(\C_s,\gamma,\varepsilon_n)|})$ and the result follows by Chebyshev's Inequality. The same proof goes through for $W_n\leq 0$ by appealing to Lemma \ref{lemma:cw_aux_var_alt_high_temp}.


\paragraph{\bf Proof of Theorem \ref{theorem:cw_known_beta}\ref{thm:cw_known_beta_lb}} We divide our proof according to the various parts of the theorem.

\paragraph{\bf Proof of Theorem \ref{theorem:cw_known_beta}\ref{thm:cw_known_beta_lb}\ref{thm:cw_known_beta_lb_high_temp} }
We follow the path of truncated second moment method with respect to a suitable prior over $\tilde{\C}_s\subset \C_s$. Owing to the exchangeability of the Curie-Weiss model it is natural to consider the uniform prior $\pi$ (say) over all $\bmu\in \Xi(\C_s,A)$ such  $\mathrm{supp}(\bmu)\in\tilde{C}_s$ and $\bmu_i=A$ for $i\in \mathrm{supp}(\bmu)$. The likelihood ratio $L_{\pi}$ (say) corresponding to this prior can be written as
\begin{align*}
    L_{\pi}=\frac{1}{|\tilde{C}_s|}\sum_{S\in \tilde{C}_s}\frac{Z(\beta,\bQ,\bmu_S(A))}{Z(\beta,\bQ,\bzero)}\exp\left(A\sum_{i\in S}X_i\right),
\end{align*}
where $\bmu_S(A)$ is the vector with support $S$ and entries equal to $A$ on its support. Now recall the definition of $Z_{S}$ from the proof of Theorem \ref{theorem:cw_known_beta}\ref{thm:cw_known_beta_ub}\ref{thm:cw_known_beta_ub_high_temp} define an event $\Omega_n(S,\delta)=\{Z_{S}\leq \tilde{t}_n(\delta)\}$ -- where for any $\delta>0$ we define $\tilde{t}_n(\delta)=\sqrt{2(1+\delta)\log|\tilde{C}_s|}$. Subsequently, we let 
\begin{align}
    \tilde{L}_{\pi}(\delta):=\frac{1}{|\tilde{C}_s|}\sum_{S\in \tilde{C}_s}\frac{Z(\beta,\bQ,\bmu_S(A))}{Z(\beta,\bQ,\bzero)}\exp\left(A\sum_{i\in S}X_i\right)\mathbf{1}(\bX\in \Omega_n(S,\delta))\label{eqn:trunctaed_lr_cw_high_temp}
\end{align}
denote a truncated likelihood ratio at level $\delta>0$. It is thereafter enough the verify that there exists $\delta>0$ such that the following three claims hold \citep{Ingster5}
\begin{align}
    \P_{\beta,\bQ,\bzero}(L_{\pi}\neq \tilde{L}_{\pi}(\delta))&\rightarrow 0 \label{claim:trunctaed_untruncated_lr_cw_high_temp};\\
\E_{\beta,\bQ,\bzero}(\tilde{L}_{\pi}(\delta))&\rightarrow 1 \label{claim:expec_trunctaed_lr_cw_high_temp};\\
\E_{\beta,\bQ,\bzero}(\tilde{L}_{\pi}^2(\delta))&\leq 1+o(1) \label{claim:var_trunctaed_lr_cw_high_temp}.
\end{align}
We now show them in sequence. To show \eqref{claim:trunctaed_untruncated_lr_cw_high_temp} note that 
\begin{align*}
    \P_{\beta,\bQ,\bzero}(L_{\pi}\neq \tilde{L}_{\pi}(\delta))=\sum_{S\in \tilde{C}_s}\P_{\beta,\bQ,\bzero}\left(\frac{1}{\sqrt{s}}\sum_{i\in S}X_i>\tilde{t}_n(\delta)\right)\rightarrow 0.
\end{align*}
The convergence to $0$ in the display above follows from Lemma \ref{lemma:cw_moderate_deviation} for any $\delta>0$ -- by the same verbatim argument that showed the Type I error convergence to $0$ in the proof of Theorem \ref{theorem:cw_known_beta}\ref{thm:cw_known_beta_ub}\ref{thm:cw_known_beta_ub_high_temp}. 

Next we turn to \eqref{claim:expec_trunctaed_lr_cw_high_temp}. To verify this, we first note that by a simple change of measure
\begin{align*}
   \E_{\beta,\bQ,\bzero}(\tilde{L}_{\pi}(\delta))=\frac{1}{|\tilde{\C}_s|}\sum_{S\in \tilde{C}_s}\P_{\beta,\bQ,\bmu_S(A)}\left(\frac{1}{\sqrt{s}}\sum_{i\in S}X_i\leq\tilde{t}_n(\delta)\right).
\end{align*}
To analyze the R.H.S of the display above note that $\frac{1}{\sqrt{s}}\sum_{i\in S}(X_i-\E_{\beta,\bQ,\bmu_S(A)})$ is tight since $\mathrm{Var}_{\beta,\bQ,\bmu_S(A)}(\sum_{i\in S}X_i)\leq \mathrm{Var}_{\beta,\bQ,\bzero}(\sum_{i\in S}X_i)=O(s)$ by G.H.S inequality. Also, by arguments similar to the control of Type II error in the proof of  Theorem \ref{theorem:cw_known_beta}\ref{thm:cw_known_beta_ub}\ref{thm:cw_known_beta_ub_high_temp} and the fact that \\$\limsup_{n \rightarrow \infty}\sqrt{s}\tanh(A) (\log|\tilde{\C}_s|)^{-1/2}<\sqrt{2}$, we have that for any $\delta>0$  that $\tilde{t}_n(\delta)-\frac{1}{\sqrt{s}}\E_{\beta,\bQ,\bmu_S(A)}(\sum_{i\in S}X_i)\rightarrow \infty$ uniformly in $S\in \tilde{\C}_s$. Therefore by Chebyshev's Inequality we conclude that $1-\E_{\beta,\bQ,\bzero}(\tilde{L}_{\pi}(\delta))\rightarrow 0$ for any $\delta>0$. 

Finally we shall make a choice of $\delta>0$ while verifying \eqref{claim:var_trunctaed_lr_cw_high_temp}. First note that since $\tilde{\C}_s$ consists of disjoint sets we have
\begin{align}
   \ &   \E_{\beta,\bQ,\bzero}\left(\tilde{L}_{\pi}^2(\delta)\right) \nonumber
    \\
    &=\frac{1}{|\tilde{\C}_s|^2}\sum_{S\in \tilde{\C}_s}\frac{Z^2(\beta,\bQ,\bzero)Z(\beta,\bQ,\bmu_{S}(2A))}{Z^2(\beta,\bQ,\bmu_S(A))Z(\beta,\bQ,\bzero)}\P_{\beta,\bQ,\bmu_{S}(2A)}( \Omega_n(S,\delta))\nonumber\\
    &+\frac{1}{|\tilde{\C}_s|^2}\sum_{S_1\neq S_2\in \tilde{\C}_s}\frac{Z^2(\beta,\bQ,\bzero)Z(\beta,\bQ,\bmu_{S_1\cup S_2}(A))}{Z(\beta,\bQ,\bmu_{S_1}(A))Z(\beta,\bQ,\bmu_{S_2}(A))Z(\beta,\bQ,\bzero)}\P_{\beta,\bQ,\bmu_{S_1\cup S_2}(A)}(\Omega_n(S_1,\delta)\cap \Omega_n(S_2,\delta))\nonumber \\
    &=T_1+T_2 \quad \text{(say),} \label{eq:second_moment_split}
\end{align}
We will first show that $T_2\leq 1+o(1)$. To see this note that
\begin{align*}
    T_2\leq \frac{1}{|\tilde{\C}_s|^2}\sum_{S_1\neq S_2\in \tilde{\C}_s}\frac{Z^2(\beta,\bQ,\bzero)Z(\beta,\bQ,\bmu_{S_1\cup S_2}(A))}{Z(\beta,\bQ,\bmu_{S_1}(A))Z(\beta,\bQ,\bmu_{S_2}(A))Z(\beta,\bQ,\bzero)}=1+o(1),
\end{align*}
by the proof for the control of the second term in equation (9) of \cite[Theorem 3]{deb2020detecting} along with the correlation bounds presented in \cite[Lemma 9(a)]{deb2020detecting}. By symmetry, the value of $T_1$ equals
\begin{equation}\label{eq:t1}
	T_1 = \frac{1}{|\tilde{\C}_s|}\frac{Z^2_n(\beta,0)}{Z^2_n(\beta,s,A)}\frac{Z^2(\beta,\bQ,\bzero)Z(\beta,\bQ,\bmu_{S}(2A))}{Z^2(\beta,\bQ,\bmu_S(A))Z(\beta,\bQ,\bzero)}\P_{\beta,\bQ,\bmu_{S}(2A)}( \Omega_n(S,\delta)).
\end{equation}
Let $\varepsilon>0$ small enough such that $\sqrt{s}\tanh(A) (\log|\tilde{\mathcal{C}}_s|)^{-1/2}= 2(1-\varepsilon)$. Set $\lambda >0$ such that $\tanh(2A- \lambda/\sqrt{s})=\tilde{t}_n(\delta)/ \sqrt{s}$. Note, $\cosh(2A) \leq (1-\frac{(t_n(\delta)+\lambda)^2}{s})^{-1/2}$. Therefore,
\begin{align}
	\P_{\beta,\bQ,\bmu_S(2A)}\Big(Z_S \leq \tilde{t}_n(\delta)\Big) & \leq e^{\lambda \tilde{t}_n(\delta)} \E_{\beta,\bQ,\bmu_S(2A)}\Big(\exp{(-\frac{\lambda}{\sqrt{s}}\sum\limits_{i \in S}X_i)}\Big) \nonumber \\
	& = (1+o(1))  e^{\lambda \tilde{t}_n(\delta)} \Big(\frac{\cosh(2A-\frac{\lambda}{\sqrt{s}})}{\cosh(2A)}\Big)^s \nonumber \\
	&\le (1+o(1))e^{\lambda \tilde{t}_n(\delta)} e^{\frac{\tilde{t}^2_n(\delta)}{2}}\frac{1}{\cosh^{s}(2A)} \nonumber\\
	& \le (1+o(1)) \exp\left(\lambda \tilde{t}_n(\delta)+\frac{\tilde{t}^2_n(\delta)}{2}-\frac{(\lambda+t_n(\delta))^2}{2}\right)\\
	&=(1+o(1))\exp(-\frac{\lambda^2}{2}).
\end{align}
where the first equality is by Lemma \ref{lem:htrat}. Choose $\eta>0$ small enough such that,
$$2\sqrt{\frac{1-\varepsilon}{1+\delta}}(1-\eta)\frac{\tilde{t}_n(\delta)}{\sqrt{s}}\leq 2\sqrt{(1-\varepsilon)}(1-\eta) \tanh(A) \leq\tanh(2A)\leq \frac{\lambda+\tilde{t}_n(\delta)}{\sqrt{s}}.$$
Hence,
\begin{equation}
    \P_{\beta,\bQ,\bmu_S(2A)}\left(Z_S \leq \tilde{t}_n(\delta)\Big) \leq (1+o(1))\exp \Big(-(2\sqrt{\frac{1-\varepsilon}{1+\delta}}(1-\eta)-1)^2\frac{\tilde{t}^2_n(\delta)}{2}\right).
\end{equation}
By \eqref{eq:t1},
\begin{align}\label{tn2}
	T_1 &= (1+o(1))\exp \Big(s\{\log \cosh(2A)-2\log\cosh(A)\}-\log|\tilde{\C}_s|-(2\sqrt{\frac{1-\varepsilon}{1+\delta}}(1-\eta)-1)^2 \frac{\tilde{t}^2_n(\delta)}{2}\Big)\nonumber \\
	& \leq (1+o(1))\exp \Big(sA^2-\log|\tilde{\C}_s|-(2\sqrt{\frac{1-\varepsilon}{1+\delta}}(1-\eta)-1)^2 \frac{\tilde{t}^2_n(\delta)}{2}\Big) =o(1),
\end{align}
for small enough $\eta>0$. This completes the proof of \eqref{claim:var_trunctaed_lr_cw_high_temp}, finishing the proof of lower bound.

\paragraph{\bf Proof of Theorem \ref{theorem:cw_known_beta}\ref{thm:cw_known_beta_lb}\ref{thm:cw_known_beta_lb_critical_temp} }

The proof is verbatim same as that for $\beta<1$ in part 
Theorem \ref{theorem:cw_known_beta}\ref{thm:cw_known_beta_lb}\ref{thm:cw_known_beta_lb_high_temp} and the only change comes from applying Lemma \ref{lemma:cw_moderate_deviation} which is applicable now for $s\ll \sqrt{n}/\log{n}$ -- a condition that is part of the theorem's assumption.

\paragraph{\bf Proof of Theorem \ref{theorem:cw_known_beta}\ref{thm:cw_known_beta_lb}\ref{thm:cw_known_beta_lb_low_temp} }
To prove the lower bound, fix $\varepsilon>0$ such that $\sqrt{s}\tanh(A) (\log|\tilde{\mathcal{C}}_s|)^{-1/2}= 2(1-\varepsilon)(1-m^2)^{-1}$. Defining
\begin{equation}\label{eq:ltlr}
 L_{\pi}=\frac{1}{|\tilde{C}_s|}\sum_{S\in \tilde{C}_s}\frac{Z(\beta,\bQ,\bmu_S(A))}{Z(\beta,\bQ,\bzero)}\exp\left(A\sum_{i\in S}X_i\right),
\end{equation}
it suffices to show $L_{\pi} \rightarrow 1$ in probability. Defining $\Omega_n(S,\delta)=\{(A_{n,S} \cup B_{n,S})^c\}$ -- where for any $\delta>0$ we define $\tilde{t}_n(\delta)=\sqrt{2(1+\delta)\log|\tilde{C}_s|}$. Subsequently, we let 
\begin{align}
    \tilde{L}_{\pi}(\delta):=\frac{1}{|\tilde{C}_s|}\sum_{S\in \tilde{C}_s}\frac{Z(\beta,\bQ,\bmu_S(A))}{Z(\beta,\bQ,\bzero)}\exp\left(A\sum_{i\in S}X_i\right)\mathbf{1}(\bX\in \Omega_n(S,\delta))\label{eqn:trunctaed_lr_cw_low_temp},
\end{align}
and it is enough to show \eqref{claim:trunctaed_untruncated_lr_cw_high_temp},\eqref{claim:expec_trunctaed_lr_cw_high_temp},\eqref{claim:var_trunctaed_lr_cw_high_temp}. The proof of type I error in Theorem \ref{theorem:cw_known_beta}\ref{thm:cw_known_beta_ub}\ref{thm:cw_known_beta_ub_low_temp} yields $\P(L_{\pi}\neq \tilde{L}_\pi) \rightarrow 0$. The proof of \eqref{claim:expec_trunctaed_lr_cw_high_temp} follows similar to the Type II error of in Theorem \ref{theorem:cw_known_beta}\ref{thm:cw_known_beta_ub}\ref{thm:cw_known_beta_ub_low_temp}. To prove \eqref{claim:var_trunctaed_lr_cw_high_temp}, recall that
\begin{align}\label{eqn:cw_lt_lr_2nd}
   \ &   \E_{\beta,\bQ,\bzero}\left(\tilde{L}_{\pi}^2(\delta)\right)\\
    &=\frac{1}{|\tilde{\C}_s|^2}\sum_{S\in \tilde{\C}_s}\frac{Z^2(\beta,\bQ,\bzero)Z(\beta,\bQ,\bmu_{S}(2A))}{Z^2(\beta,\bQ,\bmu_S(A))Z(\beta,\bQ,\bzero)}\P_{\beta,\bQ,\bmu_{S}(2A)}( \Omega_n(S,\delta))\nonumber\\
    &+\frac{1}{|\tilde{\C}_s|^2}\sum_{S_1\neq S_2\in \tilde{\C}_s}\frac{Z^2(\beta,\bQ,\bzero)Z(\beta,\bQ,\bmu_{S_1\cup S_2}(A))}{Z(\beta,\bQ,\bmu_{S_1}(A))Z(\beta,\bQ,\bmu_{S_2}(A))Z(\beta,\bQ,\bzero)}\P_{\beta,\bQ,\bmu_{S_1\cup S_2}(A)}(\Omega_n(S_1,\delta)\cap \Omega_n(S_2,\delta))\nonumber\\
    &=T_1+T_2,
\end{align}
Here, $T_2 \leq 2+o(1)$, because
$$\frac{Z^2(\beta,\bQ,\bzero)Z(\beta,\bQ,\bmu_{S_1\cup S_2}(A))}{Z(\beta,\bQ,\bmu_{S_1}(A))Z(\beta,\bQ,\bmu_{S_2}(A))Z(\beta,\bQ,\bzero)}=(2+o(1))\frac{\cosh^{2s}(\beta m+A)+\cosh^{2s}(\beta m-A)}{(\cosh^{s}(\beta m+A)+\cosh^{s}(\beta m-A))^2} =2+o(1),$$
where we have used $\frac{\cosh^{s}(\beta m-A)}{\cosh^{s}(\beta m+A)}=o(1)$. To bound $T_1$, choose $\eta>0$ small to be specified later. Set $\lambda= \frac{\tilde{t}_n(\delta)\sqrt{s}(1-\eta)}{1-m^2}$.
\begin{align*}
\ & 	\P_{\beta,\bQ,\bmu_S(2A)}(Z_S\leq \sqrt{s}+\tilde{t}_n(\delta)|W_n>0)\\
& \leq \exp\Big({\lambda(m+\tilde{t}_n(\delta)/\sqrt{s})+s [\log\cosh(\beta m+2A-\frac{\lambda}{s})-\log\cosh(\beta m+2A)]}\Big)\\
	&\leq \exp\Big(\lambda(m-\tanh(\beta m+2A)+\frac{\tilde{t}_n(\delta)}{\sqrt{s}})+\frac{\lambda^2}{2s}\sech^2(\beta m)+o(\log|\tilde{C}_s|)\Big)\\
	&\leq \exp\Big(-\frac{\lambda \tilde{t}_n(\delta)}{\sqrt{s}}(1-\eta)+\frac{\lambda^2}{2s}\sech^2(\beta m)+o(\log|\tilde{C}_s|)\Big)\\
	&= \exp\Big(-(1+\varepsilon)(1-\eta)^2 \log|\tilde{C}_s|+o(\log|\tilde{C}_s|)\Big),
\end{align*}
where the third inequality is by $\tanh(\beta m+2A)\geq m+2A\sech^2(\beta m)$ and small $\eta>0$. Similar bound holds for $W_n<0$. Therefore,
\begin{align*}
	T_1 &\leq C \exp\Big(s[\log\cosh(\beta m+2A)-2\log\cosh(\beta m+A)+\log\cosh(\beta m)] -\log|\tilde{C}_s| -(1+\varepsilon)(1-\eta)^2 \log|\tilde{C}_s|\Big)\\
	&\leq C \exp\Big(sA^2 -\log|\tilde{C}_s| -(1+\varepsilon)(1-\eta)^2 \log|\tilde{C}_s|\Big)=o(1),
\end{align*}
by small $\varepsilon,\eta>0$ and since $\tanh(A)\sim A$. Therefore, $\E_{\beta,\bQ,\bzero}(\tilde{L}^2_\pi) \leq 2+o(1)$. Hence,
$$\liminf_{n \rightarrow \infty} \inf_T\mathrm{Risk}(T,{\Xi}(\mathcal{C}_s,A),\beta,\bQ)\geq 1- \limsup_{n \rightarrow \infty}\frac{1}{2}\sqrt{\E_{\beta,\bQ,\bzero}(L^2_\pi)-1}>0.$$
Consequently, no test is asymptotically powerful.

\subsubsection{\bf Proof of Theorem \ref{thm:cw_unknown_beta}}
To obtain an adaptive test, we first test the hypothesis $H_{0,\beta}: \beta\leq 1$ vs $H_{1,\beta}:\beta>1$. This is obtained through the test $$T_n^{(1)}=\mathbbm{1}\left(|\bar{\bX}|\geq  \frac{1}{\log{n}} \right).$$ 
Subsequently, if $T_n^{(1)}=0$ we simply use the test $$T_n^{(2)}=\mathbbm{1}\left(Z_{\max}>\sqrt{2(1+\delta)\log{|\mathcal{N}(\C_s,\gamma,\varepsilon_n)|}}\right)$$ for $\delta>0$ as in the proof of  Theorem \ref{theorem:cw_known_beta}\ref{thm:cw_known_beta_ub}\ref{thm:cw_known_beta_ub_high_temp}. 
In contrast, if $T_n^{(1)}=1$ we estimate $\beta$ assuming the working model ``$\bX\stackrel{\text{working}}{\sim}\P_{\beta,\bQ,\mathbf{0}}$" using the Pseudo-likelihood method \citep{besag1974spatial,chatterjee2007estimation,bhattacharya2018inference,ghosal2020joint} and with $\hat{\beta}_{\rm working}$ denote this estimator we consider rejecting $H_0$ using 
\begin{align*}
    T_n^{(3)}&=\mathbbm{1}(Z_{\max}>m(\hat{\beta}_{\rm working})+\sqrt{2(1+\delta)(1-m^2(\hat{\beta}_{\rm working}))\log|\mathcal{N}(\C_s,\gamma,\varepsilon_n)|},W_n\geq 0)\\&+\mathbbm{1}(Z_{\max}>-m(\hat{\beta}_{\rm working})+\sqrt{2(1+\delta)(1-m^2(\hat{\beta}_{\rm working}))\log|\mathcal{N}(\C_s,\gamma,\varepsilon_n)|},W_n< 0)
\end{align*}
where $\delta>0$ can be chosen as in the proof of Theorem Theorem \ref{theorem:cw_known_beta}\ref{thm:cw_known_beta_ub}\ref{thm:cw_known_beta_ub_low_temp} and $m(\beta)$ is the unique positive root of $m=\tanh(\beta m)$. Our final $\beta$-adaptive test is thereafter given by $T_n=(1-T_n^{(1)})T_n^{(2)}
+T_n^{(1)}T_n^{(3)}$. 
We note that in this part of the proof we do not perform the Bonferroni correction of the test $T_n^{(3)}$ with a test based on conditionally centered version of $\sum_{i=1}^n X_i$ since here $\sum_{i=1}^n \mu_i\ll s\ll n$. Therefore as it is clear from the proof of Theorem \ref{theorem:cw_known_beta}\ref{thm:cw_known_beta_ub}\ref{thm:cw_known_beta_ub_low_temp}, it suffices to only consider the randomized scan test described through $T_n^{(3)}$.

We first show that uniformly over all $\bmu \in \{\mathbf{0}\}\cup \Xi(\C_s,A)$ one has that
$|\bar{\bX}|=o_{\P_{\beta,\bQ,\bmu}}(\frac{1}{\log{n}})$ for $\beta\leq 1$ and $|\bar{\bX}|\gg \frac{1}{\log{n}}$  with probability converging to $1$ for $\beta> 1$. For the first claim, first let $\beta \leq 1$. Then from Lemma \ref{lem:bayesian} we have that $|\bar{\bX}|=o_{\P_{\beta,\bQ,\bzero}}(\frac{1}{\log{n}})$ for $\beta\leq 1$. Now note that by Lemma \ref{lemma:expectation_underalt_lower_bound} we have that there exists constants $C_1,C_2>0$ such that
\begin{align*}
    \E_{\beta,\bQ,\bmu}\left(\sum_{i=1}^n X_i\right)\leq C_1\sum_{i=1}^n \tanh(\mu_i)+ C_2\sum_{i=1}^n\sum_{j=1}^n\tanh(\mu_j)\left(\bQ_{i,j}+\mathrm{Cov}_{\beta,\bQ,\bmu}(X_i,X_j)\right)
\end{align*}

Now suppose $\beta<1$. Then we have from \cite{deb2020detecting} that $0\leq \mathrm{Cov}_{\beta,\bQ,\bmu}(X_i,X_j)\lesssim 1/n$. Hence for $\beta<1$ we have 
\begin{align*}
    \E_{\beta,\bQ,\bmu}\left(\sum_{i=1}^n X_i\right)\lesssim \sum_{i=1}^n \tanh(\mu_i) \sum_{i=1}^n\sum_{j=1}^n\tanh(\mu_j)\left(\bQ_{i,j}+\mathrm{Cov}_{\beta,\bQ,\bmu}(X_i,X_j)\right)\lesssim \sum_{j=1}^n \tanh(\mu_j)\lesssim s
\end{align*}
since $\|\bmu\|_0\leq s$. Hence under any $\bmu$ and $\beta<1$ one has
\begin{align*}
   \ &  \P_{\beta,\bQ,\bmu}(T_n^{(1)}=1)\\
   &=\P_{\beta,\bQ,\bmu}\left(\sum_{i=1}^n X_i>\frac{n}{\log{n}}\right)+\P_{\beta,\bQ,\bmu}\left(\sum_{i=1}^n X_i<-\frac{n}{\log{n}}\right)\\
    &\leq \P_{\beta,\bQ,\bmu}\left(\sum_{i=1}^n X_i>\frac{n}{\log{n}}\right)+\P_{\beta,\bQ,\bzero}\left(\sum_{i=1}^n X_i<-\frac{n}{\log{n}}\right) \quad \text{by monotonicity of}\ \sum_{i=1}^n X_i\\
    &=\P_{\beta,\bQ,\bmu}\left(\frac{\sum_{i=1}^n X_i-\E_{\beta,\bQ,\bmu}(\sum_{i=1}^n X_i)}{\sqrt{\mathrm{Var}_{\beta,\bQ,\bmu}(\sum_{i=1}^n X_i)}}>\frac{n}{\log{n}\sqrt{\mathrm{Var}_{\beta,\bQ,\bmu}(\sum_{i=1}^n X_i)}}-\frac{\E_{\beta,\bQ,\bmu}(\sum_{i=1}^n X_i)}{\sqrt{\mathrm{Var}_{\beta,\bQ,\bmu}(\sum_{i=1}^n X_i)}}\right)+o(1)\\
    &\leq \P_{\beta,\bQ,\bmu}\left(\frac{\sum_{i=1}^n X_i-\E_{\beta,\bQ,\bmu}(\sum_{i=1}^n X_i)}{\sqrt{\mathrm{Var}_{\beta,\bQ,\bmu}(\sum_{i=1}^n X_i)}}>C_1\frac{\sqrt{n}}{\log{n}}-C_2\frac{s}{\sqrt{n}}\right)+o(1)
\end{align*}
where in the last line we have used the fact that by GHS inequality (Lemma \ref{lemma:GHS}) and Lemma \ref{lem:bayesian}  we have that $n\lesssim \mathrm{Var}_{\beta,\bQ,\bmu}(\sum_{i=1}^n X_i)\leq \mathrm{Var}_{\beta,\bQ,\bzero}(\sum_{i=1}^n X_i)\lesssim n$ for $\beta< 1$. Now note that by our assumptions $s/\sqrt{n}\ll \sqrt{n}/\log{n}$ since $s\ll n/\log{n}$ and hence the first term of the display above goes to $0$ uniformly in $\bmu$ for $\beta<1$.

Turning to $\beta=1$, we have from \cite{deb2020detecting} that $0\leq \mathrm{Cov}_{\beta,\bQ,\bmu}(X_i,X_j)\lesssim 1/\sqrt{n}$. Hence for $\beta=1$ we have 
\begin{align*}
    \E_{\beta,\bQ,\bmu}(\sum_{i=1}^n X_i)\lesssim \sum_{i=1}^n \tanh(\mu_i) \sum_{i=1}^n\sum_{j=1}^n\tanh(\mu_j)\left(\bQ_{i,j}+\mathrm{Cov}_{\beta,\bQ,\bmu}(X_i,X_j)\right)\lesssim \sum_{j=1}^n \tanh(\mu_j)\lesssim s\sqrt{n}
\end{align*}
since $\|\bmu\|_0\leq s$. Hence under any $\bmu$ and $\beta=1$ one has
\begin{align*}
   \ &  \P_{\beta,\bQ,\bmu}(T_n^{(1)}=1)\\
   &=\P_{\beta,\bQ,\bmu}\left(\sum_{i=1}^n X_i>\frac{n}{\log{n}}\right)+\P_{\beta,\bQ,\bmu}\left(\sum_{i=1}^n X_i<-\frac{n}{\log{n}}\right)\\
    &\leq \P_{\beta,\bQ,\bmu}\left(\sum_{i=1}^n X_i>\frac{n}{\log{n}}\right)+\P_{\beta,\bQ,\bzero}\left(\sum_{i=1}^n X_i<-\frac{n}{\log{n}}\right) \quad \text{by monotonicity of}\ \sum_{i=1}^n X_i\\
    &=\P_{\beta,\bQ,\bmu}\left(\frac{\sum_{i=1}^n X_i-\E_{\beta,\bQ,\bmu}(\sum_{i=1}^n X_i)}{\sqrt{\mathrm{Var}_{\beta,\bQ,\bmu}(\sum_{i=1}^n X_i)}}>\frac{n}{\log{n}\sqrt{\mathrm{Var}_{\beta,\bQ,\bmu}(\sum_{i=1}^n X_i)}}-\frac{\E_{\beta,\bQ,\bmu}(\sum_{i=1}^n X_i)}{\sqrt{\mathrm{Var}_{\beta,\bQ,\bmu}(\sum_{i=1}^n X_i)}}\right)+o(1)\\
    &\leq \P_{\beta,\bQ,\bmu}\left(\frac{\sum_{i=1}^n X_i-\E_{\beta,\bQ,\bmu}(\sum_{i=1}^n X_i)}{\sqrt{\mathrm{Var}_{\beta,\bQ,\bmu}(\sum_{i=1}^n X_i)}}>C_1\frac{n^{1/4}}{\log{n}}-C_2\frac{s}{n^{1/4}}\right)+o(1)
\end{align*}
where in the last line we have used the fact that by GHS inequality (Lemma \ref{lemma:GHS}) and Lemma \ref{lem:bayesian}  we have that $n^{3/4}\lesssim \mathrm{Var}_{\beta,\bQ,\bmu}(\sum_{i=1}^n X_i)\leq \mathrm{Var}_{\beta,\bQ,\bzero}(\sum_{i=1}^n X_i)\lesssim n^{3/4}$ for $\beta< 1$. Now note that by our assumptions $s/n^{1/4}\ll n^{1/4}/\log{n}$ since $s\ll \sqrt{n}/\log{n}$ and hence the first term of the display above goes to $0$ uniformly in $\bmu$ for $\beta=1$ as well. This shows that for any $\bmu$ one has that $T_n^{(1)}$ has Type I error converging to $0$ for testing $H_{0,\beta}$ vs $H_{1,\beta}$.

Next we show that uniformly over all $\bmu \in \{\mathbf{0}\}\cup \Xi(\C_s,A)$  $|\bar{\bX}|\gg \frac{1}{\log{n}}$  with probability converging to $1$ for $\beta> 1$. The claim is trivial for $\bmu=\bzero$ by \cite{Ellis_Newman}. In particular, it also follows that $\P_{\beta,\bQ,\bzero}(T_n^{(1)}=0)\leq e^{-Cn}$ for some $C>0$ depending on $\beta>1$. Then for any $\bmu$ with $\|\bmu\|_{\infty}\leq M$ for some $M>0$ one has 
\begin{align*}
   \sup_{\bmu:\|\bmu\|_{\infty}\leq M} \P_{\beta,\bQ,\bzero}(T_n^{(1)}=0)&=\sup_{\bmu:\|\bmu\|_{\infty}\leq M}\frac{Z(\beta,\bQ,\bzero)}{Z(\beta,\bQ,\bmu)}\E_{\beta,\bQ,\bzero}\left(\mathbbm{1}(T_n^{(1)}=0)e^{\sum_{i=1}^n \mu_iX_i}\right)\\
    &=\sup_{\bmu:\|\bmu\|_{\infty}\leq M}e^{2s\|\bmu\|_{\infty}}\P_{\beta,\bQ,\bzero}(T_n^{(1)}=0)\leq e^{2sM-Cn}\rightarrow 0
\end{align*}
since $s\ll n$. This completes the proof for $\beta>1$ for the consistency of testing $H_{0,\beta}$ vs $H_{1,\beta}$ using $T_n^{(1)}$ under any $\bmu\in \mathbb{R}_+^n$ s.t. $\|\bmu\|_{0}\leq s\ll n$ and $\|\bmu\|_{\infty}=O(1)$. We next turn to the consistency of the test $T_n=(1-T_n^{(1)})T_n^{(2)}
+T_n^{(1)}T_n^{(3)}$ for testing the actual hypotheses \eqref{eqn:sparse_hypo} of interest.

First consider the case $\beta\leq 1$.
Let us first consider the Type I error of the test $T_n$.
\begin{align*}
    \P_{\beta,\bQ,\bzero}(T_n=1)= \P_{\beta,\bQ,\bzero}(T_n^{(1)}=0, T_n^{(2)}=1)+\P_{\beta,\bQ,\bzero}(T_n^{(1)}=1, T_n^{(3)}=1).
\end{align*}
By consistency of the test $T-n^{(1)}$ for testing $H_{0,\beta}$ vs $H_{1,\beta}$ one has that for $\beta\leq 1$ one has $\P_{\beta,\bQ,\bzero}(T_n^{(1)}=1)\rightarrow 0$. Hence  $\P_{\beta,\bQ,\bzero}(T_n^{(1)}=1, T_n^{(3)}=1)\rightarrow 0$ as well. Hence it is enough to show that $\P_{\beta,\bQ,\bzero}( T_n^{(2)}=1)\rightarrow 0$ under $\beta\leq 1$. However this follows by arguments verbatim to the proof of Type I errors in Theorem \ref{theorem:cw_known_beta}\ref{thm:cw_known_beta_ub}\ref{thm:cw_known_beta_ub_high_temp} and Theorem \ref{theorem:cw_known_beta}\ref{thm:cw_known_beta_ub}\ref{thm:cw_known_beta_ub_critical_temp} since the test $T_n^{(2)}$ is free of $\beta$.  This completes the control of Type I error for $\beta\leq 1$. Turning to the Type II error, once again note that for any $\bmu \in \C_s$
\begin{align*}
     \P_{\beta,\bQ,\bmu}(T_n=0)=\P_{\beta,\bQ,\bzero}(T_n^{(1)}=0, T_n^{(2)}=0)+\P_{\beta,\bQ,\bzero}(T_n^{(1)}=1, T_n^{(3)}=0).
\end{align*}
Once again by consistency of the test $T-n^{(1)}$ for testing $H_{0,\beta}$ vs $H_{1,\beta}$ (uniformly against any $\bmu$ in our class having $\|\bmu\|_{\infty}=O(1)$) one has that for $\beta\leq 1$ one has $\P_{\beta,\bQ,\bmu}(T_n^{(1)}=1)\rightarrow 0$ uniformly in such $\bmu$'s. Hence  $\P_{\beta,\bQ,\bzero}(T_n^{(1)}=1, T_n^{(3)}=1)\rightarrow 0$ uniformly as well and is enough to show that $\P_{\beta,\bQ,\bmu}( T_n^{(2)}=0)\rightarrow 0$ uniformly in $\bmu$ under $\beta\leq 1$. This once again follows by arguments verbatim to the proof of Type II errors in Theorem \ref{theorem:cw_known_beta}\ref{thm:cw_known_beta_ub}\ref{thm:cw_known_beta_ub_high_temp} and Theorem \ref{theorem:cw_known_beta}\ref{thm:cw_known_beta_ub}\ref{thm:cw_known_beta_ub_critical_temp} since the test $T_n^{(2)}$ is free of $\beta$. 

We next turn to the final case of analyzing the test $T_n$ under $\beta>1$. Let us first consider the Type I error of the test $T_n$.
\begin{align*}
    \P_{\beta,\bQ,\bzero}(T_n=1)= \P_{\beta,\bQ,\bzero}(T_n^{(1)}=0, T_n^{(2)}=1)+\P_{\beta,\bQ,\bzero}(T_n^{(1)}=1, T_n^{(3)}=1).
\end{align*}
By consistency of the test $T_n^{(1)}$ for testing $H_{0,\beta}$ vs $H_{1,\beta}$ one has that for $\beta> 1$ one has $\P_{\beta,\bQ,\bzero}(T_n^{(1)}=0)\rightarrow 0$. Hence  $\P_{\beta,\bQ,\bzero}(T_n^{(1)}=0, T_n^{(2)}=1)\rightarrow 0$ as well. Hence it is enough to show that $\P_{\beta,\bQ,\bzero}( T_n^{(3)}=1)\rightarrow 0$ under $\beta> 1$. Now note that the only dependence of the test $T_n^{(3)}$ on $\hat{\beta}_{\rm working}$ is in its cutoff and the proof of the Type I error  in Theorem \ref{theorem:cw_known_beta}\ref{thm:cw_known_beta_ub}\ref{thm:cw_known_beta_ub_low_temp} shows that it is enough to show that $m(\hat{\beta}_{\rm working})\stackrel{\P_{\beta,\bQ,\bzero}}{\rightarrow} m(\beta)$ which follows from Lemma \ref{lmm:beta_estimation}. In particular, Lemma \ref{lmm:beta_estimation} is applicable since \begin{equation}
	\liminf_{n \rightarrow \infty} \frac{1}{n}\log \frac{1}{2^n}Z(\beta,\bQ^{\rm CW},\bzero)>0,
\end{equation}
using \cite[(7.9)]{infising}. Hence, by Lemma \ref{lmm:beta_estimation}, is applicable. Turning to the Type II error of our test $T_n$ once again by consistency of the test $T_n^{(1)}$ for testing $H_{0,\beta}$ vs $H_{1,\beta}$ (uniformly against any $\bmu$ in our class having $\|\bmu\|_{\infty}=O(1)$) one has that for $\beta>1$ one has $\P_{\beta,\bQ,\bmu}(T_n^{(1)}=0)\rightarrow 0$ uniformly in such $\bmu$'s. Hence  $\P_{\beta,\bQ,\bmu}(T_n^{(1)}=0, T_n^{(2)}=0)\rightarrow 0$ uniformly as well and is enough to show that $\P_{\beta,\bQ,\bmu}( T_n^{(2)}=0)\rightarrow 0$ uniformly in $\bmu$ under $\beta> 1$. Once again from the proof of the Type II error  in Theorem \ref{theorem:cw_known_beta}\ref{thm:cw_known_beta_ub}\ref{thm:cw_known_beta_ub_low_temp} it is enough to show that $m(\hat{\beta}_{\rm working})\stackrel{\P_{\beta,\bQ,\bmu}}{\rightarrow} m(\beta)$ uniformly in $\bmu$ in our class which follows from Lemma \ref{lmm:beta_estimation}.


\subsection{\bf Technical Lemmas for Proofs of Theorems in Section \ref{sec:curie_weiss}}\label{sec:technical_lemmas_mean_field}

\begin{lemma}\label{lemma:cw_moderate_deviation}
Suppose $\bX\sim \P_{\beta,\bQ^{\rm CW},\bmu}$ with $\bQ^{\rm CW}_{i,j}=\mathbf{1}(i\neq j)/n$ and let $Z_S=\sum_{i\in S}X_i/\sqrt{s}$. Define $t_n(\delta):=\sqrt{2(1+\delta)(1-m^2)\log{|\mathcal{N}(\C_s,\gamma,\varepsilon_n)|}}$ for some $\delta>0$, where $m:=m(\beta)$ is the unique positive root of $m=\tanh(\beta m)$.
\begin{enumerate}
    \item[(a)] If $\beta<1, s\ll \frac{n}{\log n}$, $\P_{\beta,\bQ^{\rm CW},\mathbf{0}}(Z_S> t_n(\delta)) \le (1+o(1))e^{-t^2_n(\delta)}$.
    \item[(b)]
    If $\beta=1, s\ll \frac{\sqrt{n}}{\log n}$, then same conclusion as (a) holds.
    \item[(c)]
    If $\beta>1, s\ll \frac{n}{\log n}$,
    \begin{equation}
        	P_{\beta,\bQ^{\rm CW},\mathbf{0}}(Z_S> m \sqrt{s}+t_n(\delta)) \leq  (1+o(1))\exp\left(-\frac{t^2_n(\delta)}{2(1-m^2)}\right)
    \end{equation}
    Further,
    \begin{equation}
        P_{\beta,\bQ^{\rm CW},\mathbf{0}}(Z_S> -m \sqrt{s}+t_n(\delta), W_n<0) \le (1+o(1))\exp\left(-(1+o(1))\frac{t^2_n(\delta)}{2(1-m^2)}\right)
    \end{equation}
\end{enumerate}

\end{lemma}
\begin{proof}
\begin{enumerate}
    \item[(a)]
    Pick $\lambda>0$ such that $\tanh\left(\frac{\lambda}{\sqrt{s}}\right)=\frac{t_n(\delta)}{\sqrt{s}}$. So, $\lambda \geq t_n(\delta)$, and 
$\cosh\left(\frac{\lambda}{\sqrt{s}}\right)=\left(1-\frac{t^2_n(\delta)}{s}\right)^{-1/2}.$ Also, $s \tanh\left(\frac{\lambda}{\sqrt{s}}\right)= \sqrt{st_n(\delta)}=O(\sqrt{s \log |\mathcal{N}(\C_s,\gamma,\varepsilon_n)|}) \ll \sqrt{n}$. Hence, 
\begin{align*}
    \P_{\beta,\bQ^{\rm CW}\mathbf{0}}(Z_S> t_n(\delta) &\le e^{-\lambda t_n(\delta)}\frac{Z(\beta,\bQ^{\rm CW}, \bmu_S(\frac{\lambda}{\sqrt{s}}))}{Z(\beta,\bQ^{\rm CW},\bzero)}\\
    &= (1+o(1))e^{-\lambda t_n(\delta)} \cosh^s\left(\frac{\lambda}{\sqrt{s}}\right) \le (1+o(1))e^{-t^2_n(\delta)},
\end{align*}
where the equality is due to Lemma \ref{lemma:curie_part_fun_ratio}.
\item[(b)]
When $\beta=1$, we pick same $\lambda$. Now, $s \tanh\left(\frac{\lambda}{\sqrt{s}}\right) \ll n^{1/4}$ and the bound is obtained using Lemma \ref{lemma:curie_part_fun_ratio} again and the $o(1)$ term does not depend on the choice of $S$.
\item[(c)]
When $\beta>1$, pick $\lambda=\frac{t_n(\delta)\sqrt{s}}{1-m^2}$. Here, $s\tanh(\lambda/s)\ll \sqrt{n}$ and $\lambda/s\rightarrow 0$.
\begin{align*}
\ &    \P_{\beta,\bQ^{\rm CW},\mathbf{0}}(Z_S> m \sqrt{s}+t_n(\delta))\\
&\leq \exp\left(-\frac{\lambda t}{\sqrt{s}}-\lambda m\right)\E_{\beta,\bQ^{\rm CW},\mathbf{0}}\exp\left(\frac{\lambda}{s}\sum_{i\in S}X_i\right)\\
&=\exp\left(-\frac{\lambda t}{\sqrt{s}}-\lambda m\right)\frac{Z(\beta,\bQ^{\rm CW}, \bmu_S(\frac{\lambda}{s}))}{Z(\beta,\bQ^{\rm CW},\bzero)}\\
&\leq (1+o(1))\exp\left(-\frac{\lambda t}{\sqrt{s}}-\lambda m\right)\left(\frac{\cosh(\beta m+\lambda/s)}{\cosh(\beta m)}\right)^s\\
&=(1+o(1))\exp\left(-\frac{\lambda t}{\sqrt{s}}-\lambda m+s\log\cosh(\beta m+\lambda/s)-s\log\cosh(\beta m)\right)\\
&\leq (1+o(1))\exp\left(-\frac{\lambda t}{\sqrt{s}}\underbrace{-\lambda m(\beta)+\frac{s\lambda}{s}\tanh(\beta m(\beta))}_{=0}+\frac{s\lambda^2}{2s^2}\sech^2(\beta m(\beta))\right)\\
&=(1+o(1))\exp\left(-\frac{\lambda t}{\sqrt{s}}+\frac{\lambda^2}{2s}\sech^2(\beta m(\beta))\right)\\
&=(1+o(1))\exp\left(-\frac{t^2}{2(1-m^2)}\right),
\end{align*}
where the second inequality is due to Lemma \ref{lemma:curie_part_fun_ratio} and the third inequality occurs since the third term of the Taylor expansion is negative. \\

When $W_n<0$, note that,
\begin{align*}
   \ &  \P_{\beta,\bQ^{\rm CW},\mathbf{0}}(Z_S> -m \sqrt{s}+t_n(\delta), W_n<0) \\
   &\le \exp\left(-\frac{\lambda t}{\sqrt{s}}+\lambda m\right)\E_{\beta,\bQ^{\rm CW},\mathbf{0}}\exp\left(\frac{\lambda}{s}\sum_{i\in S}X_i\mathbbm{1}_{W_n \le 0}\right)\\
	&\leq (1+o(1))\exp\left(-\frac{\lambda t_n(\delta)}{\sqrt{s}}+\lambda m\right)\left(\frac{\cosh(\beta m-\lambda/s)}{\cosh(\beta m)}\right)^s
\end{align*}
Previously, the third term in Taylor expansion was $<0$, here we simply bound it by $O\left( \frac{t^3_n(\delta)}{s}\right)=o(t^2_n(\delta))$.
\end{enumerate}
\end{proof}

The following Lemma provides estimates for ratios of normalizing constants in the Curie-Weiss model.

\begin{lemma}\label{lemma:curie_part_fun_ratio}

	For $s\in [n]$ and $A,\beta>0$ let $Z_n(\beta,s,A)=Z(\beta,\bQ^{\rm CW},\bmu_S(A))$ be the partition function  of the Curie-Weiss model $\P_{\beta,\bQ^{\rm CW},\bmu_{S}(A)}$ with $\bQ^{\rm CW}_{i,j}=\mathbbm{1}(i\neq j)/n$ 
	Then the following conclusions hold for any $C>0$:
	
	\begin{enumerate}
	\item[(a)]
	If $\beta<1, s\ll \frac{n}{\log n}$ and $s\tanh(A)\ll \sqrt{n}$, then we have 
	\begin{align}\label{eq:partition<1}
	\frac{Z_n(\beta,s,A)}{Z_n(\beta,0,0)}=(1+o(1))\cosh(A)^s.
	\end{align}

	\item[(b)]
	If $\beta>1, s\ll \frac{n}{\log n}$ and  $s\tanh(A)\ll \sqrt{n}$ we have 
	$$\bE_{\beta,\bQ^{\rm CW},\bzero}\left(\exp(A\sum_{i\in S} X_i)|W_n>0\right)= (1+o(1))\frac{\cosh(\beta m+A)^s}{\cosh(\beta m)^s},$$
	$$\bE_{\beta,\bQ^{\rm CW},\bzero}\left(\exp(A\sum_{i\in S} X_i)|W_n<0\right)= (1+o(1))\frac{\cosh(\beta m-A)^s}{\cosh(\beta m)^s},$$
	where $W_n$ is the auxilliary variable for the Curie-Weiss model (cf. \cite[lemma 3]{mukherjee2016global}). Therefore,
	
	\begin{align}\label{eq:partition>1}
	 \frac{Z_n(\beta,s,A)}{Z_n(\beta,0,0)}= (1+o(1))\frac{1}{2}\left(\frac{\cosh(\beta m-A)^s}{\cosh(\beta m)^s}+\frac{\cosh(\beta m+A)^s}{\cosh(\beta m)^s}\right)
	\end{align}
	where $m$ is the unique positive root of the equation $m=\tanh(\beta m)$. 
	
%
	
	\item[(c)]
	If $\beta=1$ and  $s\ll \frac{n}{\log n}$, then for $s\tanh(A)\ll n^{1/4}$ we have \eqref{eq:partition<1} holds.
	
%
	
	\end{enumerate}
	\end{lemma}
	
	\begin{proof}
	We begin with the representation used in \cite[lemma 3]{mukherjee2016global}. 
 Define a random variable $W_n$ which given $\bX$ has a distribution $ N(\bar{\mathbf{X}},1/n)$. Then under $\P_{\beta,\bQ^{\rm CW},\mathbf{0}}$, given $W_n=w_n$, each $X_i$'s are i.i.d with 
	\begin{align*}
	\P_{\beta,\bQ^{\rm CW},\mathbf{0}}(X_i=x_i|W_n=w_n)=\frac{e^{\beta w_n x_i}}{e^{\beta w_n}+e^{-\beta w_n}}.
	\end{align*}
	Therefore, for any $\mu_i\in \mathbb{R}$ one has $\E_{\beta,\bQ^{\rm CW},\mathbf{0}}(\exp(\mu_iX_i)|W_n)=\frac{\cosh(\mu_i+\beta W_n)}{\cosh(\beta W_n)}$.
	Therefore for any $\bmu\in \mathbb{R}^n$
	\begin{align}
	\frac{Z_n(\beta,\bQ^{\rm CW},\bmu)}{Z_n(\beta,\bQ^{\rm CW},\mathbf{0})}&=\frac{\sum_{\bx\in \{\pm 1\}^n}\exp\left(\frac{\beta}{n}\sum_{i<j}x_ix_j+\sumn \mu_i x_i\right)}{\sum_{\bx\in \{\pm 1\}^n}\exp\left(\frac{\beta}{n}\sum_{i<j}x_ix_j\right)}=\E_{\beta,\bQ^{\rm CW},\mathbf{0}}\left(\exp(\sumn \mu_i X_i)\right)\nonumber\\
	&=\E_{\beta,\bQ^{\rm CW},\mathbf{0}}\left[\exp\left\{\sum_{i=1}^n\left(\log\cosh(\mu_i+\beta W_n)-\log\cosh(\beta W_n)\right) \right\}\right].\label{eqn:partition_function_ratio}
	\end{align} 
	Consequently we have
	\begin{align}\label{eq:partition}
	\frac{Z_n(\beta,s,A)}{Z_n(\beta,0,0)}=\frac{Z_n(\beta,\bQ^{\rm CW},\bmu_S(A))}{Z_n(\beta,\bQ^{\rm CW},\mathbf{0})}=\E_{\beta,\bQ^{\rm CW},\mathbf{0}} (e^{T_n}),\quad T_n:=s[\log \cosh(\beta W_n+A)-\log \cosh(\beta W_n)].
	\end{align}
	Also use \cite[Lemma 3]{mukherjee2016global} to note that $W_n$ marginally has a density proportional to $e^{-nf(w)}$, with $f(w):=\beta w^2/2-\log\cosh(\beta w)$. With this we are ready to prove the lemma.

	\begin{enumerate}
	\item[(a)]

	To begin note that $|\log\cosh(\beta W_n)|\le \frac{\beta^2}{2}W_n^2$. Also, a two term Taylor's series expansion in $W_n$ gives
	\begin{align*}
	\Big|\log \cos(\beta W_n+A)-\log\cosh(A)-\beta W_n\tanh(A)\Big|\le &\frac{\beta^2}{2} W_n^2.
	\end{align*}
	Combining these two and setting $B:=\tanh(A)$ we have
	$$|T_n-s\log \cosh(A)|\le \beta|W_n|sB+\beta^2sW_n^2,$$
	the RHS of which converges to $0$ in probability, as $\sqrt{n}W_n=O_{\P_{\beta,\bQ^{\rm CW},\mathbf{0}}}(1)$ by part (a) of Lemma \ref{lem:bayesian}.
	Thus to prove \eqref{eq:partition<1}, using uniform integrability it suffices to show that  
	\begin{align}\label{eq:theta<1}
	\limsup_{n\rightarrow\infty}\E_{\beta,\bQ^{\rm CW},\mathbf{0}} e^{2T_n-2s\log\cosh(A)}<\infty.
	\end{align}
	To this effect, again using the above Taylor's expansion gives
	$$
	T_n-s\log\cosh(A)\le \beta s|W_n|B +s\beta^2W_n^2,
	$$
	using which the RHS \eqref{eq:theta<1} can be bounded as follows:
	\begin{align*}
	 \E_{\beta,\bQ^{\rm CW},\mathbf{0}} e^{2T_n-2s\log\cosh(A)} \le \E_{\beta,\bQ^{\rm CW},\mathbf{0}} e^{2\beta sB|W_n|+2s\beta^2 W_n^2}\le 2\E_{\beta,\bQ^{\rm CW},\mathbf{0}} e^{2\beta sBW_n+2s\beta^2 W_n^2}
	 \end{align*}
	Setting $\lambda_1:=\beta(1-\beta),\lambda_2:=\beta$ we have 
	$$\lambda_1 \frac{w^2}{2}\le f(w)\le \lambda_2\frac{w^2}{2},$$
	which readily gives
	\begin{align*}
	\E_{\beta,\bQ^{\rm CW},\mathbf{0}} e^{2\beta sW_nB+s\beta^2 W_n^2}\le &2\frac{\int_\R e^{2\beta sBz+2s\beta^2z^2-n\lambda_1z^2/2}dz}{\int_\R e^{-n\lambda_2z^2/2}dz}\\
	 =&\sqrt{\frac{n\lambda_2}{n\lambda_1-4\beta^2 s}}\text{exp}\Big\{\frac{2\beta^2s^2B^2}{n\lambda_1-4\beta^2 s} \Big\},
	\end{align*}
	the RHS of which converges to $1$. Thus \eqref{eq:theta<1} holds, and so the proof is complete.
	\\

	\item[(b)]
	To begin, setting $C_n:=\int_\R e^{-n f(w)}dw$ we have
	\begin{align*}
	C_n\E e^{T_n}=&\int_{-\infty}^0 \prod_{i=1}^n e^{\log\cosh(\beta w+\mu_i)-\log\cosh(\beta w)} e^{-nf(w)}dw+\int_{0}^\infty \prod_{i=1}^n e^{\log\cosh(\beta w+\mu_i)-\log\cosh(\beta w)} e^{-nf(w)}dw\\
	=&\int_{0}^\infty \prod_{i=1}^n e^{\log\cosh(-\beta w+ \mu_i)-\log\cosh(-\beta w)} e^{-nf(w)}dw+\int_{0}^\infty \prod_{i=1}^n e^{\log\cosh(\beta w+\mu_i)-\log\cosh(\beta w)} e^{-nf(w)}dw\\
	\le &2\int_{0}^\infty \prod_{i=1}^n e^{\log\cosh(\beta w+\mu_i)-\log\cosh(\beta w)} e^{-nf(w)}dw,
	\end{align*}
	where the last inequality uses the fact that
	\begin{align*}\log\cosh(\beta w+\mu_i)-\log \cosh(\beta w)=&\int_{0}^{\mu_i}\tanh(\beta w+z)dz\\
	\le &\int_{0}^{\mu_i}\tanh(-\beta w+z)dz=\log\cosh(-\beta w+\mu_i)-\log \cosh(-\beta w),
	\end{align*}
	 as $\tanh$ is monotone increasing. This shows that 
	 \begin{align*}
	      \E_{\beta,\bQ^{\rm CW},\mathbf{0}} e^{T_n}
	     &=\frac{1}{2}\E_{\beta,\bQ^{\rm CW},\mathbf{0}} (e^{T_n}|W_n>0)+\frac{1}{2}\E_{\beta,\bQ^{\rm CW},\mathbf{0}} (e^{T_n}|W_n<0).
	 \end{align*}
	We will now show that,
	 \begin{align}\label{eq:theta>1+}
	\lim_{n\rightarrow\infty} \E_{\beta,\bQ^{\rm CW},\mathbf{0}} (e^{T_n}|W_n>0)=(1+o(1))\frac{\cosh(\beta m+A)^s}{\cosh(\beta m)^s}.
	 \end{align}
	 A similar argument takes care of the second term.
For this, use part (b) of Lemma \ref{lem:bayesian} to note that
	$\Big(\sqrt{n}(W_n-m)|W_n>0\Big)=O_p(1)$.
	For $W_n>0$, a Taylor's series expansion gives
	\begin{eqnarray*}
	\Big|\log\cosh(\beta W_n+A)-\log\cosh(\beta m+A)-\beta (W_n-m)\tanh(\beta m+A)\Big|&\le &\frac{\beta^2}{2}(W_n-m)^2\\
	 \Big|\log\cosh(\beta W_n)-\log\cosh(\beta m)-\beta(W_n-m)\tanh(\beta m)\Big|&\le &\frac{\beta^2}{2}(W_n-m)^2,
	\end{eqnarray*}
	which on taking a difference gives
	 $$\Big|T_n-s(\log\cosh(\beta m+A)-\log\cosh(\beta m))\Big|\le s\tilde{B}|W_n-m|+\beta^2s(W_n-m)^2, $$
	 where $\tilde{B}:=|\tanh(\beta m+A)-\tanh(\beta m)|$. The RHS in the display above converges to $0$ in probability, conditioned on the event $T_n>0$. As before, \eqref{eq:theta>1+} will follow from this via uniform integrability if we can show that
	\begin{align}\label{eq:theta>1ui}
	\E_{\beta,\bQ^{\rm CW},\mathbf{0}} e^{2T_n-2s\log\cosh(\beta m+A)+2s\log \cosh(\beta m)}<\infty.
	\end{align}
	For showing \eqref{eq:theta>1ui}, again use the above Taylor's expansion to note that
	$$T_n-s\log \cosh(\beta m+A)+s\log\cosh(\beta m)\le  s\tilde{B}|W_n-m|+\beta^2s(W_n-m)^2.$$
	We now claim that on $[0,\infty)$ the function $f(w)$ satisfies
	\begin{align}\label{eq:claim_new}
	\frac{\lambda_2}{2}(w-m)^2\le f(w)-f(m)\le \frac{\lambda_1}{2}(w-m)^2
	\end{align}
	for some positive constants $\lambda_1>\lambda_2$. Given the claim, a similar calculation as in part (a) shows that
	\begin{align*}
	\E_{\beta,\bQ^{\rm CW},\mathbf{0}}(e^{T_n}|W_n>0)\le &2\frac{\int_\R e^{2\beta s\tilde{B}z+2s\beta^2z^2-n\lambda_1z^2/2}dz}{\int_{-m}^\infty  e^{-n\lambda_2z^2/2}dz}\\
	 =&(1+o(1))\sqrt{\frac{n\lambda_2}{n\lambda_1-4\beta^2 s}}\text{exp}\Big\{\frac{2\beta^2s^2\tilde{B}^2}{n\lambda_1-4\beta^2 s} \Big\},
	\end{align*}
	which converges to $1$ as before, thus verifying \eqref{eq:theta>1ui} and hence completing the proof of the Lemma.
	
	It thus remains to prove \eqref{eq:claim_new}. To this effect, note that the function $f(.)$ on $[0,\infty)$ has a unique global minima at $z=m$, and $f''(m)>0$. Thus the function $F:[0,\infty)\mapsto \R$ defined by
	$$F(w)=\frac{f(w)-f(m)}{2(w-m)^2},z\ne m,\quad F(m):=f''(m)$$
	is continuous and strictly positive on $[0,\infty)$, and $\lim_{t\rightarrow\infty}F(t)=\beta_2>0$. Thus setting $\lambda_2:=\inf_{t\in [0,\infty)}F(t)$, $\lambda_1:=\sup_{t\in [0,\infty)}F(t)$, it follows that $\lambda_2>0$ and $\lambda_1<\infty$, and so the proof of \eqref{eq:claim_new} is complete.

	\item[(c)]
	We now use a four term Taylor expansion to get
	\begin{eqnarray*}
	\Big|\log\cosh(W_n+A)-\log\cosh(A)-\frac{W_n^2}{2}\text{sech}^2(A)|&\le &B|W_n|
	+\frac{1}{2}B|W_n|^3+ \frac{1}{4}|W_n|^4\\
	\Big|\log\cosh(W_n)-\frac{W_n^2}{2}\Big|&\le &\frac{1}{4}|W_n|^4
	\end{eqnarray*}
	  This gives
	$$|T_n-s\log\cosh(A)|\le sB|W_n|+sB^2\frac{W_n^2}{2}+\frac{1}{2}sB|W_n|^3+\frac{1}{2}sW_n^4$$
	The RHS above converges to $0$ on noting that $sB\ll  n^{1/4}$, and $W_n=O_p(n^{-1/4})$  (which follows by part (c) of Lemma \ref{lem:bayesian}).
	\\

%
As before, using uniform integrability it suffices to show that
	\begin{align}\label{eq:unif1}
	\limsup_{n\rightarrow\infty}\E_{\beta,\bQ^{\rm CW},\mathbf{0}} e^{2T_n-2s\log\cosh(A)}<\infty.
	\end{align}
	To this end, 
	use the bound $x-\log 2\le \log\cosh(x)\le x$ for $x\ge 0$ to conclude
	$$T_n-s\log\cosh(A)\le 2s\log 2.$$
	
	Also, since $$\lim_{z\rightarrow0}\frac{f(z)}{z^4}=\frac{1}{12},$$ there exists $\delta\in (0,1)$ and $0<\lambda_2\le\lambda_1<\infty$ such that for $|z|\le \delta$ we have $$\lambda_1 z^4\le f(z)\le \lambda_2 z^4,$$ which gives
	\begin{align}\label{eq:theta=1}
	\E_{\beta,\bQ^{\rm CW},\mathbf{0}} e^{2W_n-2s\log\cosh(A)}\le &e^{2s\log 2}\P(|T_n|>\delta)+\frac{\int_{-\delta}^\delta e^{5sB|z|-n\lambda_1z^4}dz}{\int_{-\delta}^\delta e^{-n\lambda_2z^4}dz},
	\end{align}
	where the last line uses the bound $$\max\Big(sB^2z^2, sB|z|^3+sz^4\Big)\le sB|z|$$ for all $|z|\le \delta$.
	We now bound each of the terms in the RHS of \eqref{eq:theta=1}. The denominator in the RHS of \eqref{eq:theta=1} by a change of variable equals $$n^{-1/4}\int_{-\delta n^{1/4}}^{\delta n^{1/4}} e^{-\lambda_2 t^4}dt=(1+o(1))n^{-1/4}\int_{-\infty}^\infty e^{-\lambda_2t^4}dt.$$ For estimating the numerator of RHS, set $t_n:=\Big(\frac{10sB}{n\lambda_1}\Big)^{1/3}$ and note that for $t>t_n$ we have $5sB|z|\le \frac{1}{2}n\lambda_1 z^4$. This gives
	\begin{align*}
	\int_{-\delta}^\delta e^{5sB|z|-n\lambda_1z^4}dz\le &2\int_0^{t_n}e^{5sBz}dz+2\int_{t_n}^\delta e^{-\frac{n\lambda_1}{2}t^4}dt\\
	\le &2t_ne^{5sBt_n}+2\int_{0}^\infty e^{-\frac{n\lambda_1}{2}t^4}dt\\
	= &2t_ne^{5sBt_n}+2n^{-1/4}\int_0^\infty e^{-\lambda_1 t^4/2}dt.
	\end{align*}
	The first term above is asymptotically $2t_n\ll n^{-1/4}$, where we use the bound $sB\ll n^{1/4}$. Thus the  numerator in the second term in the RHS of \eqref{eq:theta=1} is $O(n^{-1/4})$. Proceeding to estimate the first term, use Lemma \ref{lem:bayesian} to note that $\P(|W_n|>\delta)$ decays exponentially in $n$, whereas $e^{2s\log 2}\le e^{\frac{2n\log 2}{\log n} }$ is sub exponential. It thus follows that the RHS of \eqref{eq:theta=1} is bounded, and so the proof of part (c) is complete.

	\end{enumerate}
	\end{proof}

\begin{lemma}\label{lem:htrat}
Suppose $\bX\sim \P_{\beta,\bQ^{\rm CW},\bmu_S(A)}$ with $\bQ^{\rm CW}_{i,j}=\mathbf{1}(i\neq j)/n$. 	If $\lambda \in \R$, $\beta<1$, $s \ll n/ \log n$, $B := s \max \{\tanh (\lambda+A),\tanh (A)\} \ll \sqrt{n}$, then
	\begin{equation}\label{eq:htaltmgf}
		\E_{\beta,\bQ^{\rm CW},\bmu_S(A)}\big(\exp{(\lambda\sum\limits_{i \in S}X_i)}\big)= (1+o(1))\Big(\frac{\cosh(A+\lambda)}{\cosh(A)}\Big)^s.
	\end{equation}
\end{lemma}
\begin{proof} Observe that,
\[ \E_{\beta,\bQ^{\rm CW},\bmu_S(A)}\big(\exp{(\lambda\sum\limits_{i \in S}X_i)}\big)= \frac{Z(\beta,\bQ^{\rm CW},\bmu_S(A+\lambda))}{Z(\beta,\bQ^{\rm CW},\bmu_S(A))}=\frac{Z(\beta,\bQ^{\rm CW},\bmu_S(A+\lambda))/Z(\beta,\bQ,\bzero)}{Z(\beta,\bQ^{\rm CW},\bmu_S(A))/Z(\beta,\bQ,\bzero)}.
\]
Now, Lemma \ref{lemma:curie_part_fun_ratio} completes the proof.
\end{proof}

\begin{lemma}\label{lemma:cw_aux_var_alt_high_temp} Suppose $\bX\sim \P_{\beta,\bQ^{\rm CW},\bmu}$ with $\bQ^{\rm CW}_{i,j}=\mathbf{1}(i\neq j)/n$. 
When $\beta>1$, let $m$ be the unique positive root of $m=\tanh(\beta m)$. Fix $\bmu\in \Xi(s,A)$. 
Then the following hold uniformly in $\bmu\in \mathbb{R}_+^n$ such that $\sum_{\i=1}^n \bmu_i\ll n$.
	
	\begin{enumerate}
	    \item [(a)] 
	Fix any sequence $\kappa_n\rightarrow 0$. Then  then there exists a sequence $\zeta_n\rightarrow 0$ (depending only on the pre-fixed sequence $\kappa_n$) such that, $$\sup_{\bmu:\sum_{i=1}^n \mu_i\leq \kappa_n n}\P_{\beta,\bQ^{\rm CW},\bmu}\left(|W_n - m|>\zeta_n, W_n\geq 0\right) \rightarrow 0$$
	
	\item [(b)] 	Fix any sequence $\kappa_n\rightarrow 0$. Then {\color{black} then there exists a sequence $\zeta_n\rightarrow 0$ (depending only on the pre-fixed sequence $\kappa_n$) such that, $$\sup_{\bmu:\sum_{i=1}^n \mu_i\leq \kappa_n n}\P_{\beta,\bQ^{\rm CW},\bmu}\left(|W_n + m|>\zeta_n, W_n\leq 0\right) \rightarrow 0$$
	}
	\end{enumerate}
\end{lemma}

\begin{proof}

For the proof of the case $W_n\geq 0$ we first note that by \cite{ellis2006entropy} there exists $C>0$ $\P_{\beta,\bQ,\bzero}\left(|W_n-m|>{\zeta_n},W_n\geq 0\right)\leq \exp\left(-Cn\zeta_n^2\right)$ and hence for $\sum_{i=1}^n \mu_i\ll n$ we have with $\mathcal{E}_n=\{|W_n-m|>\frac{\zeta_n}{\sqrt{n}},W_n\geq 0\}$ and any sequence $\kappa_n\rightarrow 0$ that
\begin{align*}
   \sup_{\bmu:\sum_{i=1}^n \mu_i\leq \kappa_n n} \P_{\beta,\bQ,\bmu}\left(\mathcal{E}_n\right)&= \sup_{\bmu:\sum_{i=1}^n \mu_i\leq \kappa_n n}\frac{Z(\beta,\bQ,\bzero)}{Z(\beta,\bQ,\bmu)}\E_{\beta,\bQ,\bzero}\left(\mathbbm{1}(\mathcal{E}_n)e^{\sum_{i=1}^n \mu_iX_i}\right)\\
    &=\sup_{\bmu:\sum_{i=1}^n \mu_i\leq \kappa_n n}e^{2\sum_{i=1}^n \mu_i}\P_{\beta,\bQ,\bzero}(\mathcal{E}_n)\leq e^{2\sum_{i=1}^n\bmu_i-Cn\zeta_n^2}\rightarrow 0
\end{align*}
whenever $n\zeta_n^2\gg \kappa_n n$. This completes the proof of the lemma. For the case $W_n \leq 0$ \cite{ellis2006entropy} there exists $C>0$ $\P_{\beta,\bQ,\bzero}\left(|W_n+m|>{\zeta_n},W_n\geq 0\right)\leq \exp\left(-Cn\zeta_n^2\right)$ and the rest of the proof is similar.

\end{proof}

The next lemma follows from \cite[Theorems 1-3]{Ellis_Newman}. 
\begin{lemma}\label{lem:bayesian} Suppose $\bX\sim \P_{\beta,\bQ^{\rm CW},\bmu}$ with $\bQ^{\rm CW}_{i,j}=\mathbbm{1}(i\neq j)/n$.
\begin{enumerate}
\item[(a)]
If $\beta\in (0,1)$ then under $\P_{\beta,\bQ^{\rm CW},\mathbf{0}}$ we have $$\sqrt{n}\bar{\bX}\stackrel{d}{\rightarrow}N\left(0,\frac{1}{1-\beta}\right).$$

\item[(b)]
If $\beta>1$ then under $\P_{\beta,\bQ^{\rm CW},\mathbf{0}}$ we have $$\sqrt{n}(\bar{\bX}-m|\bar{\bX}>0)\stackrel{d}{\rightarrow}N\left(0,\frac{1-m^2}{1-\beta(1-m^2)}\right),$$
where $m$ is the unique positive root of the equation $t=\tanh(\beta t)$.

\item[(c)]
If $\beta=1$ then under $\P_{\beta,\bQ^{\rm CW},\mathbf{0}}$ we have $$n^{1/4}\bar{\bX}\stackrel{d}{\rightarrow}Y,$$
where $Y$ is a random variable with density proportional to $e^{-y^4/12}$.

\end{enumerate}
\end{lemma}

\subsection{\bf Proof of Results in Section \ref{sec:dense_reg}}

\subsubsection{\bf Proof of Theorem \ref{theorem:er_known_beta}} We will divide the proof into three parts based on high $(\beta<1)$, critical $(\beta=1)$, and low $(\beta>1)$ temperature. Throughout we will use $\bQ$ for either scaled adjacency matrix of of the graph $G(n,p)$ (with $p=\Theta(1)$) or random regular graph on $n$ vertices each having degree $d=\Theta(n)$.

\paragraph{\bf Proof for upper bound in High Temperature $(0\leq \beta<1)$ Regime:}
We first focus on the upper bound. As before, the optimal test is given the by the scan test which rejects for large values of $Z_{\max}:=\max\limits_{S\in S\in \mathcal{N}(\C_s,\gamma,\varepsilon_n)} Z_S$ (recall that for $S\in \mathcal{N}(\C_s,\gamma,\varepsilon_n)$ we defined $Z_S=\sum_{i\in S}X_i/\sqrt{s}$). The cut-off for the test is decided by the moderate deviation behavior of $Z_S$'s given in Lemma \ref{lemma:er_regular_moderate_deviation} -- which implies that for any $\delta>0$, the test given by $T_n(\delta)=\mathbf{1}\left\{Z_{\max}>\sqrt{2(1+\delta)\log{|\mathcal{N}(\C_s,\gamma,\varepsilon_n)|}}\right\}$ has Type I error converging to $0$.

Turning to the Type II error, consider any $\P_{\beta,\bQ,\bmu}\in \Xi(\C_s,A)$.  First note that by GHS inequality $\mathrm{Var}_{\beta,\bQ,\bmu}(\sum_{i\in \tilde{S}^{\star}} X_i)\leq \mathrm{Var}_{\beta,\bQ,\bzero}(\sum_{i\in \tilde{S}^{\star}} X_i)=O(s)$. As a result, $\frac{\sum_{i\in \tilde{S^{\star}}}(X_i-\E_{\beta,\bQ,\bmu}(X_i))}{\sqrt{s}}=O_{\P_{\beta,\bQ,\bmu}}(1)$. Therefore, as usual it is enough to show that there exists $\delta>0$ such that $t_n(\delta)-\frac{1}{\sqrt{s}}\E_{\beta,\bQ,\bmu}\left(\sum_{i\in \tilde{S}^{\star}}X_i\right)\rightarrow -\infty$. To show this end, first let $S^{\star}\in \C_s$ be such that the signal lies on $S^\star$ i.e. for all $i\in S^{\star}$ one has $\bmu_i\geq A$. Note that by monotonicity arguments it is enough to consider $\bmu_i=A$. By definition of covering, we can find a $\tilde{S}^{\star}\in \mathcal{N}(\C_s,\gamma,\varepsilon_n)$ such that $\gamma\left(\tilde{S}^{\star},S^{\star}\right)\leq \varepsilon_n$ i.e. $|\tilde{S}^{\star}\cap S^{\star}|\geq s(1-\varepsilon_n/\sqrt{2})$. 
Thereafter note that by Lemma 
\ref{lemma:expectation_underalt_lower_bound} we have that
\begin{align}
        \E_{\beta,\bQ,\bmu}(\sum_{i\in \tilde{S}^{\star}}X_i)\geq A|\tilde{S}^{\star}\cap S^{\star}|-A^2\sum_{i\in \tilde{S}^{\star}\cap S^{\star}}\sum_{j\in S}\mathrm{Cov}_{\beta,\bQ,\tilde{\mathbf{\eta}}_{S^{\star}}(A)}(X_i,X_j),\label{eq:er_high_t2}
    \end{align}
However, by \cite[Lemma 9 (a) and Theorem 7]{deb2020detecting} we have that
\begin{align*}
    A^2\sum_{i\in \tilde{S}^{\star}\cap S^{\star}}\sum_{j\in S}\mathrm{Cov}_{\beta,\bQ,\tilde{\mathbf{\eta}}_{S^{\star}}(A)}(X_i,X_j)&\lesssim \frac{\log{|\mathcal{N}(\C_s,\gamma,\varepsilon_n)|}}{s}\left(\frac{|\tilde{S}^{\star}\cap S^{\star}|| S^{\star}|}{n}+|\tilde{S}^{\star}\cap S^{\star}|\right)\\
    &\leq \frac{\log{|\mathcal{N}(\C_s,\gamma,\varepsilon_n)|s}}{{n}}+\log{|\mathcal{N}(\C_s,\gamma,\varepsilon_n)|}\ll As,
\end{align*}
 since $\log{|\mathcal{N}(\C_s,\gamma,\varepsilon_n)|} \ll s\ll n/\log{|\mathcal{N}(\C_s,\gamma,\varepsilon_n)|}$. Consequently, we immediately have that there exists $\epsilon>0$ such that
\begin{align*}
     \E_{\beta,\bQ,\bmu}\left(\sum_{i\in \tilde{S}^{\star}}X_i\right)&\geq As(1+o(1))\geq \sqrt{2(1+\epsilon)s\log{|\mathcal{N}(\C_s,\gamma,\varepsilon_n)|}}
\end{align*}
Therefore, we can conclude that for any $\delta<\epsilon$ we have  $t_n(\delta)-\frac{1}{\sqrt{s}}\E_{\beta,\bQ,\bmu}\left(\sum_{i\in \tilde{S}^{\star}}X_i\right)\rightarrow -\infty$. This completes the proof of the upper bound for $0<\beta<1$. \par 

\paragraph{\bf Proof for upper bound at Critical Temperature $(\beta=1)$ Regime:} Similar to $\beta<1$ regime, the optimal test is given the by the scan test which rejects for large values of $Z_{\max}:=\max\limits_{S\in S\in \mathcal{N}(\C_s,\gamma,\varepsilon_n)} Z_S$ (recall that for $S\in \mathcal{N}(\C_s,\gamma,\varepsilon_n)$ we defined $Z_S=\sum_{i\in S}X_i/\sqrt{s}$). The cut-off for the test is decided by the moderate deviation behavior of $Z_S$'s given in Lemma \ref{lemma:er_regular_moderate_deviation} -- which implies that for any $\delta>0$, the test given by $T_n(\delta)=\mathbf{1}\left\{Z_{\max}>\sqrt{2(1+\delta)\log{|\mathcal{N}(\C_s,\gamma,\varepsilon_n)|}}\right\}$ has Type I error converging to $0$.

 \par 
Turning to the Type II error, we again consider any $\P_{\beta,\bQ,\bmu}\in \Xi(\C_s,A)$. First note that by GHS inequality $\mathrm{Var}_{\beta,\bQ,\bmu}(\sum_{i\in \tilde{S}^{\star}} X_i)\leq \mathrm{Var}_{\beta,\bQ,\bzero}(\sum_{i\in \tilde{S}^{\star}} X_i)=O(s)$ even for $\beta=1$ since $s\ll \frac{\sqrt{n}}{\log{n}}$ by appealing to \cite[Lemma 9(c)]{deb2020detecting}. Hence, it is enough to show that there exists $\delta>0$ such that $t_n(\delta)-\frac{1}{\sqrt{s}}\E_{\beta,\bQ,\bmu}\left(\sum_{i\in \tilde{S}^{\star}}X_i\right)\rightarrow -\infty$. As before, let $S^{\star}\in \C_s$ be such that the signal lies on $S^\star$ i.e. for all $i\in S^{\star}$ one has $\bmu_i\geq A$. By monotonicity arguments it is enough to consider $\bmu_i=A \mathbbm{1}_{i \in S^{\star}}$. Following \eqref{eq:er_high_t2}, it is enough to lower bound $A^2\sum_{i\in \tilde{S}^{\star}\cap S^{\star}}\sum_{j\in S}\mathrm{Cov}_{\beta,\bQ,\tilde{\mathbf{\eta}}_{S^{\star}}(A)}(X_i,X_j)$. By \cite[Lemma 9 (c) and Theorem 7]{deb2020detecting}, we have that
\begin{align*}
    A^2\sum_{i\in \tilde{S}^{\star}\cap S^{\star}}\sum_{j\in S}\mathrm{Cov}_{\beta,\bQ,\tilde{\mathbf{\eta}}_{S^{\star}}(A)}(X_i,X_j)&\lesssim \frac{\log{|\mathcal{N}(\C_s,\gamma,\varepsilon_n)|}}{s}\left(\frac{|\tilde{S}^{\star}\cap S^{\star}|| S^{\star}|}{\sqrt{n}}+|\tilde{S}^{\star}\cap S^{\star}|\right)\\
    &\leq \frac{\log{|\mathcal{N}(\C_s,\gamma,\varepsilon_n)|s}}{\sqrt{n}}+\log{|\mathcal{N}(\C_s,\gamma,\varepsilon_n)|}\ll As,
\end{align*}
since $\log{|\mathcal{N}(\C_s,\gamma,\varepsilon_n)|} \ll s \ll \frac{n}{\log{|\mathcal{N}(\C_s,\gamma,\varepsilon_n)|}}$. This completes the proof of the upper bound for $\beta=1$.\par 

\paragraph{\bf Proof for upper bound in Low Temperature $(\beta>1)$ Regime:}
To prove the upper bound on detection threshold, we will use the same randomized test as described in Theorem \ref{theorem:cw_known_beta}\ref{thm:cw_known_beta_ub}\ref{thm:cw_known_beta_ub_low_temp}. Let the rejection region of the test be denoted by $\Omega_n$. Using \cite[Theorem 1.6]{debsum}, we obtain 
$$
    \log \P_{\beta,\bQ,\bzero}(\Omega_n) \leq C_\delta \Big( \log \P_{\beta,\bQ^{\rm CW},\bmu_S(A)}(\Omega_n)+\|\bQ\|^2_F+\sum_{i=1}^{n}(R_i-1)^2\Big).
$$
Now, $\P_{\beta,\bQ^{\rm CW},\bzero}(\Omega_n)=o(1)$ by Theorem \ref{theorem:cw_known_beta}\ref{thm:cw_known_beta_ub}\ref{thm:cw_known_beta_ub_low_temp}. The error term in RHS is $O(1)$ by \cite[Section 1.2]{debsum}. Hence, type-I error of the proposed test converges to $0$.\par 
For the type-II error, fix $\eta>0$  and $\tanh(A)=\sqrt{2(1+\delta)^2(1-m^2)^{-1}\log{|\mathcal{N}(\C_s,\gamma,\varepsilon_n)|}/s}$. To get the desired cut-off of the test, take $\delta>0$ to be chosen later and set $t_n(\delta)=\sqrt{2(1+\delta)(1-m^2)\log{|\mathcal{N}(\C_s,\gamma,\varepsilon_n)|}}$. Assume $\bmu$ be the true signal with support $S^\star$ and we scan over $S$ such that $\gamma(S, S^\star) \le \varepsilon_n=o(1)$. We will show that $\bP_{\beta, \bQ, \bmu}(\sum_{i\in S} X_i> ms+t \sqrt{s}|\bar{\bX} \ge 0) \rightarrow 1$. Note that,
\begin{align*}
    &\bP_{\beta, \bQ, \bmu}(\sum_{i\in S} X_i> ms+t \sqrt{s}|\bar{\bX} \ge 0) \\ &= \bP_{\beta, \bQ, \bmu}(\frac{\sum_{i\in S} X_i-\bE_{\beta,\bQ,\bmu}(\sum_{i\in S} X_i|\bar{\bX}>0)}{\sqrt{Var_{\beta,\bQ,\bmu}(\sum_{i\in S} X_i|\bar{\bX}>0)}}> \frac{ms+t \sqrt{s}-\bE_{\beta,\bQ,\bmu}(\sum_{i\in S} X_i|\bar{\bX}>0)}{\sqrt{Var_{\beta,\bQ,\bmu}(\sum_{i\in S} X_i|\bar{X}>0)}}|\bar{\bX} \ge 0)
\end{align*}
Since LHS is $O_{\bP_{\beta,\bQ,\bmu}}(1)$, it is enough to show the RHS (henceforth denoted by $T_n$) converges to $-\infty$. Defining $\alpha_n=\sqrt{\frac{\log n}{n}}$, we obtain the following estimates from \cite[Lemma 9]{deb2020detecting} for some absolute constants $c_i >0$, $i \in [3]$:
\begin{align}
    \max_{i,j}\mathrm{Cov}_{\beta,\bQ,\bmu}(X_i,X_j|\bar{\bX}\ge 0) \le \frac{c_1}{n} \label{eq:cov_er_lt}\\
    \max_{i}|\mathrm{Var}_{\beta,\bQ,\bmu}(X_i|\bar{X}\ge 0)-\sech^2 (\beta m+\mu_i)|\le c_2 \alpha_n \label{eq:var_er_lt}\\
    \max_{i}|\bE_{\beta,\bQ,\bmu}(X_i|\bar{X}\ge 0)-\tanh (\beta m+\mu_i)|\le c_3 \alpha_n \label{eq:mean_er_lt}.
\end{align}
All the above statements hold with high probability with respect to the randomness of $\mathbb{G}_n$.


Using the bounds, we obtain 
\begin{align*}
    T_n &\le \frac{ms-t_n(\delta)\sqrt{s}-\sum_{i \in S}\tanh(\beta m+\mu_i)+sc_3\alpha_n}{\sqrt{\sum_{i \in S}\sech^2(\beta m+\mu_i)+sc_2\alpha_n- \frac{cs^2}{n}}}
    \le \frac{ms-t_n(\delta)\sqrt{s}-s\tanh(\beta m+A)+sc_3\alpha_n}{\sqrt{s\sech^2(\beta m+A)+sc_2\alpha_n- \frac{cs^2}{n}}}\\
    &\le \frac{1}{C\sqrt{s}}\left(ms-t_n(\delta)\sqrt{s}-s\big(m-(1+o(1)A \sech^2(\beta m)\big)+sc_3\alpha_n\right)\\
    & \le \frac{1}{C} \left(-t_n(\delta)+As(1+o(1))(1-m^2)+o(As)\right),
\end{align*}
since $\alpha_n=o(A)$. By choosing $\delta$ based on $\eta$, the final bound converges to $\infty$. {\color{black} The same calculation holds for $\P_{\beta,\bQ,\bmu}\big(\sum_{i \in S} X_i > -ms+ t\sqrt{s}|\bar{\bX}<0\big)$} yielding Type-II error converging to $0$.

\paragraph{\bf Proof for lower bound in $\beta>0$ Regime:}
To prove the lower bound, fix $\beta>$ and $\varepsilon>0$ such that $\sqrt{s}\tanh(A) (\log|\tilde{\mathcal{C}}_s|)^{-1/2}= 2(1-\varepsilon)c_{\beta}$ where $c_{\beta}=1$ if $\beta \le 1$, and $(1-m^2)^{-1}$ otherwise. Suppose there exists a test with rejection region $\Omega_n$ which can test $H_0$ vs $H_1$. Hence, $\bP_{\beta,\bQ,\bzero}(\Omega_n)\rightarrow 0$. Using Lemma \ref{lemma:compare_sets_low_temp}, we have $\bP_{\beta,\bQ^{\rm CW},\bzero}(\Omega_n)\rightarrow 0$. However, since every test not asymptotically powerful for Curie-Weiss model in this regime of signal, implying that,
$\bP_{\beta,\bQ^{\rm CW},\bmu_S(A)}(\Omega^c_n) \ge \eta$, for some $\eta>0$ for all large enough $n$. This implies, using Lemma \ref{lemma:compare_sets_low_temp}, there exists $\nu >0$ such that $\bP_{\beta,\bQ,\mu_S(A)}(\Omega_n) \ge \nu$. Hence, the type-I and type-II errors cannot converge to $0$ simultaneously completing the proof of the lower bound.

\subsubsection{\bf Proof of Theorem \ref{thm:er_unknown_beta}}
To obtain an adaptive test, we apply the same procedure as described by the proof of Theorem \ref{thm:cw_unknown_beta}. The proof of type-I error convergence when $\bmu=0$ follows from concentration of $\bar{\bX}$ which is immediate by \cite[Theorem 1.1, Theorem 1.2]{debsum}. For general $\bmu$, we only require the fact $\mathrm{Cov}_{\beta,\bQ,\bmu}(X_i,X_j)\lesssim 1/n$ which follows from \cite[Lemma 9]{deb2020detecting}. Finally we also can obtain consistent estimator of $\beta$ thanks to \cite[Corollary 3.1, Corollary 3.2]{infising} and Lemma \ref{lmm:beta_estimation}. This completes our proof exactly as Theorem \ref{thm:cw_unknown_beta}.

\subsection{\bf Technical Lemmas for Proofs of Theorems in Section \ref{sec:dense_reg}}\label{sec:technical_lemmas_mean_field}

\begin{lemma}\label{lemma:compare_sets_low_temp}
Fix $\beta>0$. 
Let $\bQ^{\rm CW}$ be the scaled adjacency matrix of the complete graph and $\bQ$ is either scaled adjecency matrix of of the graph $G(n,p)$ (with $p=\Theta(1)$) or random regular graph on $n$ vertices each having degree $d=\Theta(n)$. Fix any $\delta>0$. For any event $\mathcal{E}_n$: the following holds with probability $\ge 1-\delta$: If $\exists \eta> 0$ such that $\bP_{\beta,\bQ^{\rm CW},\mu_S(A)}(\mathcal{E}_n) \ge \eta$, then $\exists \nu >0$ such that $\bP_{\beta,\bQ,\mu_S(A)}(\mathcal{E}_n) >\nu$. 

\end{lemma}

\begin{proof}
Defining $\mathcal{A}_n:= G_n- 11^T/n+ I/n$, note that
\[
\P_{\beta,\bQ,\bmu_S(A)}(\mathcal{E}_n)= \frac{Z(\beta,\bQ^{\rm CW},\bmu_S(A))}{Z(\beta,\bQ,\bmu_S(A))}\E_{\beta,\bQ^{\rm CW},\bmu_S(A)}e^{\frac{\beta}{2}\sigma^\top \mathcal{A}_n \sigma} \mathbbm{1}_{\sigma \in \mathcal{E}_n}.\]
For the auxiliary variable $W_n$, define a vector $\tilde{T}_n=\tilde{T}_n(W_n)$ such that $\tilde{T}_{n,i}=\tanh(\beta W_n +A \mathbbm{1}_{i \in S})$. Observe that,
\[\sigma^\top \mathcal{A}_n \sigma= (\sigma^\top-\tilde{T}_n) \mathcal{A}_n (\sigma-\tilde{T}_n)+2 \tilde{T}^\top_n\mathcal{A}_n (\sigma-\tilde{T}_n)+\tilde{T}^\top_n\mathcal{A}\tilde{T}_n =: Y_n+\tilde{T}^\top_n\mathcal{A}\tilde{T}_n.
\]
Define the good set $J_{n,\varepsilon}:= \{m-\varepsilon \le |W_n| \le m+\varepsilon\}$. Then,
\begin{equation}
    \P_{\beta,\bQ,\bmu_S(A)}(\mathcal{E}_n)\ge \frac{Z(\beta,\bQ^{\rm CW},\bmu_S(A))}{Z(\beta,\bQ,\bmu_S(A))} \left( \E_{\beta,\bQ^{\rm CW},\bmu_S(A)}e^{\sigma^\top \mathcal{A}_n \sigma} \mathbbm{1}_{J_{n,\varepsilon}}\right)
\end{equation}
The ratio of partition function is $\ge 1/C_u$ for some $C_u>0$, w.p. $\geq 1-\delta$ by Lemma \ref{lemma:log_partition_function_comparison_er}.  

To analyse the expectation above, define $A_M=\{\sigma: |Y_n| \le M\}$ for some $M>0$ and observe that,
\begin{align*}
    \E_{\beta,\bQ^{\rm CW},\bmu_S(A)}e^{\frac{\beta}{2}\sigma^\top \mathcal{A}_n \sigma} \mathbbm{1}_{J_{n,\varepsilon}}\mathbbm{1}_{\sigma \in \mathcal{E}_n} &\geq \E\left(\E_{\beta,\bQ^{\rm CW},\bmu_S(A)}\Big(e^{\frac{\beta}{2}\sigma^\top \mathcal{A}_n \sigma} \mathbbm{1}_{J_{n,\varepsilon}}\mathbbm{1}_{\sigma \in \mathcal{E}_n} \mathbbm{1}_{A_M}|W_n\Big)\right)\\
    &\ge e^{-\frac{\beta}{2}M}\E\left(e^{\frac{\beta}{2}\tilde{T}^\top_n\mathcal{A}\tilde{T}_n}\bP_{\beta,\bQ,\bmu_S(A)}(J_{n,\varepsilon} \cap \mathcal{E}_n \cap A_M|W_n)\right)
\end{align*}
The proof of Lemma \ref{lemma:log_partition_function_comparison_er} yields $\exists \theta>0$ and $\delta>0$ small enough such that
$\inf_{W_n}\exp(\frac{\beta}{2}\tilde{T}_n(W_n)^T\mathcal{A}_n\tilde{T}_n(W_n))\ge e^{-\theta}$ with probability $\geq 1-\delta$.
Therefore,
\begin{align*}
    \E_{\beta,\bQ^{\rm CW},\bmu_S(A)}e^{\frac{\beta}{2}\sigma^\top \mathcal{A}_n \sigma} \mathbbm{1}_{J_{n,\varepsilon}}\mathbbm{1}_{\sigma \in \mathcal{E}_n} \ge e^{-\frac{\beta}{2}M-\theta} \bP_{\beta,\bQ,\bmu_S(A)}(J_{n,\varepsilon} \cap \mathcal{E}_n \cap A_M).
\end{align*}
By Hanson-Wright inequality(cf. Lemma \ref{lem:hanson_wright}), pick $M$ large enough such that $\P_{\beta,\bQ^{\rm CW},\bmu_S(A)}(A_M|W_n)\ge 1-\eta/4$. Hence, $\P_{\beta,\bQ^{\rm CW},\bmu_S(A)}(A_M)\ge 1-\eta/4$  By picking $\varepsilon>0$ small enough, we can also get $\P_{\beta,\bQ^{\rm CW},\bmu_S(A)}(J_{n,\varepsilon})\ge 1-\eta/4$. Hence,
\begin{align*}
    \E_{\beta,\bQ^{\rm CW},\bmu_S(A)}e^{\frac{\beta}{2}\sigma^\top \mathcal{A}_n \sigma} \mathbbm{1}_{J_{n,\varepsilon}} \ge \frac{\eta}{2} e^{-\frac{\beta}{2}M-\theta}=:\nu >0.
\end{align*}
This concludes the proof of this lemma.
\end{proof}

\begin{lemma}\label{lemma:auxillary_variable_conc_alt_low_temp}
 There exists a constant $C_{\varepsilon}>0$ such that 
 \begin{align*}
     \P_{\beta,\bQ^{\rm CW},\bmu_S(A)}(J^c_{n,\varepsilon})\leq e^{-Cn},
 \end{align*}
 where $\bQ^{\rm CW}$ corresponds to the coupling matrix of a Curie-Weiss model on $n$ vertices.
\end{lemma}
\begin{proof}
\begin{align*}
     \P_{\beta,\bQ^{\rm CW},\bmu_S(A)}(J^c_{n,\varepsilon})&=\frac{Z(\beta,\bQ^{\rm CW},\mathbf{0})}{Z(\beta,\bQ^{\rm CW} ,\bmu_S(A))}\E_{\beta,\bQ^{\rm},\mathbf{0}}\left(e^{A\sum_{i\in S}{X_i}}\mathbbm{1}_{J^c_{n,\varepsilon}}\right)\\
     &\leq e^{As}\P_{\beta,\bQ^{\rm CW},\mathbf{0}}(J^c_{n,\varepsilon})\leq e^{As-Cn},
\end{align*}
for some $C>0$ depending on $\varepsilon>0$ (where the proof of $\P_{\beta,\bQ^{\rm CW},\mathbf{0}}(J^c_{n,\varepsilon})\leq e^{-Cn}$ follows from \cite{ellis2006entropy}). The proof follows from the assumption that $As \ll n$.
\end{proof}

\begin{lemma}\label{lemma:log_partition_function_comparison_er}
Let $\bQ^{\rm CW}_{i,j}$ denote the coupling matrix of Curie-Weiss  model and $\bQ$ be the scaled adjacency matrix of either the graph $G(n,p)$ (with $p=\Theta(1)$) or random regular graph on $n$ vertices each having degree $d=\Theta(n)$. Then for any $\delta>0$ the following holds with probability $\geq 1-\delta$: for \textit{any} $\beta>0,\bmu_S(A)$:
\begin{align*}
     C_l\leq \frac{Z(\beta,\bQ,\bmu_S(A))}{Z(\beta,\bQ^{\rm CW},\bmu_S(A))} \leq C_u\quad,
\end{align*}
for constants $C_l,C_u$ depending on $\beta,\delta$. 


\end{lemma}

\begin{proof}
Define $\mathcal{A}_n:=G-\mathbf{1}\mathbf{1}^T+I/n$. For the auxiliary variable $W_n$, define a vector $\tilde{T}_n=\tilde{T}_n(W_n)$ such that $\tilde{T}_{n,i}=\tanh(\beta W_n +A \mathbbm{1}_{i \in S})$. Observe that,
\[\bX^\top \mathcal{A}_n \bX= (\bX-\tilde{T}_n)^\top \mathcal{A}_n (\bX-\tilde{T}_n)+2 \tilde{T}^\top_n\mathcal{A}_n (\bX-\tilde{W}_n)+\tilde{T}^\top_n\mathcal{A}\tilde{T}_n =: Y_n+\tilde{T}^\top_n\mathcal{A}\tilde{T}_n.
\]
Define the good set $J_{n,\varepsilon}:= \{m-\varepsilon \le |W_n| \le m+\varepsilon\}$, where $m$ is the unique positive root of $m=\tanh(\beta m)$. For the upper bound, note that,
\begin{align*}
    \frac{Z(\beta,\bQ,\bmu_S(A))}{Z(\beta,\bQ^{\rm CW},\bmu_S(A))}&=\E_{\beta,\bQ^{\rm CW},\bmu_S(A)}\Big(e^{\frac{\beta}{2}\bX^T\mathcal{A}_n\bX}\Big)\\ 
    &=\E_{\beta,\bQ^{\rm CW},\bmu_S(A)}e^{\bX^\top \mathcal{A}_n \bX} \mathbbm{1}_{J^c_{n,\varepsilon}}+ \E_{\beta,\bQ^{\rm CW},\bmu_S(A)}e^{\bX^\top \mathcal{A}_n \bX} \mathbbm{1}_{J_{n,\varepsilon}},
\end{align*}
To bound the first summand, note that for any $\delta>0$ one has that there exists a sequence $\epsilon_n(\delta)\rightarrow 0$ such that with probability $1-\epsilon_n(\delta)$ one has that $\|\mathcal{A}_n\|_{\rm op}\leq n^{-1/2+\delta}$ for either $G(n,p)$~\cite[Theorem 1.1]{vu2005spectral} or for random regular graph~\cite[Theorem A]{tikhomirov2019spectral}. As a result, for any $\delta>0$ one has that with probability with probability $1-\epsilon_n(\delta)$, $\sup\limits_{\sigma \in \{\pm\}^n}e^{\bX^T\mathcal{A}_n\bX}\leq e^{n^{1/2+\delta}}$. Subsequently, with probability larger than $1-\epsilon_n(\delta)$ the following hold
\begin{align*}
    \E_{\beta,\bQ^{\rm CW},\bmu_S(A)}e^{\sigma^\top \mathcal{A}_n \sigma} \mathbbm{1}_{J^c_{n,\varepsilon}}\leq e^{n^{1/2+\delta}}\P_{\beta,\bQ^{\rm CW},\bmu_S(A)}(J^c_{n,\varepsilon}).
\end{align*}
This inequality and Lemma \ref{lemma:auxillary_variable_conc_alt_low_temp} yields that the first summand is $o(1)$. To analyze $\E_{\beta,\bQ^{\rm CW},\bmu_S(A)}e^{\bX^\top \mathcal{A}_n \bX} \mathbbm{1}_{J_{n,\varepsilon}}$ we plan to invoke Hanson-Wright inequality (cf. Lemma \ref{lem:hanson_wright}). To this end, we choose $\varepsilon>0$ small enough such that $$\lim\sup \max_i s_{\tilde{T}_{n,i}} \lambda_1(\beta \mathcal{A}_n) <1,$$
where $s$ in Lemma \ref{lem:hanson_wright}. The choice of $\varepsilon$ is possible since $A=o(1)$. Now, we apply \ref{lem:hanson_wright} on $\bE (e^{\frac{\beta}{2}Y_n}|W_n)$ and the error term in Lemma \ref{lem:hanson_wright} is $O(1)$ as shown in \cite[Section 1.2]{debsum}. Hence, the proof concludes once we observe that $\log \E^{\rm CW} e^{\tilde{T}^\top_n\mathcal{A}\tilde{T}_n}= O(1)$ \textit{uniformly} over $S$. To see this, note that it is enough to show that for any $\delta>0$ there exists a $C_{\delta}>0$ such that 
$\sup_{W_n}\exp(\tilde{T}_n(W_n)^T\mathcal{A}_n\tilde{T}_n(W_n))\leq C_{\delta}$ with probability $\geq 1-\delta$ under either Erd\H{o}s-R\'{e}nyi randomness or random regular graph. To that end, note that with $S(\mu)$ denoting the support of $\bmu$ one has
\begin{align*}
    \tilde{T}_n(W_n)^T\mathcal{A}_n\tilde{T}_n(W_n)&= T_1+2T_2+T_3
    \end{align*}
        where
\begin{align*}
        T_1&=\frac{1}{\tilde{d}}\sum_{(i,j)
    \in S(\bmu)^c\times S(\bmu)^c}\tanh^2(\beta W_n)\left(G_{i,j}-\frac{\tilde{d}}{n}\right);\\
    T_2&=\frac{1}{\tilde{d}}\sum_{(i,j)
    \in S(\bmu)\times S(\bmu)^c}\tanh(\beta W_n+\mu_i)\tanh(\beta W_n)\left(G_{i,j}-\frac{\tilde{d}}{n}\right);\\
    T_3&=\frac{1}{\tilde{d}}\sum_{(i,j)
    \in S(\bmu)\times S(\bmu)}\tanh(\beta W_n+\mu_i)\tanh(\beta W_n+\mu_j)\left(G_{i,j}-\frac{\tilde{d}}{n}\right).
    \end{align*}
Above $\tilde{d}=np$ when $\mathbb{G}_n\sim \mathcal{G}_n(n,p)$ and $\tilde{d}=d$ when $\mathbb{G}_n\sim \mathcal{G}_n(n,d)$. In particular, in that case
\begin{align*}
    T_1&=\tanh^2(\beta W_n)\frac{1}{\tilde{d}}\sum_{(i,j)
    \in S(\bmu)^c\times S(\bmu)^c}\left(G_{i,j}-\frac{\tilde{d}}{n}\right)\\
    &=O_{\P_{\beta,\bQ,\bmu}}(1)O_{\P_{G_n}}\left(\frac{(n-s)}{\tilde{d}}\right) \quad \text{by Lemma \ref{lem:random_regular_covariance}}.
\end{align*}
Similarly, 
    \begin{align*}
    T_2&=\tanh(\beta W_n)\tanh(\beta W_n+A)\frac{1}{\tilde{d}}\sum_{(i,j)
    \in S(\bmu)\times s(\bmu)^c}\left(G_{i,j}-\frac{\tilde{d}}{n}\right)\\
    &=O_{\P_{\beta,\bQ,\bmu}}(1)O_{\P_{G_n}}\left(\frac{\sqrt{(n-s)s}}{\tilde{d}}\right) \quad \text{by Lemma \ref{lem:random_regular_covariance}},
\end{align*}
and
 \begin{align*}
    T_3&=\tanh^2(\beta W_n+A)\frac{1}{\tilde{d}}\sum_{(i,j)
    \in S(\bmu)\times s(\bmu)}\left(G_{i,j}-\frac{\tilde{d}}{n}\right)\\
    &=O_{\P_{\beta,\bQ,\bmu}}(1)O_{\P_{G_n}}\left(\frac{\sqrt{(n-s)s}}{\tilde{d}}\right) \quad \text{by Lemma \ref{lem:random_regular_covariance}},
\end{align*}
    These estimates immediately verifies claim that $T_j$'s $O_{\P}(1)$ whenever $\tilde{d}=\Theta(n)$ -- which is guaranteed by either $p=\Theta(1)$ when $\mathbb{G}_n\sim \mathcal{G}_n(n,p)$ or $d=\Theta(n)$ when $\mathbb{G}_n\sim \mathcal{G}_n(n,d)$.

For the lower bound on ratio of partition functions we begin by noting that \textcolor{black}{by convexity of $\bQ\rightarrow \log{Z(\beta,\bQ,\bmu)}$ for every $\bmu$},\footnote{for any differentiable convex function $f:\mathbb{R}^d\rightarrow \mathbb{R}$ one has $f(y)-f(x)\geq \langle \nabla f(x),y-x\rangle$}
\begin{align*}
  \exp\left(\beta\langle \E_{\beta,\bQ^{\rm CW},\bmu_S(A)}(\hat{\Sigma}),-\Delta_{\bQ}\rangle_{\rm TR}\right) \leq  \frac{Z(\beta,\bQ,\bmu_S(A))}{Z(\beta,\bQ^{\rm CW},\bmu_S(A))}\leq \exp\left(-\beta\langle \E_{\beta,\bQ,\bmu_S(A)}(\hat{\Sigma}),\Delta_{\bQ}\rangle_{\rm TR}\right)
\end{align*}
where $\hat{\Sigma}_{i,j}=X_iX_j$, $\Delta_{\bQ}=\bQ-\bQ^{\rm CW}$, and $\langle\cdot, \cdot \rangle_{\rm TR}$ denote the matrix trace norm.

 Thereafter it is enough to show that $$\exp\left(-\beta\langle \E_{\beta,\bQ^{\rm CW},\bmu_S(A)}(\hat{\Sigma}),\bQ^{\rm CW}-\bQ^{\rm ER}\rangle_{\rm TR}\right)=O(1)$$ with high probability. This part is similar to above and is only easier since one can simply compute mean and variance of $\langle \E_{\beta,\bQ^{\rm CW},\bmu_S(A)}(\hat{\Sigma}),\bQ^{\rm CW}-\bQ^{\rm ER}\rangle_{\rm TR}$ to conclude by Chebyshev's Inequality. In particular, 
 \begin{align*}
    \ &  \langle \E_{\beta,\bQ^{\rm CW},\bmu_S(A)}(\hat{\Sigma}),\bQ^{\rm CW}-\bQ^{\rm ER}\rangle_{\rm TR}\\
    &=\xi(S,S)\sum_{(i,j)\in S\times S}\Delta_{i,j}+2\xi(S,S^c)\sum_{(i,j)\in S\times S^c}\Delta_{i,j}+\xi(S^c,S^c)\sum_{(i,j)\in S^c\times S^c}\Delta_{i,j}
 \end{align*}
 where $\Delta_{i,j}=\frac{1}{\tilde{d}}\left(G_{i,j}-\frac{\tilde{d}}{n}\right)$ and $\xi(S,S),\xi(S,S^c),\xi(S^c,S^c)$ stand for the common values (by the exchangeability in the Curie-Weiss Model) of $\{\E_{\beta,\bQ^{\rm CW},\bmu_{S}(A)}(X_i X_j)$, for  $(i,j)\in S\times S\}$, $(i,j)\in S\times S^c\}$ and $(i,j)\in S^c\times S^c\}$ respectively. Since the $\xi's$ are bounded by $1$ in absolute value, the proof of the fact that   $ \langle \E_{\beta,\bQ^{\rm CW},\bmu_S(A)}(\hat{\Sigma}),\bQ^{\rm CW}-\bQ^{\rm ER}\rangle_{\rm TR}$ has bounded variance then follows from Lemma \ref{lem:random_regular_covariance}. This concludes the proof of this lemma.
  \end{proof}

\begin{lemma}\label{lem:random_regular_covariance}
   Let $G$ be the adjacency matrix of either the graph $G(n,p)$ (with $p=\Theta(1)$) or random regular graph on $n$ vertices each having degree $d=\Theta(n)$. For any two subsets $S_1,S_2\subset \{1,\ldots,n\}$, with $\min \{|S_1|,|S_2|\} \rightarrow \infty$, $\Psi= \Psi(S_1,S_2):= \frac{1}{\sqrt{|S_1||S_2|}}\sum_{i,j\in S_1\times S_2}(G_{ij}-\bE(G_{ij}))=O_{\P}(1)$, where $\bE(G_{ij})$ is $p$ or $d/n$ respectively.
\end{lemma}
\begin{proof}
If $G$ is the adjacency matrix of either the graph $G(n,p)$, $G_{ij}$, $i \le j$ are independent and the conclusion is immediate using Chebyshev's inequality. \par 
When $G$ is the adjacency matrix of random regular graph, for any $i,j,j' \in [n]$, $Cov(G_{ij},G_{ij'})$ are same by symmetry and will be denoted by $c_1$. Similarly, for $i \neq i'$, $Cov(G_{ij},G_{i'j'})$ will be same and will denote by $c_2$. For any $i\in [n]$, $\sum_j G_{ij}=d$. Hence, $Var(\sum_j G_{ij})=0$. So,
\[ n^2c_1 \geq 0= \sum_j \mathrm{Var}(G_{ij})+2 \sum_{j \neq j'} \mathrm{Cov}(G_{ij},G_{ij'}) \leq n+n^2c_1,
\]
implying $c_1=O(1/n)$. Similarly, $\mathrm{Var}(\sum_{i,j} G_{ij})=0$ and number of tuples ${(i,j),(i',j')}$, $i \neq i'$ are $\Theta(n^4)$ and $i=i'$ are $\Theta(n^3)$. So, 
\[
0= \sum_{i,j}\mathrm{Var}(G_{ij})+\Theta(n^3)c_1+\Theta(n^4)c_2,
\]
implying $c_2=O(1/n^2)$. Next, we use the covariance estimates to prove the Lemma. \par 
To this end, note that $\bE(\Psi)=0$. Also,
\begin{align*}
    \mathrm{Var} (\Psi)= \frac{1}{|S_1||S_2|}\Big(\sum_{i,j\in S_1\times S_2}\mathrm{Var}(G_{ij})+ (|S_1||S_2|^2+|S_2||S_1|^2)c_1+|S_1|^2|S_2|^2c_2 \Big)=O(1),
\end{align*}
proving $\Psi=O_{\P}(1)$, as desired.
\end{proof}

\begin{lemma}\label{lemma:er_regular_moderate_deviation}
Suppose $\bX\sim \P_{\beta,\bQ,\bmu}$ with $\bQ$ be either scaled adjecency matrix of of the graph $G(n,p)$ (with $p=\Theta(1)$) or random regular graph on $n$ vertices each having degree $d=\Theta(n)$. Let $Z_S=\sum_{i\in S}X_i/\sqrt{s}$. Define $t_n(\delta):=\sqrt{2(1+\delta)(1-m^2)\log{|\mathcal{N}(\C_s,\gamma,\varepsilon_n)|}}$ for some $\delta>0$, where $m:=m(\beta)$ is the unique positive root of $m=\tanh(\beta m)$.
\begin{enumerate}
    \item[(a)] If $\beta<1, s\ll \frac{n}{\log n}$, $\P_{\beta,\bQ^{\rm CW},\mathbf{0}}(Z_S> t_n(\delta)) \le (1+o(1))\frac{C_u}{C_l}e^{-t^2_n(\delta)}$, w.h.p. where $C_u,C_l$ are defined as Lemma \ref{lemma:log_partition_function_comparison_er}.
    \item[(b)]
    If $\beta=1, s\ll \frac{\sqrt{n}}{\log n}$, then same conclusion as (a) holds.
\end{enumerate}
\end{lemma}
\begin{proof}
\begin{enumerate}
    \item[(a)]
    Pick $\lambda>0$ such that $\tanh\left(\frac{\lambda}{\sqrt{s}}\right)=\frac{t_n(\delta)}{\sqrt{s}}$. So, $\lambda \geq t_n(\delta)$, and 
$\cosh\left(\frac{\lambda}{\sqrt{s}}\right)=\left(1-\frac{t^2_n(\delta)}{s}\right)^{-1/2}.$ Also, $s \tanh\left(\frac{\lambda}{\sqrt{s}}\right)= \sqrt{st_n(\delta)}=O(\sqrt{s \log |\mathcal{N}(\C_s,\gamma,\varepsilon_n)|}) \ll \sqrt{n}$. Hence, 
\begin{align*}
    \P_{\beta,\bQ,\mathbf{0}}(Z_S> t_n(\delta) &\le e^{-\lambda t_n(\delta)}\frac{Z(\beta,\bQ, \bmu_S(\frac{\lambda}{\sqrt{s}}))}{Z(\beta,\bQ,\bzero)} \le \frac{C_u}{C_l}e^{-\lambda t_n(\delta)}\frac{Z(\beta,\bQ^{\rm CW}, \bmu_S(\frac{\lambda}{\sqrt{s}}))}{Z(\beta,\bQ^{\rm CW},\bzero)}\\
    &= (1+o(1))\frac{C_u}{C_l}e^{-\lambda t_n(\delta)} \cosh^s\left(\frac{\lambda}{\sqrt{s}}\right) \le (1+o(1))\frac{C_u}{C_l}e^{-t^2_n(\delta)}
\end{align*}
w.h.p., where the second inequality is due to Lemma \ref{lemma:log_partition_function_comparison_er} and the equality is due to Lemma \ref{lemma:curie_part_fun_ratio}.
\item[(b)]
When $\beta=1$, we pick same $\lambda$. Now, $s \tanh\left(\frac{\lambda}{\sqrt{s}}\right) \ll n^{1/4}$ and the bound is obtained using Lemma \ref{lemma:log_partition_function_comparison_er} and Lemma \ref{lemma:curie_part_fun_ratio} again and the $o(1)$ term does not depend on the choice of $S$.
\end{enumerate}
\end{proof}

\subsection{\bf Proofs in Section \ref{sec:lattice}}
We shall prove Theorems \ref{thm:variance_limit_lattice} and \ref{thm:lattice_known_beta} here. However, to do so we will need some technical preparations.

\subsubsection{\bf Technical Preparations}\label{sec:technical_prep_lattice}
We begin with the general idea behind the proof of Theorem \ref{thm:variance_limit_lattice}. The crux of the proof of Theorem \ref{thm:variance_limit_lattice} is a coupling argument between the finite volume and infinite volume measures and then exploiting the various correlation decay results of the infinite volume limit which is already known. Let us denote by $\P^{\rm{bc}}_{\beta,\bQ( \bZ^d),\mathbf{0}}$, the infinite volume limit of $\P^{\rm bc}_{\beta,\bQ(\Lambda_{n,d}),\mathbf{0}}$ for $\rm{bc} \in \{\sf f, +\}$. 
It is known that these limits exist and in fact it is also known that there exists a constant $c = c(\beta,d)$ such that for all $j \in \mathbb Z^d$, 
\begin{equation}
0\le \mathrm{Cov}^{\rm bc}_{\beta,\bQ(\mathbb Z^d),\mathbf{0}}(X_0,X_j) \le e^{-c \|j\|} \label{eq:correlation_decay}
\end{equation}
 where $\rm bc \in \{ \sf f,+\} $ if $\beta < \beta_c$ and $\rm bc = + $ if $\beta> \beta_c$ and  $\|\cdot\|$ is the $L^1$ norm in $\bZ^d$. The first inequality is known as the GKS inequality (see the lecture notes \cite{friedli2017statistical}) and the second inequality is proved in \cite{aizenman1987phase,duminil2016new} for $\beta<\beta_c$ and in \cite{duminil2018exponential} for $\beta>\beta_c$.
 From now on, we assume that 
 
 $$
 \text{if } \beta<\beta_c, \rm bc \in \{\sf f, +\} \text{ and if }\beta>\beta_c, \rm bc = +.
 $$
 Note that \eqref{eq:correlation_decay} justifies the existence of the following constant  for all $\beta \neq \beta_c$:
 
 \begin{equation}
   \chi^{\rm bc} = \chi^{\rm bc}(\beta, d) := \sum_{j \in \bZ^d} \mathrm{Cov}^{\rm bc}_{\beta,\bQ(\bZ^d),\mathbf{0}}( X_0,X_j)  \in (0,\infty)  \label{eq:chi_finite}
 \end{equation}
 In fact for $\beta<\beta_c$, since the infinite volume limits for $+$ and $\sf f$ coincide, $\chi^{\sf f}(\beta ,d) = \chi^{\sf +}(\beta,d)$. From now on, we will drop the superscript $\rm bc$ from $\chi^{\rm bc}$ and the boundary condition should be clear from context from the temperature of the model.

The constant $\chi(\beta, d)$ is known as the \emph{susceptibility} of the Ising model. We quickly mention here that $\chi(\beta_c(d), d) = \infty$ (\cite{simon1980correlation,aizenman2015random}).
We first state an easy consequence of \eqref{eq:chi_finite}.
\begin{lemma}\label{lem:varlatinfinite}
Pick $S \in \C_{s,\rm rect}$ with volume $s$. Fix $\beta\neq \beta_c$ and $\text{if } \beta<\beta_c, \rm bc \in \{\sf f, +\} \text{ and if }\beta>\beta_c, \rm bc = +.$ Then
$$
\lim_{s \to \infty} \frac1s {\mathrm{Var}^{\rm bc}_{\beta,\bQ(\mathbb Z^d),\mathbf{0}}} \left (\sum_{i\in S} X_i \right) = \chi
$$
\end{lemma}
\begin{proof}
Exponential decay of covariance implies that for any vertex $i \in S$ which is at $L^1$ distance at least $k$ from the boundary of $S$, 
$$
|\sum_{j\in S} {\mathrm{Cov}^{\rm bc}_{\beta,\bQ(\mathbb Z^d),\mathbf{0}}}(X_i, X_j) -\chi| \le e^{-ck}.
$$
for some constant $c = c(d)$.
Let $S' \subset S$ denote the set of vertices at a distance at least $m$  from the boundary of $S$ where $1 \ll m \ll s^{1/d}$ and let $s'$ denote the number of vertices in $S'$. Note that for this choice of $m$, $(s-s')/s \to 0$.
Thus summing the above estimate over all vertices $i \in S'$, we find that 
\begin{equation}
|\sum_{i \in S'} \sum_{j\in S}  {\mathrm{Cov}^{\rm bc}_{\beta,\bQ(\mathbb Z^d),\mathbf{0}}}(X_i, X_j)  - s'\chi| =  s'e^{-cm} \label{eq:varlatinfinite1}
\end{equation}
We trivially bound the remaining term, 
\begin{equation}
\sum_{i \in S \setminus S'} \sum_{j\in S} {\mathrm{Cov}^{\rm bc}_{\beta,\bQ(\mathbb Z^d),\mathbf{0}}}(X_i, X_j)  \le \chi (s-s')\label{eq:varlatinfinite2} 
\end{equation}
simply using the fact that the covariance is always nonnegative (first inequality of \eqref{eq:correlation_decay}).
Combining \eqref{eq:varlatinfinite1} and \eqref{eq:varlatinfinite2}, we get
\begin{align*}
\left |\frac1s {\mathrm{Var}^{\rm bc}_{\beta,\bQ(\mathbb Z^d),\mathbf{0}}} \left(\sum_{i\in S} X_i \right) - \chi \right| & \le \left | \frac1s \sum_{i \in  S'} \sum_{j\in S}  {\mathrm{Cov}^{\rm bc}_{\beta,\bQ(\mathbb Z^d),\mathbf{0}}}(X_i, X_j)   - \frac{s'}{s}\chi \right|\\& + \frac1s \left |\sum_{i \in S \setminus S'} \sum_{j\in S} {\mathrm{Cov}^{\rm bc}_{\beta,\bQ(\mathbb Z^d),\mathbf{0}}}(X_i, X_j)\right| + \chi \left|\frac{s'}{s}-1 \right|\\
& \le \frac{s'}{s}e^{-cm} + \frac2s (s-s')\chi 
\end{align*}
By the choice of $m$, both terms above converge to 0, thereby completing the proof.
\end{proof}

To prove Theorem \ref{thm:variance_limit_lattice}, we need a finite volume version of Lemma \ref{lem:varlatinfinite}. The key idea we use here is a coupling argument. To that end, let us define the FK-Ising model on $\bZ^d$ and the celebrated \emph{Edward-Sokal coupling} between the Ising model and the FK-Ising model.  

Let $G = (V,E)$ be a finite graph.  The FK-Ising measure is a probability measure on $\{0,1\}^{E}$ defined as follows
\begin{equation}
\phi_{G,p}(\omega) = p^{|\{e:\omega_e =1\}|} (1-p)^{|\{e:\omega_e =0\}|}2^{k(\omega)} \qquad \omega \in \{0,1\}^{E}.
\end{equation}
where $k(\omega)$ denotes the number of connected components of $\omega$ (where we think of $\omega$ as the graph $(V, \{e:\omega_e=1\})$). Similar to Ising, it is well known (see e.g. \cite{grimmett2006random}) that $\phi_{\Lambda_{n,d},p}$ converges to an infinite volume limit which is called the \emph{free} FK-Ising measure and we denote it by $\phi^{\sf f}_{\bZ^d,p}$.  We also need to consider a \emph{wired} FK-Ising which corresponds to taking a limit of $\phi_{\Lambda^{\sf w}_{n,d},p}$ where $\Lambda^{\sf w}_{n,d}$ formed by identifying all the vertices in $\bZ^d \setminus \Lambda_{n,d}$ into a single vertex and erasing all the self loops. It is also well known (cf. \cite{grimmett2006random}) that this limit exists. We denote the infinite volume wired FK-Ising  model by $\phi^{\sf w}_{\bZ^d,p}$. We mention that for $p \neq p_c = 1-e^{-2\beta_c}$,  {$\phi^{\sf w}_{\bZ^d,p}=\phi^{\sf f}_{\bZ^d,p}$} and furthermore for $p>p_c$, there exists a unique connected component almost surely in a sample from FK-Ising (free or wired).

Let $E'$ be a subset of the edges of $G$ and let $H$ be the subgraph of $G$ formed by the edges not in $E'$ and the vertices incident to them. For any element $\xi \in \{0,1\}^{E'}$, let $\phi^\xi_{H, p}$ denote the FK-Ising measure on $H$ conditioned on $\xi$. We say $\xi \preceq \xi'$ if $\xi(e )\le \xi'(e)$ for all $e \in E'$. There is another way to view boundary conditions when the edge set we are conditioning on is the complement of a  box. For $k<n$, let $\Lambda_{k,n,d} := \Lambda_{n,d} \setminus \Lambda_{k,d}$ and let $E(\Lambda_{k,n,d})$ denote the edge set induced by its vertices. Given $\xi \in \{0,1\}^{E(\Lambda_{k,n,d})}$ we can consider a partition $[\xi]$ of $\partial \Lambda_{k,d}$ with $i,j$ being in the same partition class if they belong in the same component of $\xi$. It is easy to see that the conditional law  $\phi^\xi_{\Lambda_{k,d}, p}$ is completely described by $[\xi]$, and this property is known as the \emph{domain Markov property} of the FK-Ising. Observe that the free boundary condition corresponds to the partition where every vertex is in its own class, while the wired boundary condition corresponds to the partition where every boundary vertex is in the same class.

We now state a crucial monotonicity property of FK-Ising.

\begin{lemma}[Monotonicity, \cite{duminil2017lectures}]\label{lem:monotonicity}
			Consider any subgraph $H=(V(H),E(H))$ of  a finite graph $G$ and consider $\phi_{H,p}^{\xi}$ on $H$ with boundary condition $\xi \in \{0,1\}^{E(G) \setminus E(H)}$ and parameter $p\in[0,1]$. Let $\{i\stackrel{H}{\leftrightarrow}j\}$ denotes the event that $i,j$ are connected with edges in the random subgraph of $H$ obtained from the open edges in $\phi_{H,p}^{\xi}$.  
		Fix $i,j\in V(H)$ and suppose $\xi_1 \preceq \xi_2$ be any two boundary conditions. Then
				\begin{equation}
				\phi_{H,p}^{\xi_1}(i\stackrel{H}{\leftrightarrow} j) \leq \phi_{H,p}^{\xi_2}(i\stackrel{H}{\leftrightarrow} j).
				\end{equation} 
				

\end{lemma}
A simple consequence of Lemma \ref{lem:monotonicity} is the existence of free and wired limits of FK-Ising as asserted earlier.

A crucial connection between the the Ising model and the FK-Ising model is given by the so called Edwards-Sokal coupling which we now describe. We refer to \cite[Theorem 4.91]{grimmett2006random} for a proof.  
\begin{theorem}[Edwards-Sokal coupling]\label{thm:ES}
Let $G$ be a finite graph and let $p = 1-e^{-2\beta}$. Then there exists a coupling between $\bP_{\beta,G,\mathbf{0}}$ and $\phi_{G,p}$ such that the following holds. Suppose $(\omega, \sigma) $ is sampled from the coupling where $\omega \sim \phi_{G,p} $ and $\sigma \sim \bP_{\beta,G,\mathbf{0}}$. Then $\sigma $ can be sampled by first sampling $\omega$, and then for each connected component of $\omega$ assign the same value in $\{+1,-1\}$ on all its vertices by tossing independent, unbiased coins. The same coupling holds between 
$\phi^{\sf f}_{\bZ^d,p}$ and $\P^{\sf f}_{\beta,\bQ( \bZ^d),\mathbf{0}}$.

Finally, the same coupling holds between $\phi_{\Lambda^{\sf w}_{n,d},p}$ (resp. $\phi^{\sf w}_{\bZ^d,p}$) and $\P^{+}_{\beta,\bQ(\Lambda_{n,d}),\boldsymbol 0}$ (resp. $\P^{+}_{\beta,\bQ(\bZ^d),\boldsymbol 0}$) with only one modification: the component containing the boundary vertex (resp. the a.s. unique infinite component) is assigned a value $+1$ deterministically.
\end{theorem}

We now write a lemma which establishes a strong mixing property of FK-Ising at all but critical temperature.
\begin{lemma}\label{lem:coupling}
Fix $n,d,p $ and suppose $S$ is a rectangle of volume $s$ which is at least at $L^1$ distance $m$ from the boundary of $\Lambda_{n,d}$.

Fix $p<p_c$ and ${\rm bc} \in \{\sf f, \sf w\}$.  There exists a coupling $\Phi^{\rm bc}_{n,d,p}$ between $\phi^{\rm bc}_{\Lambda_{n,d},p}$ and $\phi^{\rm bc}_{\bZ^d,p}$ such that the following holds with probability at least $1-{n^d}\exp(-cm)$. Let $(\omega^{\rm bc}_n, \omega^{\rm bc})$ be sampled from this coupling with $\omega^{\rm bc}_n \sim \phi^{\rm bc}_{\Lambda_{n,d},p}$ and $\omega^{\rm bc} \sim \phi^{\rm bc}_{\bZ^d,p}$. Then
\begin{itemize}
\item $\omega^{\rm bc}$ restricted to edges of $S$ is the same as that of $\omega^{\rm bc}_n$ restricted to $S$
\item If $i$ and $j$ in $S$ are connected in $\omega^{\rm bc}$ if and only if they are connected in $\omega^{\rm bc}_n$.
\end{itemize}
For $p>p_c$, a coupling $\Phi^{\sf w}_{n,d,p}$ holds between $\phi_{\Lambda^{\sf w}_{n,d},p}$ and $\phi^{\sf w}_{\bZ^d,p}$ where with probability at least $1-{n^d}\exp(-cm)$ in addition to the above two items, the following item holds.  Let $(\omega^{\sf w}_n, \omega^{\sf w})$ be sampled from this coupling with $\omega^{\sf w}_n \sim \phi_{\Lambda^{\sf w}_{n,d},p}$ and $\omega^{\sf w} \sim \phi^{\sf w}_{\bZ^d,p}$
\begin{itemize}
    \item $i \in S$ is in the infinite cluster of $\omega^{\sf w}$ if and only if $i$ is in the boundary cluster of $\omega^{\sf w}_n$.
\end{itemize}
\end{lemma}

\begin{proof}
\textbf{Free case.}
Let us first consider the free case. We will show that we can couple $\omega_n^{\sf f} \sim \phi^{\sf f}_{\Lambda_{n,d},p}$ and $\omega_n^{\sf w} \sim \phi^{\sf w}_{\Lambda_{n,d},p}$ with the required properties. We claim that this is enough to prove the lemma. Indeed by the domain Markov property and the monotonicity of FK-Ising, for any boundary condition described by the partitioning $[\xi]$, the conditional law of $\omega^{\xi} \sim \phi^{\xi}_{\Lambda_{n,d},p}$ on $\Lambda_{n,d}$ is sandwiched between that of the wired and the free measures. Therefore if the coupling between $\omega_n^{\sf f}, \omega_n^{\sf w}$ satisfy both items, then both pairs $(\omega^{\xi}, \omega_n^{\sf w})$ and $(\omega^{\xi}, \omega_n^{\sf f})$ satisfy both the items. 

 The coupling is done in a Markovian way which exposes the `boundary cluster'. Initially all the edges of the box is `unexplored' and the set of `explored' edges $E_0 = \emptyset$. Let $\xi^{\sf w}_{k}, \xi^{\sf f}_k \in \{0,1\}^{E_k}$ be the configurations revealed at step $k$ and assume $\xi^{\sf f}_{k} \preceq  \xi^{\sf w}_k$. In step $k+1$ we find an unexplored edge $e$ incident to either a boundary vertex or a vertex connected
 to the boundary vertex through a path in $E_k$ all of whose edges are open in $\xi^{\sf w}_k$. We stop if there is no such edge $e$. Otherwise, we reveal the status of $e$ according to the conditional law of $\omega_n^{\sf f}$ and $\omega_n^{\sf w}$ given $\xi^{\sf f}_{k} , \xi^{\sf w}_k$ respectively. Furthermore we do so in such a way that $\xi^{\sf f}_{k+1} \preceq \xi^{\sf w}_{k+1}$. Indeed this is possible because of the monotonicity (Lemma \ref{lem:monotonicity}). Suppose the exploration stops at a (random) step $T$. Notice that by description of the coupling, when we stop, all the explored edges incident to the boundary vertices of every component of unexplored edges have $\xi^{\sf w}$ value 0. By monotonicity of the coupling, all these edges have $\xi^{\sf f}$ value 0 as well. In other words, the boundary condition of the unexplored region is the same for the conditional law of $\omega_n^{\sf f}, \omega_n^{\sf w}$ given $\xi^{\sf f}_{T}$ and $\xi^{\sf w}_{T}$ respectively (in particular, it is a free boundary FK-Ising in the respective regions). We thus can and will sample both $\omega_n^{\sf f}, \omega_n^{\sf w}$ in the unexplored region to be the same according to this common conditional law.

Clearly from the description of the coupling it is enough to show that we stop before reaching any edge incident to a vertex in $S$ with probability at least $1-{n^d}\exp(-cm)$. This follows from \cite[Theorem 1.2]{duminil2017sharp} since on the complement of this event, one of the vertices between distance $m/4$ and $m/2$ from the boundary of $\partial \Lambda_n$ must be connected to the boundary of $\partial \Lambda_n$ in $\omega_n^{\sf w}$. There is a small subtlety here regarding the location of this vertex being not at the center of the box. We skip the details here see \cite[Section 5.5.1]{mukherjee2019testing} for more details on how this technicality can be handled. Overall, this event has probability at most $V_m\exp(-cm)$ by a union bound, where $V_m$ is the number of vertices in the said distance range from the boundary. Trivially upper bounding $V_m$ by $n^d$, we are done.

\textbf{Wired case.} Let us now prove the wired case. The idea is similar to the free case, except we need to do one step of renormalization using Pisztora's coarse graining approach and then crucially use the result of \cite{duminil2018exponential}. Again we argue that it is enough to show the coupling between $\omega_n^{\sf w} \sim \phi_{\Lambda^{\sf w}_{n,d},p}$ and $\omega_n^{\sf f} \sim \phi^{}_{\Lambda^{}_{n,d},p}$ satisfying the first two items. Observe that by domination, the only way a vertex $i\in S$ is  a boundary cluster for $\omega_n^{\sf w}$ but is not in an infinite cluster of $\omega^{\sf w}$ is if it is in a finite cluster of diameter at least $m$. The probability of this happening for any $i \in S$ is $O(|S|e^{-cm})$ by \cite[Theorem (5.104)]{grimmett2006random} when combined with the result of \cite{bodineau2003slab}. Thus the third item is satisfied so we concentrate on the first two.

For $x \in k\bZ^d$ define $\Lambda_{k,d}(x)$ to be the box $\Lambda_{k,d} +x$, and let $\mathcal B_k$ denote the set of boxes of this form. For $\omega \in \{0,1\}^{E(\Lambda_{n,d})}$ we say a box $\Lambda_k(x) \subset \Lambda_{n,d}$ is good if 
\begin{itemize}
\item there exists a cluster in $\omega$ restricted to $\Lambda_k(x)$ touching all its sides (this is called the \emph{crossing cluster} of the box), and,
\item every path of length at least $k/2$ intersects the above cluster.
\end{itemize}
We use the following result from \cite{pisz96,bodineau2003slab}: there exists a $c>0$ such that for all $k \ge 1$,
\begin{equation}
\sup_{\xi} \phi^{\xi}_{\Lambda_{2k,d},p}(\Lambda_{k} \text{ is good}) \ge 1-e^{-ck} \label{eq:good}.
\end{equation}
It also follows from \cite[Proposition 1.4]{duminil2018exponential} that for all $\ve>0$ we can choose $k_0$ large so that for all $k>k_0$
\begin{equation}
\sum_{e \in E(\Lambda_k)} (\phi^{\sf w}_{\Lambda_{2k,d},p}(\omega_e =1) - \phi^{\sf f}_{\Lambda_{2k,d},p}(\omega_e =0)) <\ve. \label{eq:tv}
\end{equation}

An immediate consequence of \eqref{eq:good} and \eqref{eq:tv} and Holley's criterion (see \cite[Theorems 2.1 and 2.3]{grimmett2006random}) is that for any coupling between $\omega_n^{\sf w}$ and $\omega^{\sf f}$ the set of blocks $B$ for which  $\omega_n^{\sf w}|_B =\omega_n^{\sf f}|_B $ and $B$ is good dominates a 5-dependant Bernoulli site percolation $\eta$ on the lattice $k\bZ^d$ the  with parameter $1-2\ve$ for large enough $k$. Using the main theorem of \cite{liggett1997domination}, we know that $\eta$ dominates an i.i.d.\ Bernoulli percolation with parameter $\tilde \ve (\ve)$ which tends to 0 as $\ve \to 0$. 

We now use the Markovian coupling similar to \cite[Lemma 3.3]{duminil2018exponential} or \cite[Lemma 5.2]{mukherjee2019testing}, we provide a brief sketch here. We assume $k$ is fixed but large as described in the previous paragraph and divides $n$ without loss of generality. We start with a set of explored boxes which are all the elements of $\mathcal B_k$ with centers in the boundary of $\Lambda_{n+k}$ and a set of unexplored boxes  which consist of all the elements of $\mathcal B_k$ intersecting them. In each step $t$ we pick an unexplored box $B$ and sample $\omega_n^{\sf w}|_{B \setminus C_t} , \omega_n^{\sf f}|_{B \setminus C_t} $ where $C_t$ is the set of explored boxes at step $t$. We also maintain $ \omega_n^{\sf w}|_{B \setminus C_t}  \preceq \omega_n^{\xi}|_{B \setminus C_t}$ in this sampling which is possible by monotonicity (Lemma \ref{lem:monotonicity}). If actually $ \omega_n^{\sf w}|_{B \setminus C_t}  = \omega_n^{\sf f}|_{B \setminus C_t}$ and $B$ is good, then we declare the box explored and add nothing to the set of unexplored boxes. Otherwise, we add the set of boxes in $\mathcal B_k$ intersecting $B$ to the set of unexplored boxes.

Let $T$ be the (stopping) time when the set of unexplored boxes is empty.
It can be easily checked that if at step $T$ the set of unexplored boxes is empty, the conditional law of  $\omega_n^{\sf w}|_{\Lambda_n \setminus C_T}$ is the same as that of $\omega_n^{\sf f}|_{\Lambda_n \setminus C_T}$. This follows from the definition of good boxes: if two vertices on the boundary of ${\Lambda_n \setminus C_T}$ belong to different clusters $K,K'$ in one and in the same cluster of the other, the clusters $K,K'$ must exit their respective boxes in $C_T$. However by definition of a good box, this means that the clusters must be the same as the crossing clusters in the good boxes. On the other hand, it is easy to see that the crossing clusters in adjacent boxes actually belong to the same cluster. Inducting this argument, we can conclude that $K,K'$ are actually the same cluster leading to a contradiction. We refer to \cite[Lemma 3.3]{duminil2018exponential} for details of this fact. Overall, this allows us to sample $\omega_n^{\sf w}|_{\Lambda_n \setminus C_T}=\omega_n^{\xi}|_{\Lambda_n \setminus C_T}$.
Furthermore if $C_T$ comes within distance $m/2$ of $S$, then the Bernoulli $(1-\tilde \ve)$ site percolation it dominates must also have a path of 0 of length at least $\frac{m}{2k}$ from the boundary. This has probability at most $n^{d-1}\exp(-cm)$ by a union bound and standard properties of Bernoulli percolation (see \cite{gri_perc}).

It remains to show that this coupling satisfies the three items on the event $C_T$ does not intersect $S$. Indeed the first item is clearly satisfied. The second item is also satisfied again because of the definition of good boxes. Recall that we already argued how to handle the final item, thereby finishing the proof.
\end{proof}

\begin{corollary}\label{cor:ising_coupling}
Fix $n,d,p $ and suppose $S$ is a rectangle of volume $s$ which is at least at $L^1$ distance $m$ from the boundary of $\Lambda_{n,d}$. For any $\beta < \beta_c$, there exists a coupling $\Psi$ between $\P^{\rm{f}}_{\beta,\bQ( \bZ^d),\bf 0}$ and $\P_{\beta,\bQ( \Lambda_{n,d}),\bf 0}$ such that if $(X_i,Y_j)_{i \in \Lambda_{n,d},j\in \bZ^d}$ be a sample from this coupling with $(X_i)_{i \in \Lambda_{n,d}} \sim \P_{\beta,\bQ( \Lambda_{n,d}),\bf 0}$ and $(Y_i)_{i \in \bZ^d} \sim \P_{\beta,\bQ( \bZ^d),\bf 0}$ then with probability at least $1-n^de^{-cm}$, $X_i = Y_i$ for all $i \in S$. A coupling with the same properties holds between $\P^{+}_{\beta,\bQ( \bZ^d),\bf 0}$ and $\P^{
+}_{\beta,\bQ( \Lambda_{n,d}),\bf 0}$

For $\beta > \beta_c$, a coupling with the same properties holds  between $\P^{+}_{\beta,\bQ( \bZ^d),\bf 0}$ and $\P^{
+}_{\beta,\bQ( \Lambda_{n,d}),\bf 0}$
\end{corollary}
\begin{proof}
This follows by combining Lemma \ref{lem:coupling} and Theorem \ref{thm:ES}. Let us first consider the free case. Suppose $(\omega_n, \omega)$ is sampled from the coupling described in Lemma \ref{lem:coupling} for the free case. Suppose we are on the event that the two items described in Lemma \ref{lem:coupling} are satisfied. Now for every $i \in S$ we sample the sign of the cluster containing $i$ in $\omega_n$ and $\omega$ by tossing the same unbiased coin. Note that this is well defined since by the property of the coupling, $i,j$ are connected in $\omega_n$ if and only if it is connected in $\omega$. In the complement of the event described in Lemma \ref{lem:coupling}
sample the signs independently. This completes the description of the coupling which clearly satisfies the required claim.

The proof for the wired case is exactly the same except the boundary cluster of $\omega_n$ and the infinite cluster of $\omega$ are always assigned $+1$ in the coupling.
\end{proof}

With this we are now finally ready to prove Theorem \ref{thm:variance_limit_lattice}.

\subsubsection{\bf Proof of Theorem \ref{thm:variance_limit_lattice}}
We prove the free case only as the proof of the wired case is exactly the same and for notational simplicity we also drop the superscript $\sf {f}$ from the proof.

Let $(X_i)_{i \in \Lambda_{n,d}} \sim \P^{\rm{f}}_{\beta,\bQ( \bZ^d),\bf 0}$ and $(Y_i)_{i \in \bZ^d} \sim \P_{\beta,\bQ( \Lambda_{n,d}),\bf 0}$ coupled together as in Corollary \ref{cor:ising_coupling}. We write $\mathrm{Var}, \mathrm{Cov}$ to denote the variance and covariance under this coupling measure. Let $\mathcal G$ be the event that $X_i = Y_i$ for all $i \in S$ and recall that by the coupling, $\mathcal G$ has probability at least $1-n^d\exp(-cm)$ which converges to 1 super polynomially fast by the choice of $m$. We break up the variance as follows
\begin{align}
 {\mathrm {Var}} \left (\sum_{i\in S} Y_i \right)& =  {\mathrm {Var}} \left (\sum_{i\in S} Y_i (1_{\mathcal G} + 1_{\mathcal G^c})\right)\\
&   =  \left ( {\mathrm {Var}} \left(\sum_{i \in S}Y_i 1_{\mathcal G}\right) +{\mathrm {Var}} \left(\sum_{i \in S}Y_i 1_{\mathcal G^c}\right) + 2{\mathrm {Cov}}\left (\sum_{i \in S}Y_i 1_{\mathcal G},  \sum_{i \in S}Y_i1_{\mathcal G^c}\right )\right)\label{eq:varlat3}
\end{align}
By definition of $\mathcal G$, \begin{equation}
    {\mathrm {Var}} \left(\sum_{i \in S}X_i 1_{\mathcal G}\right) = {\mathrm {Var}} \left(\sum_{i \in S}Y_i 1_{\mathcal G}\right)\label{eq:equal}.
\end{equation}
Now notice
\begin{align}
\left|{\mathrm {Var}} \left(\sum_{i \in S}Y_i \right)  - {\mathrm {Var}} \left(\sum_{i \in S}Y_i 1_{\mathcal G} \right) \right| \le {\mathrm {Var}} \left(\sum_{i \in S}Y_i 1_{\mathcal G^c}\right) + 2\left|{\mathrm {Cov}}\left (\sum_{i \in S}Y_i 1_{\mathcal G},  \sum_{i \in S}Y_i1_{\mathcal G^c}\right )\right|\label{eq:varlat1}
\end{align}
and observe, 
\begin{equation}
{\mathrm {Var}} \left(\sum_{i \in S}Y_i 1_{\mathcal G^c}\right)  \le \E\left (\left(\sum_{i \in S}Y_i\right)^2 1_{\mathcal G^c}\right) \le s^2n^de^{-cm}; \quad {\mathrm {Var}} \left(\sum_{i \in S}Y_i 1_{\mathcal G}\right)  \le \E\left (\left(\sum_{i \in S}Y_i\right)^2 1_{\mathcal G}\right) \le s^2 \label{eq:varlat2}.
\end{equation}
where we used the exponential bound on the probability of $\mathcal G^c$ and the trivial pointwise bound $|\sum_{i \in S}Y_i| \le s$. Using Cauchy--Schwarz inequality to bound the covariance term in \eqref{eq:varlat1}, we obtain
\begin{equation}
\left|{\mathrm {Var}} \left(\sum_{i \in S}Y_i \right)  - {\mathrm {Var}} \left(\sum_{i \in S}Y_i 1_{\mathcal G} \right) \right| \le s^2n^de^{-cm}  + 2s^{2}n^{d/2}e^{-cm/2} \label{eq:varlat3}
\end{equation}
Clearly, both inequalities of \eqref{eq:varlat2} is also true if we replace $Y$ by $X$. Using this and \eqref{eq:equal}, we can conclude 
\begin{equation*}
\left|{\mathrm {Var}} \left(\sum_{i \in S}X_i \right)  - {\mathrm {Var}} \left(\sum_{i \in S}X_i 1_{\mathcal G} \right) \right|\le s^2n^de^{-cm}  + 2s^{2}n^{d/2}e^{-cm/2}
\end{equation*}
Combining this with \eqref{eq:equal} and \eqref{eq:varlat3} and using the triangle inequality, we obtain
\begin{equation*}
\left |\frac1s {\mathrm {Var}} \left (\sum_{i\in S} X_i \right)  - \frac1s {\mathrm {Var}} \left (\sum_{i\in S} Y_i \right) \right| \le  2sn^de^{-cm}  + 4sn^{d/2}e^{-cm/2}.
\end{equation*}
The right hand side converges to 0 by the choice of $m$. Combining the above with Lemma \ref{lem:varlatinfinite}, the proof is complete


\medskip
\subsubsection{\bf Proof of Theorem \ref{thm:lattice_known_beta}}

We prove the upper bound first. The proof is the same for both $\rm bc = +$ or $\rm bc = \sf f$ so we drop it from the superscript and also write $\chi = \chi(\beta)$ to simplify notations. Also, throughout the proof, we will use $\bQ$ instead of $\bQ(\Lambda_{n,d})$ to simplify notations. \par
For two vectors $\mathbf{a}=(a_1, \ldots, a_d)$ and $\mathbf{b}=(b_1, \ldots, b_d)$ in $\bR^d$, we say $\mathbf{b} \succeq \mathbf{a}$ if $b_j \geq a_j$, $j \in [d]$. For two such vectors, we also denote by $\text{Rect}(\mathbf{a},\mathbf{b})=\prod_{j=1}^d[a_j,b_j]$ the rectangle with endpoints $\mathbf{a}$ and $\mathbf{b}$ and by $\text{Vol}(\text{Rect}(\mathbf{a},\mathbf{b}))=\prod_{j=1}^d(b_j-a_j)$ its volume. Let $\text{Rect}(\mathbf{a})$ denote $\text{Rect}(\mathbf{a},\mathbf{b})$ when $b_j-a_j= \lceil s^{1/d}\rceil$, $\forall j \in[d]$. Now, define a sequence $\eta_n>0$ such that,
\begin{enumerate}
	\item $\eta_n \rightarrow 0$,
	\item $s^{1/d} \eta_n \rightarrow \infty$,
	\item $|\log(\eta_n)|=o(\log(n/s))$.
\end{enumerate}
Define $t_0=t_0(\eta_n):=\{-n^{1/d}+(j-1)\eta_n s^{1/d}, 1\leq j \leq (2n/s)^{1/d}\}$. Finally, set $\mathcal{C}_s= \mathcal{C}_s(\eta_n) := \{\text{Rect}(\mathbf{a}): a_i \in t_0, i \in [d]\}$. Note that $|\mathcal{C}_s|=(1+o(1))  (2n/s)\eta^d_n$.

Recall that for a set $S$, $Z_S=\frac{1}{\sqrt{s}}\sum_{i \in S}X_i$ and $Z_{\max}=\max_{S \in \mathcal{C}_s} Z_S$. Take $\delta>0$, our test is given by $T_n(\delta)=\mathbf{1}\{Z_{\max}>t_n(\delta)\}$ where $t_n(\delta):=\sqrt{2(1+\delta)\chi\log{|\mathcal{N}(\C_s,\gamma,\varepsilon_n)|}}$. Choose $\lambda  =t_n(\delta)/\chi$. Then,
\begin{align*}
	\bP_{\beta,\bQ,\bzero}\left(Z_S \geq t_n(\delta)\right) & \leq e^{-\lambda t_n(\delta)} \bE_{\beta,\bQ,\bzero}\left(\exp(\frac{\lambda}{\sqrt{s}}\sum_{i \in S} X_i)\right)\\
	&= \exp\left(-\lambda t_n(\delta) + \frac{\lambda^2}{2s}\mathrm{Var}\Big(\sum_{i \in S} X_i\Big)+ O\Big(\frac{\lambda^3}{s}\Big)\right)\\
	&= \exp\left(-(1+o(1)+\delta)\log(n/s)\right),
\end{align*}
where the equalities are due to \cite[Equation 25]{martin1973mixing} and Theorem \ref{thm:variance_limit_lattice}.
By our assumption of $\eta_n$,
$$\bP_{\beta,\bQ,\bzero}(T_n(\delta)) \le |\mathcal C_s|\max_{S\in \mathcal C_s}\bP_{\beta,\bQ,\bzero}\left(Z_S \geq t\right)=o(1),$$
implying that the type 1 error of our test converges to $0$. 
\par 
For the type II error, consider $S^\star$ be the subset with signal $A$ with $\tanh(A)=\sqrt{\frac{2}{\chi}(1+\varepsilon)\frac{\log (n/s)}{s}}$ for some small $\varepsilon >0$. Since $s \gg (\log n)^d$ and $\eta_n \rightarrow 0$, we can obtain a set $S \in \mathcal C_s$ such that $|S \cap S^\star| = (1+o(1))s$. It is enough to show 
$$\bP_{\beta,\bQ,\bmu_{S^\star}(A)}(Z_S \leq t_n(\delta)) \rightarrow 0.$$
Since $(Z_S- \bE_{\beta,\bQ,\bmu_{S^\star}(A)}(Z_S))=O_p(1)$, it is enough to show $t_n(\delta)-\frac{1}{\sqrt{s}}\bE_{\beta,\bQ,\bmu_{S^\star}(A)}(\sum_{i \in S} X_i) \rightarrow -\infty$.

\begin{lemma}\label{lem:lattice_mean_alt}
	\begin{equation}
		\bE_{S^*}(\sum_{i \in {S^\star}}X_i)=sA \chi(1+o(1)).
	\end{equation}
\end{lemma}
\begin{proof}
	Define $f(\nu):=\log Z_n(\beta,\bQ,\bmu_{S^\star}(\nu))$. 
	\begin{align*}
		\bE_{\beta,\bQ, \bmu_{S^\star}(A)}(\sum_{i \in {S^\star}}X_i) &= \bE_{\beta,\bQ, \bmu_{S^\star}(A)}(\sum_{i \in {S^\star}}X_i)-\bE_{\beta, \bQ,\bzero}(\sum_{i \in {S^\star}}X_i)\\
		&=\frac{\partial }{\partial \eta}f(\nu)|_{\nu=A}- \frac{\partial }{\partial \eta}f(\nu)|_{\nu=0}\\
		&=Af^{''}(0)+O(A^2 f^{'''}(\tilde{\nu})).
	\end{align*}
Since $f^{''}(0)=s\chi$ and $f^{'''}(\tilde{\nu})=O(s)$ (by \cite[Lemma 5]{martin1973mixing}), we have $A^2 f^{'''}(\tilde{\nu})=o(As)$ obtaining the result.
\end{proof}

Finally recall that $\bE_{\beta,\bQ, \bmu_{S^\star}(A)}(X_i)=O(A)$. This, along with Lemma \ref{lem:lattice_mean_alt} yields $\bE_{\beta,\bQ, \bmu_{S^\star}(A)}(X_{S})= sA \chi(1+o(1))$ concluding the proof for small $\varepsilon >0$.
To prove the lower bound, we define a sub-collection $\tilde{\mathcal{C}}$ of rectangles as follows: Throughout we assume that $s^{1/d}$ and $n^{1/d}$ are integers for notational convenience, otherwise we work with corresponding ceiling functions. We will assume without loss of generality that $3s^{1/d}$ divides $n^{1/d}$. First, let $\widehat{C_n}$ be the class of disjoint sub-cubes of $\Lambda_{n,d}$ obtained by translating along
each axis (by $3s^{1/d}$ in each direction each time) the cube of side lengths $3s^{1/d}$ from the bottom left corner of $\Lambda_{n,d}$. Consequently, subdivide each cube in $\widehat{C_n}$ into $3^d$ cubes of side length $s^{1/d}$ each and take the center sub-cube of each cube in $\widehat{C_n}$
to be elements of our class $\tilde{\mathcal{C}}$. It is easy to see that $|\tilde{\mathcal{C}}|=|\widehat{C_n}|= (2/3)^d \frac{n}{s}$. Further,  $\min_{i \in S_1, j\in S_2, S_1 \neq S_2 \in \tilde{\mathcal{C}}} \|i-j\|_1 \ge 4s^{1/d}$. We choose a prior $\pi$ as uniform distribution over $\mu_S(A)$, $S \in \tilde{\mathcal{C}}$.\par 
Fix $\varepsilon>0$ small and set $\tanh(A)=\sqrt{\frac{2}{\chi}(1-\varepsilon)\frac{\log (n/s)}{s}}$. Using the prior $\pi$, we obtain the likelihood ratio:
\begin{align*}
    L_{\pi}=\frac{1}{|\tilde{\mathcal{C}}|}\sum_{S\in \tilde{\mathcal{C}}}\frac{Z(\beta,\bQ,\bmu_S(A))}{Z(\beta,\bQ,\bzero)}\exp\left(A\sum_{i\in S}X_i\right).
\end{align*}
Define $\tilde{t}_n(\delta)=\sqrt{2(1+\delta)\log|\\\tilde{\mathcal{C}}|}$. Subsequently, we let 
\begin{align}
    \tilde{L}_{\pi}(\delta):=\frac{1}{|\tilde{\mathcal{C}}|}\sum_{S\in \tilde{\mathcal{C}}}\frac{Z(\beta,\bQ,\bmu_S(A))}{Z(\beta,\bQ,\bzero)}\exp\left(A\sum_{i\in S}X_i\right)\mathbf{1}(Z_S \le \tilde{t}_n(\delta)).
\end{align}
Since $\bP_{\beta,\bQ,\bzero}(L_\pi \neq \tilde{L}_\pi(\delta))\rightarrow 0$ similar to the proof of type-I error convergence in the upper bound, it is enough to show $\tilde{L}_{\pi}(\delta) \rightarrow 1$ in probability. To this end, note that 
\begin{align*}
   \E_{\beta,\bQ,\bzero}(\tilde{L}_{\pi}(\delta))=\frac{1}{|\tilde{\mathcal{C}}|}\sum_{S\in \tilde{\mathcal{C}}}\P_{\beta,\bQ,\bmu_S(A)}\left(\frac{1}{\sqrt{s}}\sum_{i\in S}X_i\leq\tilde{t}_n(\delta)\right).
\end{align*}
Observe that, $\frac{1}{\sqrt{s}}\sum_{i\in S}(X_i-\E_{\beta,\bQ,\bmu_S(A)})$ is tight since $\mathrm{Var}_{\beta,\bQ,\bmu_S(A)}(\sum_{i\in S}X_i)\leq \mathrm{Var}_{\beta,\bQ,\bzero}(\sum_{i\in S}X_i)=O(s)$ by G.H.S inequality and Theorem \ref{thm:variance_limit_lattice}. Also, by arguments similar to the control of Type II error in the proof of upper bound, we have that for any $\delta>0$  that $\tilde{t}_n(\delta)-\frac{1}{\sqrt{s}}\E_{\beta,\bQ,\bmu_S(A)}(\sum_{i\in S}X_i)\rightarrow \infty$ uniformly in $S\in \tilde{\C}_s$. Therefore by Chebyshev's Inequality we conclude that $1-\E_{\beta,\bQ,\bzero}(\tilde{L}_{\pi}(\delta))\rightarrow 0$. \par 
To upper bound second moment $\tilde{L}_{\pi}$, we split  $\E_{\beta,\bQ,\bzero}(\tilde{L}^2_{\pi}(\delta))$ into sum of $T_1$ and $T_2$ as \eqref{eq:second_moment_split}. Note that, $T_2 \le 1+o(1)$ using \cite[Equation 11]{deb2020detecting}. Recall that, 
\begin{equation}\label{eq:tnlat}
T_1 = \frac{1}{|\tilde{\mathcal{C}}|}\frac{Z^2(\beta,\bQ,\bzero)Z(\beta,\bQ,\bmu_{S}(2A))}{Z^2(\beta,\bQ,\bmu_S(A))Z(\beta,\bQ,\bzero)}\P_{\beta,\bQ,\bmu_{S}(2A)}(Z_S \leq\tilde{t}_n(\delta)).
\end{equation}
For $k >0$,
\begin{align*}
	\frac{Z_n(\beta,\bQ,\bmu_S(kA))}{Z_n(\beta,\bQ,\bzero)}&= \bE_{\beta,\bQ,\bzero}(\exp(kA \sum_{i \in S}X_i))\\
	&= \exp(\frac{k^2A^2}{2}s \chi (1+o(1))+o(\log(n/s)))\\
	&= \exp(k^2(1-\varepsilon+o(1))\log(n/s)).
\end{align*}
Also,
\begin{align}
	\bP_{\beta,\bQ,\bmu_{S}(2A)}\left(Z_S \leq \tilde{t}_n(\delta)\right) & \leq e^{\lambda \tilde{t}_n(\delta)} \bE_{\beta,\bQ,\bmu_{S}(2A)}(\exp(-\lambda Z_S))= e^{\lambda \tilde{t}_n(\delta)} \frac{Z_n(\beta,\bQ,\bmu_S(2A-\frac{\lambda}{\sqrt{s}}))}{Z_n(\beta,\bQ,\bmu_S(2A))}. \label{eq:lattice_lbd_chernoff}
\end{align}
and hence the probability is upper bounded by $\exp(-(1+\tau)\log(n/s))$ for some small $\tau>0$. Hence,
$$T_1 \le \exp (\{-(1+o(1)) +(2-2\varepsilon) - (1+\tau)\}\log(n/s)),$$
where the first summand is obtained from count of rectangle, second summand from three ratio of two partition functions, third from \eqref{eq:lattice_lbd_chernoff}. Hence, $T_1 \rightarrow 0$ and we obtain the desired conclusion.

\subsection{\bf Proof of Theorem \ref{thm:lattice_unknown_beta}}
\textcolor{black}{Following Lemma \ref{lmm:beta_estimation}, we will use maximum pseudo-likelihood estimator for $\bzero$ magnetization and it is still consistent as long as the set of positive magnetization $S$ satisfies $|S| =o(n)$, since 
\begin{equation}
	\liminf_{n \rightarrow \infty} \frac{1}{n}\log \frac{1}{2^n}Z(\beta,\bQ,\bzero)>0,
\end{equation}
using \cite[equation (7.9)]{infising} for any $\beta>0$. When $\rm bc=\mathsf{f}$, we set $O(s^{1/d})=o(n)$ many vertices as $+1$ and still the conclusion of Lemma \ref{lmm:beta_estimation} holds. Once we obtain such an estimator $\hat{\beta}_n$ of $\beta$, we can consistently estimate the cutoff of our scan test since susceptibility of the Ising model is a continuous function of $\beta$ away from criticality.}

\subsubsection{\bf Proof of Proposition \ref{prop:monotonicity_susceptibility}}

Define
\begin{equation}
\chi_n(\beta):= \sum_{j \in \Lambda_{n,d}} \mathrm{Cov}^{}_{\beta,\bQ(\Lambda_{n,d}),\mathbf{0}}( X_0,X_j) = \sum_{j \in \Lambda_{n,d}} \mathbb{E}^{}_{\beta,\bQ(\Lambda_{n,d}),\mathbf{0}}( X_0X_j)
\end{equation}
A consequence of Lemma \ref{lem:monotonicity} and Theorem \ref{thm:ES} is that $\mathbb{E}^{}_{\beta,\bQ(\Lambda_{n,d}),\mathbf{0}}( X_0X_j)$ is non-decreasing in $\beta $ for $\beta \ge 0$.
We obtain the following by straightforward calculation:

\begin{equation*}
    \frac d{d\beta} \chi_{n}(\beta) = \sum_{j,k,l \in \Lambda_{n,d}} \big(\mathbb{E}^{}_{\beta,\bQ(\Lambda_{n,d}),\mathbf{0}}( X_0X_jX_kX_l)  -  \mathbb{E}^{}_{\beta,\bQ(\Lambda_{n,d}),\mathbf{0}}( X_0X_j)\mathbb{E}^{}_{\beta,\bQ(\Lambda_{n,d}),\mathbf{0}}( X_kX_l)\big)
\end{equation*}
To show that the quantity on the right hand side above is strictly positive for $\beta \in (0,\beta_c)$ \footnote{It is easy to come up with examples where the quantity is 0 but the system is not something trivial like i.i.d.} we invoke the following identity (see \cite[Section 3]{duminil2016random})
\begin{multline}
    \big(\mathbb{E}^{}_{\beta,\bQ(\Lambda_{n,d}),\mathbf{0}}( X_0X_jX_kX_l)  -  \mathbb{E}^{}_{\beta,\bQ(\Lambda_{n,d}),\mathbf{0}}( X_0X_j)\mathbb{E}^{}_{\beta,\bQ(\Lambda_{n,d}),\mathbf{0}}( X_kX_l)\big)\\
    = \mathbb{E}^{}_{\beta,\bQ(\Lambda_{n,d}),\mathbf{0}}( X_0X_jX_kX_l) \mathbf P_{\beta,n,d}^{\emptyset \otimes \{0,j\} \Delta \{k,l\}} (k \not \leftrightarrow l). \label{eq:current_identity}
\end{multline}
Here $\mathbf P_{\beta,n,d}^{\emptyset \otimes \{0,j\} \Delta \{k,l\}} $ denotes the \emph{double random current probability measure} where one current has no source and one current has sources $\{0,j\}\Delta \{k,l\}$. We refer to \cite{duminil2016random} for details on this model.

It remains to take a limit as $n\to \infty$. It is a consequence of the existence of the limit of correlations of the free Ising model that the infinite volume limit of $P_{\beta,n,d}^{\emptyset \otimes \{0,j\} \Delta \{k,l\}}$ exists  (\cite{duminil2016random}). Call this limit $P_{\beta,d}^{\emptyset \otimes \{0,j\} \Delta \{k,l\}}$. To justify that the limit can be inserted inside the sum,  we invoke a result of \cite{lebowitz1972bounds} and sharpness of phase transition (\cite{duminil2017sharp,aizenman1987phase}) and obtain that for all $\beta <\beta_c$, $\chi(\beta)$ is differentiable and 
\begin{equation}
   \frac{d}{d\beta} \chi(\beta) =  \lim_{n \to \infty} \frac{d}{d\beta}(\chi_n(\beta)) = \sum_{j,k,l \in \bZ^d}\mathbb{E}^{}_{\beta,\bQ(\bZ^d),\mathbf{0}}( X_0X_jX_kX_l) \mathbf P_{\beta,d}^{\emptyset \otimes \{0,j\} \Delta \{k,l\}} (k \not \leftrightarrow l)
\end{equation}

 Therefore all we need to show is that the sum in the right hand side is strictly positive. To prove this we assume familiarity with the random current model for which we refer the reader to \cite{duminil2016random}. Take $0,j,k,l$ to be the consecutive corners of a square (or a plaquette) in $\bZ^d$. We need the following property of the random current model which is a version of the finite energy property. Take any event $\mathcal A$ with $P_{\beta,d}^{\emptyset \otimes \{0,j\} \Delta \{k,l\}}(\mathcal A) >0$. Introduce a map $\psi$ which takes an element of $\mathcal A$ and adds and removes a finite number of  some odd or even edges, keeping the source set to be the same. There exists a constant $c>0$ such that $P_{\beta,d}^{\emptyset \otimes \{0,j\} \Delta \{k,l\}}(\psi(\mathcal A) ) \ge c P_{\beta,d}^{\emptyset \otimes \{0,j\} \Delta \{k,l\}}(\mathcal A )$. Now take $\mathcal A$ to be $k \leftrightarrow l$. Remove both odd or even edges so that all the edges incident to $0,j,k,l$ are closed in both the currents. This operation possibly introduces some sources in $\bZ^d \setminus \{0,j,k,l\}$. Next we add some odd edges none of which intersect $\{0,j,k,l\}$ so that there are no sources in each current (for example, one can add odd paths between the sources). Next we add an odd edge between $\{0,l\}$ and between $\{k,j\}$. This finishes the mapping, which clearly maps to a subset of $\{k \not \leftrightarrow l\}$. The proof is complete by the finite energy property mentioned above.

 \subsection{General Technical Lemmas}\label{sec:lemmas_general}
 In this section we collect the lemmas   which have been  used in the proofs of results in both Sections \ref{sec:mean_field} and \ref{sec:lattice}. 

\begin{lemma}[GHS Inequality \citep{lebowitz1974ghs}]\label{lemma:GHS}
		Suppose $X\sim \P_{\beta,\bQ,\bmu}$ with $\beta>0$, $\bQ_{ij}\geq 0$ for all $i,j\in [n]$ and 
	 $\bmu\in \left(\mathbb{R}^+\right)^{n}$. Then for any $(i_1,i_2,i_3)\in [n]^{\otimes 3}$ one has
	\begin{align*}
      \frac{\partial^3\log{Z_n(\beta,\bQ,\bmu)}}{\partial \mu_{i_1}\partial \mu_{i_2}\partial \mu_{i_3}}\leq 0.
	\end{align*}
	Consequently, for any $\bmu_1\succcurlyeq\bmu_2\succcurlyeq\mathbf{0}$ (i.e. coordinate-wise inequality) one has
\begin{align}
\mathrm{Cov}_{\beta,\bQ,\bmu_1}(X_i,X_j)&\leq \mathrm{Cov}_{\beta,\bQ,\bmu_2}(X_i,X_j),\label{eqn:correlation_ordering}
\end{align}
whenever $\beta\bQ_{ij}\geq 0$ for all $i,j\in [n]$.
\end{lemma}


\begin{lemma}[GKS Inequality \citep{friedli2017statistical}]\label{lemma:GKS}
		Suppose $X\sim \P_{\beta,\bQ,\bmu}$ with $\beta>0$, $\bQ_{ij}\geq 0$ for all $i,j\in [n]$ and $\bmu\in \left(\mathbb{R}^+\right)^{n}$. Then the following hold for any $i,j\in [n]$
		\begin{align*}
		\mathrm{Cov}_{\beta,\bQ,\bmu}(X_i,X_j)\geq 0; \quad \E_{\beta,\bQ,\bmu}(X_i)\geq 0.
		\end{align*}
\end{lemma}

\begin{lemma}[Lemma 8 of \cite{daskalakis2019testing}]\label{lemma:griffith_second}
	Suppose $X^{(k)}\sim \P_{\beta^{(k)},\bQ^{(k)},\mathbf{0}}$ for  $k=1,2$ with $\beta^{(1)}\bQ^{(1)}_{ij}\geq \beta^{(2)}\bQ^{(2)}_{ij}\geq 0$ for all $i,j$. 
	Then
	\begin{align*}
	\mathrm{Cov}_{\beta^{(1)},\bQ^{(1)},\mathbf{0}}(X_i,X_j)\geq \mathrm{Cov}_{\beta^{(2)},\bQ^{(2)},\mathbf{0}}(X_i,X_j), \quad \forall i,j.
	\end{align*}
\end{lemma}

\begin{lemma}\label{lemma:expectation_vs_externalmag}
	Suppose $X\sim \P_{\beta,\bQ,\bmu}$  with $\bmu\in (\mathbb{R}^+)^{n}$, and $\beta \bQ_{i,j}\ge 0$ for all $i\ne j$.
%
			
		\begin{enumerate}
		\item[(a)]
		Setting $c=(1-\tanh(\beta \|\bQ\|_{\infty\rightarrow\infty})$ we have $\E_{\beta, \bQ,\bmu}(X_i)\ge c\tanh(\mu_i)$.
		
		\item[(b)]
		With $C:=\frac{1}{\tanh(1)}$ we have
		$$\E_{\beta,\bQ,\bmu}(X_i)\le C\left\{\tanh(\mu_i)+\sum_{j}\tanh(\mu_j)\Big[\bQ_{ij}+\mathrm{Cov}_{\beta,\bQ,\bmu}(X_i,X_j)\Big]\right\}.$$
		
		\end{enumerate}
		\end{lemma}
		
		\begin{proof}
		\begin{enumerate}
		\item[(a)]
			Note that
	\begin{align*}
	\E_{\beta,\bQ,\bmu}(X_i)=&\E_{\beta,\bQ,\bmu}\tanh(\beta \sum_{j\in \Lambda_n(d)}\bQ_{ij}X_j+\mu_i)\\\ge &\E_{\beta,\bQ,\bmu}\tanh(\beta \sum_{j\in \Lambda_n(d)}\bQ_{ij}X_j)+(1-\tanh(\beta \|\bQ\|_{\infty\rightarrow\infty})\tanh(\mu_i),
	\end{align*}
	from which the result follows on noting that 
	$$\E_{\beta,\bQ,\bmu}\tanh(\beta \sum_{j\in \Lambda_n(d)}\bQ_{ij}X_j)\ge \E_{\beta,\bQ,{\bf 0}}\tanh(\beta \sum_{j\in \Lambda_n(d)}\bQ_{ij}X_j)=0.$$
	In the above display, the first inequality follows the fact that the Ising model is stochastically non decreasing in $\bmu$, along with the observation that the function $(x_j,j \in \Lambda_n(d))\mapsto \tanh(\beta \sum_{j\in \Lambda_n(d)}\bQ_{ij}x_j)$ is non decreasing, and the second equality follows by symmetry of the Ising model when $\bmu={\bf 0}$.
	
	\item[(b)]
	A straightforward calculus gives the existence of  $C_1,C_2<\infty$ such that for all $x\in [-c,c]$ and $y>0$ we have 
	\begin{align}\label{eq:tanh_upper}
	 C_1\tanh(y)\le \tanh(x+y)-\tanh(x)\le C_2\tanh(y).
	\end{align}
Note that if $\mu_i>1$ then we have
$$\E_{\beta,\bmu,\bQ}(X_i)\le 1 \le C \tanh(\mu_i),$$
where $C:=\frac{1}{\tanh(1)}$, and so we are done.

Setting $\mathcal{J}:=\{j\in  \Lambda_n(d):\mu_j\le  1\}$ define a vector $\tilde{\bmu}$ by setting
\begin{align*}
\tilde{\mu}_i:=&\mu_i+\sum_{j\notin \mathcal{J}}\bQ_{ij},\\
\tilde{\mu}_j:=&\mu_j\text{ if }j\ne i, j\in\mathcal{J},\\
:=&0\text{ if }j \notin \mathcal{J}.
\end{align*}
Also define an $n\times n$ matrix $\tilde{\bQ}$  by setting
\begin{align*}
\tilde{\bQ}_{j,k}:=&\bQ_{j,k}\text{ if both }j,k\in \mathcal{J},\\
:=&0\text{ if either }j,k\notin \mathcal{J}.
\end{align*}
Then invoking Lemma \ref{lemma:GKS} in the first and the last line of the display below, we have
\begin{align*}
\E_{\beta,\bQ,\bmu}(X_i)\le \lim_{\mu_j\rightarrow\infty, j\notin \mathcal{J}}\E_{\beta,\bQ,\bmu}(X_i)=&\E_{\beta,\tilde{\bQ},\tilde{\bmu}}(X_i)\\
=&\sum_{j\in \mathcal{J}}\tilde{\mu}_j\mathrm{Cov}_{\beta,\tilde{\bQ},\tilde{\bmu}}(X_i,X_j)\\
= &\sum_{j\in \mathcal{J}}\tilde{\mu}_j \lim_{\mu_j\rightarrow\infty, j\notin \mathcal{J}}\mathrm{Cov}_{\beta,\bQ,\bmu}(X_i,X_j)\\
\le& \sum_{j\in \mathcal{J}}\tilde{\mu}_j\mathrm{Cov}_{\beta,\bQ,\bmu}(X_i,X_j).
\end{align*}
Finally with $C=\frac{1}{\tanh(1)}$ the RHS above can be bounded as 
\begin{align*}
\sum_{j\in \mathcal{J}}\tilde{\mu}_j\mathrm{Cov}_{\beta,\bQ,\bmu}(X_i,X_j)=&\mu_i+\sum_{j\notin \mathcal{J}}\bQ_{ij}+\sum_{j\in \mathcal{J}, j\ne i}\mu_j\mathrm{Cov}_{\beta,\bQ,\bmu}(X_i,X_j)\\
\le &C\tanh(\mu_i)+C\sum_{j\notin \mathcal{J}}\bQ_{ij}\tanh(\mu_j)+C\sum_{j \in \mathcal{J}}\tanh(\mu_j)\mathrm{Cov}_{\beta,\bQ,\bmu}(X_i,X_j)
\end{align*}
from which the desired conclusion follows on noting that $\mathrm{Cov}_{\beta,\bQ,\bmu}(X_i,X_j)\ge 0$ for all $i,j$. 

	
	
	\end{enumerate}
	
%
		\end{proof}

\begin{lemma}\label{lmm:beta_estimation}
    If $\sup_n \|\bQ\|<\infty$ and \begin{equation}\label{eq:llog}
	\liminf_{n \rightarrow \infty} \frac{1}{n}\log \frac{1}{2^n}Z(\beta,\bQ,\bzero)>0,
\end{equation}
then maximum pseudo-likelihood estimator of $\beta$ under $\bmu=0$ based on our data is consistent for $\beta$ even under arbitrary magnetization $\bmu \in \bR^n_{+}$ when $S:=\mathrm{supp}(\bmu)$ with $|S|= o(n)$ uniformly over $\bmu$. The same conclusion holds true if we further assume there a subset $S_1$ with $|S_1|=o(n)$ and we set the vertices of $S_1$ being set to $+1$.
\end{lemma}

\begin{proof}
By assumption \eqref{eq:llog},
	$$\liminf_{n \rightarrow \infty} \frac{1}{n}\log \frac{1}{2^n}Z(\beta,\bQ,\bmu) \ge \liminf_{n \rightarrow \infty} \frac{1}{n}\log \frac{1}{2^n}Z(\beta,\bQ,\bzero) >0$$
Our maximum pseudolikelihood estimator $\hat{\beta}_n$ satisfies the equation $\tilde{S}_{\bsigma}(\beta)=0$, where $\tilde{S}_{\bsigma}(\beta)$ is defined by
\[\tilde{S}_{\bsigma}(\beta):=\frac{1}{n}\sum_i m_i({\bsigma})(\sigma_i-\tanh(\beta m_i({\bsigma})+\bmu_i)),\]
and $m_i({\bsigma}):= \sum_{j} \bQ_{ij}\sigma_j$. Define,
$$S_{\bsigma}(\beta)=\frac{1}{n}\sum_i m_i({\bsigma})(\sigma_i-\tanh(\beta m_i({\bsigma}))).$$
Following the proof techniques of \cite{infising}, we want a lower bound on $\frac{1}{n}\sum_i m^2_i({\bsigma})$ and an upper bound on $\bE_{\beta,\bQ,\bmu}(S^2_{\bsigma}(\beta))$. For the upper bound, one can redo \cite[Lemma 1.2]{chaspinglass} for $\bmu$ yielding $\bE_{\beta,\bQ,\bmu}(\tilde{S}^2_{\bsigma}(\beta))\leq c/n$, for some $c>0$.
Also 
$$|S_{\bsigma}(\beta)- \tilde{S}_{\bsigma}(\beta)| \leq \frac{1}{n}\sum_{i \in S} 2|m_i(
\bsigma)|=O\Big(\frac{s}{n}\Big).$$ Hence, 
$$\bE_{\beta,\bQ,\bmu}(S^2_{\bsigma}(\beta)) \leq 2 (\bE_{\beta,\bQ,\bmu}(\tilde{S}^2_{\bsigma}(\beta)) + o(1/n)) =O(\max\{\frac{s^2}{n^2},\frac{1}{n}\})=o(a_n),$$ 
for some $a_n=o(1)$. Therefore, $\bP_{\beta,\bQ,\bmu}(|S_{\bsigma}(\beta)|>K\sqrt{a_n})  \leq C/K^2$.\par 
To obtain a lower bound on $\frac{1}{n}\sum_i m^2_i(\bsigma)$, one needs to lower bound $H(\bsigma)=\sum_{i,j} \bQ_{ij}\sigma_i\sigma_j$. \eqref{eq:llog} and \cite[Lemma 5.1]{infising} yields
\begin{equation}\label{eq:haml}
	\lim\limits_{\varepsilon \rightarrow 0} \limsup_{n \rightarrow \infty} \sup_{\bmu} \bP_{\beta,\bQ,\bmu}(H(\bsigma)< \varepsilon n)=0.
\end{equation}
Fixing $\delta >C/K^2$. $\exists \varepsilon >0$ such that
\begin{equation}
	\bP_{\beta,\bQ,\bmu}(H(\bsigma)< \varepsilon n) \leq \delta,
\end{equation}
Consider the set $\mathcal{A}=\{|S_{\bsigma}(\beta)|\leq K\sqrt{a_n}, H(\bsigma)\geq \varepsilon n\}$. $\bP_{\beta,\bQ,\bmu}(\mathcal{A}) \geq 1-2\delta$. On $\mathcal{A}$,
\begin{align*}
	\sum \beta m^2_i({\bsigma}) & \geq \sum m_i(\bsigma) \tanh(m_i({\bsigma}))= H({\bsigma})-n S_{\bsigma}(\beta) \geq \eta n,
\end{align*} 
for some small $\eta >0$. So, $\bP_{\beta,\bQ,\bmu}(\frac{1}{n}\sum_i m^2_i({\bsigma}) \geq \eta) \geq 1-2\delta$. Also on $\mathcal{A}$,
\begin{align*}
	|S'_{\bsigma}(\beta)|&= \frac{1}{n}\sum_{i=1}^{n}m^2_i({\bsigma})\sech^2(\beta m_i({\bsigma}))\geq \eta \sech^2(\beta).
\end{align*}  
Therefore,
\begin{equation*}
	K\sqrt{a_n} \geq |S_{\bsigma}(\beta)|= |S_{\bsigma}(\beta)-S_{\bsigma}(\hat{\beta}_n)|=\int_{\beta \vee \hat{\beta}_n}^{\beta \wedge \hat{\beta}_n}|S'_{\bsigma}(x)| dx
	\geq \eta |\tanh(\beta)-\tanh(\hat{\beta}_n)|
\end{equation*}
This proves $\hat{\beta}_n$ is consistent estimator of $\beta$.\par 
If we set $S_1$ many vertices to be $+1$ with $|S_1|=o(n)$, then we have non-zero magnetization on the set $S \cup S_1$ with $|S \cup S_1|=o(n)$. Hence the previous argument still holds yielding the same conclusion.
\end{proof}

\begin{lemma}\label{lemma:expectation_underalt_lower_bound}
    Suppose $\bX\sim \P_{\beta,\bQ,\bmu}$ with $\bmu\in \Xi(\C_s, A)$. Then for any set $\tilde{S}^{\star}\subset [n]$ one has 
    \begin{align*}
        \E_{\beta,\bQ,\bmu}(\sum_{i\in \tilde{S}^{\star}}X_i)\geq A|\tilde{S}^{\star}\cap S^{\star}|-A^2\sum_{i\in \tilde{S}^{\star}\cap S^{\star}}\sum_{j\in S}\mathrm{Cov}_{\beta,\bQ,\tilde{\mathbf{\eta}}_{S^{\star}}(A)}(X_i,X_j),
    \end{align*}
    where $S^{\star}$ denotes the support of $\bmu$ and $\tilde{\mathbf{\eta}}_{S^{\star}}(A)$ is any point on the line segment joining $\bmu_{S^{\star}}(A)$ and $\mathbf{0}$.
\end{lemma}
\begin{proof}
   We begin by noting that by G.K.S. Inequality (cf. Lemma \ref{lemma:GKS}), we have
\begin{align}
    \E_{\beta,\bQ,\bmu}\left(\sum_{i\in \tilde{S}^{\star}}X_i\right)&\geq \E_{\beta,\bQ,\bmu_S^{\star}(A)}\left(\sum_{i\in \tilde{S}^{\star}\cap S^{\star}}X_i\right)\nonumber\\
    &=\sum_{i\in \tilde{S}^{\star}\cap S^{\star}}\left\{\E_{\beta,\bQ,\bzero}(X_i)+\sum_{j=1}^n \bmu_{j,S^*}(A)\mathrm{Cov}_{\beta,\bQ,\tilde{\bmu}_S^{\star}(A)}(X_i,X_j)\right\}\nonumber\\
    &=\sum_{i\in \tilde{S}^{\star}\cap S^{\star}}\sum_{j=1}^n \bmu_{j,S^*}(A)\mathrm{Cov}_{\beta,\bQ,\tilde{\bmu}_S^{\star}(A)}(X_i,X_j),\label{eqn:expectation_lb_t1}
\end{align}
for some $\tilde{\bmu}_S^{\star}(A)$ lying on the line joining $\bzero$ and $\bmu_S^{\star}(A)$. Now
\begin{align}
     \sum_{i\in \tilde{S}^{\star}\cap S^{\star}}\sum_{j=1}^n \bmu_{j,S^{\star}}(A)\mathrm{Cov}_{\beta,\bQ,\tilde{\bmu}_S^{\star}(A)}(X_i,X_j)
    &=A\sum_{i\in \tilde{S}^{\star}\cap S^{\star}}\sum_{j\in S^{\star}}\mathrm{Cov}_{\beta,\bQ,\tilde{\bmu}_S^{\star}(A)}(X_i,X_j)\nonumber\\
    &=A\sum_{i\in \tilde{S}^{\star}\cap S^{\star}}\mathrm{Var}_{\beta,\bQ,\tilde{\bmu}_S^{\star}(A)}(X_i)+\tilde{R}_n\nonumber\\
   & \geq A\sum_{i\in \tilde{S}^{\star}\cap S^{\star}}\mathrm{Var}_{\beta,\bQ,\tilde{\bmu}_S^{\star}(A)}(X_i)\label{eqn:expectation_lb_t2}
\end{align}
since $\tilde{R}_n=A\sum_{i\in \tilde{S}^{\star}\cap S^{\star}}\sum_{j\neq i\in S^{\star}}\mathrm{Cov}_{\beta,\bQ,\tilde{\bmu}_S^{\star}(A)}(X_i,X_j)\geq 0$ by by G.K.S. Inequality (cf. Lemma \ref{lemma:GKS}). 
Thereafter note that
\begin{align*}
    \mathrm{Var}_{\beta,\bQ,\tilde{\bmu}_S^{\star}(A)}(X_i)=1-\E_{\beta,\bQ,\tilde{\bmu}_S^{\star}(A)}^2(X_i).
\end{align*}
Now, by repeating the argument above we have that for some $\tilde{\boldsymbol{\eta}}_S^{\star}(A)$ on the line joining $\bzero$ and $\tilde{\bmu}_S^{\star}(A)$, and hence on on the line joining $\bzero$ and $\bmu_S^{\star}(A)$, one has
\begin{align*}
    \E^2_{\beta,\bQ,\tilde{\bmu}_S^{\star}(A)}(X_i)\leq \E_{\beta,\bQ,\tilde{\bmu}_S^{\star}(A)}(X_i)=\sum_{j=1}^n \tilde{\bmu}_{j,S}^{\star}(A)\mathrm{Cov}_{\beta,\bQ,\tilde{\boldsymbol{\eta}}_S^{\star}(A)}(X_i,X_j)\leq A\sum_{j\in S^*} \mathrm{Cov}_{\beta,\bQ,\tilde{\boldsymbol{\eta}}_S^{\star}(A)}(X_i,X_j),
\end{align*}
where the last inequality follows by noting that $\tilde{\bmu}_S^{\star}(A)$ lies on the line joining $\bzero$ and ${\bmu}_S^{\star}(A)$. Therefore
\begin{align*}
   A\sum_{i\in \tilde{S}^{\star}\cap S^{\star}}\mathrm{Var}_{\beta,\bQ,\tilde{\bmu}_S^{\star}(A)}(X_i)\geq  A|\tilde{S}^{\star}\cap S^{\star}|-A^2\sum_{i\in \tilde{S}^{\star}\cap S^*}\sum_{j\in S^*}\mathrm{Cov}_{\beta,\bQ,\tilde{\boldsymbol{\eta}}_S^{\star}(A)}(X_i,X_j). \label{eqn:expectation_lb_t3}
\end{align*} 
The proof of the lemma follows.
\end{proof}

\begin{lemma}\label{lem:hanson_wright}
Suppose $X_i$ are independent $\pm 1$ valued random variables such that $\E(X_i)=\mu_i$. Define $\tilde{X}_i=X_i-\mu_i$. For $\theta \in (-1,1) \setminus \{0\}$, define $s_\theta:= \frac{2 \theta}{\log(1+\theta)-\log(1-\theta)}$. Set $s_0=1$. Let $D_n$ be an $n \times n$ symmetric matrix with $D_{n,ii}=0$ $\forall i \in [n]$. Let $\lambda_1$ denote the largest eigenvalue of $D_n$. For any vector $\mathbf{c}= (c_1,\ldots,c_n)^\top$, if $\lambda_1 \max_{i \le n} s_{\mu_i} <1$, then
\begin{equation}
    \log\left( \bE \Big( \frac{1}{2} \sum_{i,j=1}^{n} D_{n,ij} \tilde{X}_i \tilde{X}_j + \sum_i c_i \tilde{X}_i \Big)\right) \lesssim \|D_n\|^2_F+ \sum_{i=1}^n c^2_i.
\end{equation}
As an immediate consequence, we obtain $\frac{1}{2} \sum_{i,j=1}^{n} D_{n,ij} \tilde{X}_i \tilde{X}_j + \sum_i c_i \tilde{X}_i$ is tight.
\end{lemma}
\begin{proof}
Define $s^\star:= \max_i s_{\mu_i}$. The proof follows same steps as \cite[Lemma 3.1]{debsum} with the change that we replace $\tilde{X}_i$ by $\sqrt{s^\star}Z_i$ in \cite[Equation 5.1]{debsum}.
\end{proof}

\end{document}